\documentclass{amsart}[12 pt]

\usepackage{latexsym}

\usepackage{amsmath,amssymb,dsfont,amsfonts}
\usepackage{mathrsfs}
\usepackage{verbatim}
\usepackage{graphicx}
\usepackage{graphics}
\usepackage{geometry}
\usepackage{booktabs}
\usepackage{enumitem}

\newtheorem{proposition}{Proposition}[section]
\newtheorem{theorem}{Theorem}[section]
\newtheorem{definition}{Definition}[section]
\newtheorem{corollary}{Corollary}[section]
\newtheorem{lemma}{Lemma}[section]
\newtheorem{remark}{Remark}[section]

\DeclareMathOperator*{\esssup}{ess\,sup} 
\DeclareMathOperator*{\argmax}{arg\,max} 

\numberwithin{equation}{section}

\allowdisplaybreaks

\begin{document}
\markboth{}{}

\title[Stochastic Near-Optimal Controls for Path-Dependent Systems] {Stochastic Near-Optimal Controls for Path-Dependent Systems}

\author{Dorival Le\~ao}

\address{Departamento de Matem\'atica Aplicada e Estat\'istica. Universidade de S\~ao
Paulo, 13560-970, S\~ao Carlos - SP, Brazil} \email{leao@estatcamp.com.br}

\author{Alberto Ohashi}

\address{Departamento de Matem\'atica, Universidade Federal da Para\'iba, 13560-970, Jo\~ao Pessoa - Para\'iba, Brazil}\email{amfohashi@gmail.com}

\author{Francys Souza}

\address{Departamento de Matem\'atica Aplicada e Estat\'istica. Universidade de S\~ao
Paulo, 13560-970, S\~ao Carlos - SP, Brazil}\email{francysouz@gmail.com}

\thanks{}
\date{\today}


\begin{center}
\end{center}

\begin{abstract}
In this article, we present a general methodology for control problems driven by the Brownian motion filtration including non-Markovian and non-semimartingale state processes controlled by mutually singular measures. The main result of this paper
is the development of a concrete pathwise method for characterizing and computing near-optimal controls for abstract controlled Wiener functionals. The theory does not require ad hoc functional differentiability assumptions on the value process and elipticity conditions on the diffusion components. The analysis is pathwise over suitable finite dimensional spaces and it is based on the weak differential structure introduced by Le\~ao, Ohashi and Simas \cite{LEAO_OHASHI2017.1} jointly with measurable selection arguments. The theory is applied to stochastic control problems based on path-dependent SDEs where both drift and possibly degenerated diffusion components are controlled. Optimal control of drifts for path-dependent SDEs driven by fractional Brownian motion is also discussed. We finally provide an application in the context of financial mathematics. Namely, we construct near-optimal controls in a non-Markovian portfolio optimization problem.

\end{abstract}
\maketitle
\section{Introduction}
Let $\mathbf{C}^n_T$ be the set of continuous functions from $[0,T]$ to $\mathbb{R}^n$, let $\xi:\mathbf{C}^n_T\rightarrow\mathbb{R}$ be a Borel functional, let $\mathbb{F} = (\mathcal{F}_t)_{t\ge 0}$ be a fixed filtration and let $U^T_t; 0\le t \le T$ be a suitable family of admissible $\mathbb{F}$-adapted controls defined over $(t,T]$. The goal of this paper is to develop a systematic approach to solve a generic stochastic optimal control problem of the form

\begin{equation}\label{INTROpr1}
\sup_{\phi\in U^T_0}\mathbb{E}\Big[\xi\big(X^\phi\big)\Big],
\end{equation}
where $\{X^\phi; \phi\in U^T_0\}$ is a given family of $\mathbb{F}$-adapted controlled continuous processes. A common approach to such generic control problem (see e.g \cite{Davis_79, elkaroui, elliot}) is to consider for each control $u\in U_0^T$, the value process given by

\begin{equation}\label{valueINTR}
V(t,u) = \esssup_{\phi;\phi=u~\text{on}~[0,t]}\mathbb{E}\Big[\xi\big(X^\phi\big)|\mathcal{F}_t\Big]; 0\le t \le T.
\end{equation}

Two fundamental questions in stochastic control theory rely on sound characterizations of value processes and the development of concrete methods to produce optimal controls $u^*\in U^T_0$ (when exists)

\begin{equation}\label{INTROpr2}
\mathbb{E}\big[\xi(X^{u^*})\big] = \sup_{\phi\in U^T_0}\mathbb{E}\Big[\xi\big(X^\phi\big)\Big].
\end{equation}
Optimal controls realizing (\ref{INTROpr2}) are called exacts. Besides the fact exact optimal controls may fail to exist due to e.g lack of convexity, they are very sensitive to perturbations and numerical rounding. An alternative to the exact optimal control is the so-called near-optimal controls (see e.g \cite{zhou}) which realize
\begin{equation}\label{INTROpr3}
\mathbb{E}\big[\xi(X^{u^*})\big] > \sup_{\phi\in U^T_0}\mathbb{E}\Big[\xi\big(X^\phi\big)\Big]-\epsilon,
\end{equation}
for an arbitrary $\epsilon>0$. The original problem (\ref{INTROpr1}) (dynamically described by (\ref{valueINTR})) can be greatly simplified in analysis and implementation by considering near-optimal controls which exist under minimal hypotheses and are sufficient in most practical cases.


In the Markovian case, a classical approach in solving stochastic control problems is given by the dynamic programming principle based on Hamilton-Jacobi-Bellman (HJB) equations. One popular approach is to employ verification arguments to check if a given solution of the HJB equation coincides with the value function at hand, and obtain as a byproduct the optimal control. Discretization methods also play an important role towards the resolution of the control problem. In this direction, several techniques based on Markov chain discretization schemes \cite{kushner2}, Krylov's  regularization and shaking coefficient techniques (see e.g \cite{krylov1,krylov2}) and Barles-Souganidis-type monotone schemes \cite{barles} have been successfully implemented. We also refer the more recent probabilistic techniques on fully non-linear PDEs given by Fahim, Touzi and Warin \cite{fahim} and the randomization approach of Kharroubi, Langren\'e and Pham \cite{pham, pham1,pham2}.

Beyond the Markovian context, the value process (\ref{valueINTR}) can not be reduced to a deterministic PDE and the control problem (\ref{INTROpr1}) is much more delicate. Nutz \cite{nutz2} employs techniques from quasi-sure analysis to characterize one version of the value process as the solution of a second order backward SDE (2BSDE) (see \cite{soner}) under a non-degeneracy condition on the diffusion component of a path-dependent controlled SDE $X^\phi$.  Nutz and Van Handel \cite{nutz1} derive a dynamic programming principle in the context of model uncertainty and nonlinear expectations. Inspired by the work \cite{pham}, under the weak formulation of the control problem, Fuhrman and Pham \cite{fuhrman} shows a value process can be reformulated under a family of dominated measures on an enlarged filtered probability space where the controlled SDE might be degenerated. It is worth to mention that under a nondegeneracy condition on diffusion components of controlled SDEs, (\ref{valueINTR}) can also be viewed as a fully nonlinear path-dependent PDE (PPDE) in the sense of \cite{touzi2} via its relation with 2BSDEs (see section 4.3 in \cite{touzi2}). In this direction, Possama\"i, Tan and Zhou \cite{possamai} derived a dynamic programming principle for a stochastic control problem with respect to a class of nonlinear kernels. Based on this dynamic programming principle, they obtained a well-posedness result for general 2BSDEs and established a link with PPDE in possibly degenerated cases.

Discrete-type schemes which lead to approximation of the optimal value (\ref{INTROpr1}) for controlled non-Markovian SDEs driven by Brownian motion was studied by Zhang and Zhuo \cite{zhang2}, Ren and Tan \cite{ren} and Tan \cite{tan}. In \cite{zhang2,ren}, the authors provide monotone schemes in the spirit of Barles-Souganidis for fully nonlinear PPDEs in the sense of \cite{touzi2} and hence one may apply their results for the study of (\ref{INTROpr1}). Under elipticity conditions, by employing weak convergence methods in the spirit of Kushner and Depuis, \cite{tan} provides a feasible discretization method for the optimal value (\ref{INTROpr1}).

\subsection{Main setup and contributions}
The main goal of this paper is to deepen the analysis of non-Markovian stochastic control problems. Rather than developing new representation results, we aim to provide a systematic pathwise approach to extract near-optimal controls based on a given family of non-anticipative state functionals $\{X^u; u\in U^T_0\}$ adapted to the Brownian motion filtration and parameterized by possibly mutually singular measures. The theory developed in this article applies to virtually any control problem of the form (\ref{INTROpr1}) (see also Remark \ref{integralCOST}) under rather weak integrability conditions where none elipticity condition is required from the controlled state. For instance, controlled path-dependent degenerated SDEs driven by possibly non-smooth transformations of the Brownian motion (such as fractional Brownian motion) is a typical non-trivial application of the theory.

Our methodology is based on a weak version of functional It\^o calculus developed by Le\~ao, Ohashi and Simas \cite{LEAO_OHASHI2017.1}. A given Brownian motion structure is discretized which gives rise to differential operators acting on piecewise constant processes adapted to a jumping filtration in the sense of \cite{jacod} and generated by what we call a discrete-type skeleton $\mathscr{D} = \{\mathcal{T},A^{k,j}; j=1, \ldots, d,; k\ge 1\}$ (see Definition \ref{discreteskeleton}). For a given controlled state process $\{X^\phi; \phi\in U_0^T\}$, we construct a \textit{controlled imbedded discrete structure} $\big((V^k)_{k\ge 0},\mathscr{D}\big)$ (for precise definitions, see Sections \ref{CIDSsection} and \ref{CIDSVALUEsection}) for the value process (\ref{valueINTR}). This is a non-linear version of the imbedded discrete structures introduced by \cite{LEAO_OHASHI2017.1} and it can be interpreted as a discrete version of (\ref{valueINTR}). In Proposition \ref{AGREGATION}, by using measurable selection arguments, we aggregate the controlled imbedded discrete structure $\big((V^k)_{k\ge 0},\mathscr{D}\big)$ into a single finite sequence of upper semianalytic value functions $\mathbb{V}^k_n:\mathbb{H}^{k,n}\rightarrow\mathbb{R}; n=0,\ldots, e(k,T)-1$. Here, $\mathbb{H}^{k,n}$ is the $n$-fold cartesian product of $\mathbb{A}\times \mathbb{S}_k$, where $\mathbb{A}$ is the action space, $\mathbb{S}_k$ is suitable finite-dimensional space which accommodates the dynamics of the structure $\mathscr{D}$ and $e(k,T)$ is a suitable number of periods to recover (\ref{INTROpr1}) over the entire period $[0,T]$ as the discretization level $k$ goes to infinity. In Corollary \ref{PDECHARACTERIZATION}, we then show this procedure allows us to derive a pathwise dynamic programming equation. More importantly, we provide a rather general pathwise method to select candidates to near-optimal controls for (\ref{INTROpr3}) by means of a feasible maximization procedure based on integral functionals

\begin{equation}\label{INTROpr4}
\argmax_{a^k_n\in \mathbb{A}} \int_{\mathbb{S}_k}\mathbb{V}^k_{n+1}(\mathbf{o}^k_n, a^k_n, s^k_{n+1},\tilde{i}^k_{n+1})\nu^k_{n+1}(ds^k_{n+1}d\tilde{i}^k_{n+1}|\mathbf{b}^k_n); \quad n=e(k,T)-1, \ldots, 0,
\end{equation}
subject to a terminal condition $\mathbb{V}^k_{e(k,T)}(\mathbf{o}^k_{e(k,T)})$, where $\nu^k_{n+1}$ is the transition probability kernel of $\mathscr{D}$ acting on $\mathbb{S}_k$ (see Proposition \ref{disinteresult}), $\mathbf{o}^k_n \in \mathbb{H}^{k,n}$ is the history of the imbedded discrete system and $\mathbf{b}^k_n \in \mathbb{S}_k^n$ is the noise information at the step $n$.


If the controlled state and the associated value process (\ref{valueINTR}) are continuous controlled Wiener functionals (see Definition \ref{continuouscontrolled}), then Theorem \ref{VALUEconv} shows that (\ref{valueINTR}) admits a rather weak continuity property w.r.t the controlled imbedded discrete structure $\big((V^k)_{k\ge 0},\mathscr{D}\big)$. More importantly, Theorem \ref{VALUEconvcontrol} reveals that near-optimal controls associated with the controlled structure $\big((V^k)_{k\ge 0},\mathscr{D}\big)$ and computed via (\ref{INTROpr4}) are near-optimal in the sense of (\ref{INTROpr3}). As a by-product, we are able to provide a purely pathwise description of near-optimal controls for a generic optimal control problem (\ref{INTROpr1}) based on a given $\{X^\phi; \phi\in U^T_0\}$. This gives in particular an original method to solve stochastic control problems for abstract controlled Wiener functionals, without requiring ad hoc assumptions on the value process in the sense of functional It\^o calculus \cite{cont2} and elipticity conditions on the system. The regularity conditions of the theory boils down to mild integrability hypotheses, path continuity on the controlled process jointly with its associated value process and a H\"{o}lder modulus of continuity on the payoff functional $\xi:\mathbf{C}^n_T\rightarrow\mathbb{R}$.

We remark that our approach does not rely on a given representation of the value process (\ref{valueINTR}) in terms of PPDE or 2BSDE, but rather on its inherent $U^T_0$-supermartingale property (for precise definition, see Remark \ref{remark_consist}). In particular, it is required the existence of versions of (\ref{valueINTR}) with continuous paths (see Lemma \ref{continuousV}) for each control and none pathwise or quasi-sure representation of (\ref{valueINTR}) is needed in our framework. Rather than exploring 2BSDEs or PPDEs, we develop a fully pathwise structure $\mathbb{V}^k_j; j=0,\ldots, e(k,T)-1$ which allows us to make use the classical theory of analytic sets to construct path wisely the near-optimal controls for (\ref{INTROpr3}) by means of a list of analytically measurable functions $C_{k,j}:\mathbb{H}^{k,j}\rightarrow\mathbb{A}; j=0, \ldots, e(k,T)-1$. By composing those functions with the skeleton $\mathscr{D}$, we are able to construct pure jump $\mathscr{D}$-predictable near optimal controls

$$\phi^{\star,k} = \big(\phi^k_0, \ldots, \phi^{k}_{e(k,T)-1}\big)$$
realizing (\ref{INTROpr3}) for $k$ sufficiently large, where the near-optimal control at the $j$-th step depends on previous near-optimal controls $n=0, \ldots,j-1$ by concatenating $\phi^{k}_0\otimes\ldots\otimes \phi^k_{j-1}$ for $j=1, \ldots, e(k,T)$. This allows us to treat very concretely the intrinsic path-dependence of the stochastic control problem under rather weak regularity conditions. We also emphasize that there is no conceptual obstruction in our approach in getting explicit rates of convergence. Indeed, it will depend on more refined estimates associated with the convergence of the filtrations and the derivative operator given by \cite{LEAO_OHASHI2017.1}. We postpone this analysis to a further investigation.


As a test of the relevance of the theory, we then show that it can be applied to controlled SDEs with rather distinct types of path-dependence:

$$
\textbf{Case (A)}\quad dX^u(t) = \alpha(t,X^u,u(t))dt + \sigma(t,X^u,u(t))dB(t),
$$

$$\textbf{Case (B)}\quad dX^u(t) = \alpha(t,X^u,u(t))dt + \sigma dB_H(t),
$$
where $B_H$ is the fractional Brownian motion (FBM) with exponent $0 < H < 1$ and $B$ is the Brownian motion. In case (A), the lack of Markov property is due to the coefficients $\alpha$ and $\sigma$ which may depend on the whole path of $X^u$. In this case, the controlled state $X^u$ satisfies a pseudo-Markov property in the sense of \cite{claisse}. The theory developed in this article applies to case (A) without requiring elipticity conditions on the diffusion component $\sigma$. Case (B) illustrates a fully non-Markovian case: The controlled state $X^u$ is driven by a path-dependent drift and by a very singular transformation of the Brownian motion into a non-Markovian and non-semimartingale noise. In particular, there is no probability measure on the path space such that the controlled state in (B) is a semimartingale.

To the best of our knowledge, despite the recent efforts on representation theorems for value processes (\cite{nutz1,nutz2,touzi2,fuhrman}) driven by
path-dependent SDEs in (A) and numerical schemes for PPDEs (\cite{tan,zhang2,ren}), obtaining optimal controls (either exact or near) is novel. In particular, we do not assume a priori regularity assumptions on the value process in the sense of functional It\^o calculus (see section 8.3 in Cont \cite{cont2}) and none nondegeneracy condition on the controlled system is required. As far as (B), the current literature on the control theory for FBM driving force (see \cite{biagini, han, buckdahn_shuai}) relies on the characterization of optimal controls via Pontryagin-type maximum principles based on BSDEs (with implicit or explicit FBM) at the expense of Malliavin differentiability of controls with exception of \cite{buckdahn_shuai}. We mention that for non path-dependent quadratic costs, exact optimal controls for linear state controlled processes driven by FBM are obtained by Hu and Zhou \cite{Hu1} via solutions of BSDEs driven by FBM and Brownian motion. We stress the theory of this article provides a systematic way to extract near-optimal controls for control problems of the form (\ref{INTROpr2}) with possibly path-dependent payoff functionals composed with non-linear controlled SDEs driven by FBM and, more generally, singular transformations of Brownian motions. In order to illustrate the use of theory, we present a concrete example in financial mathematics. Namely, we construct near-optimal controls in a non-Markovian portfolio optimization problem (see section \ref{portfolioOPT}).

The remainder of this article is organized as follows. The next section summarizes some useful properties of the value process (\ref{valueINTR}). Section \ref{CIDSsection} presents the concept of controlled imbedded discrete structure which is a fundamental object in our methodology. Section \ref{CIDSVALUEsection} presents the pathwise dynamic programming equation and the obtention of near-optimal controls for a given approximation level. Section \ref{CONVERGENCEsection} presents the abstract convergence results. Section \ref{APPLICATIONsection} presents applications to cases (A-B) and section \ref{portfolioOPT} presents an application to a non-Markovian portfolio optimization problem.

\

\textbf{Notation.} The paper is quite heavy with notation, so here is a partial list for ease of reference:
\noindent $U^N_M$ ($N,M$ stopping times), $U^{k,n}_m$: Set of admissible controls; equations (\ref{controlset}) and (\ref{uknm}).

\noindent $U^m_\ell$ ($m,\ell$ positive integers): Equation (\ref{abbreviatedU}).

\noindent $\mathcal{A}^k_n$: Equation (\ref{Acaligrafico}).

\noindent $u\otimes_N v$ ($N$ stopping time), $u^k\otimes_n v^k$ ($n$ positive integer): Concatenations; equations (\ref{concatenation}) and (\ref{Kconcatenation}).

\noindent $\xi_X(u), \xi_{X^k}(u^k):$ The payoff functional $\xi$ applied to controlled processes $X$ and $X^k$, respectively; equations (\ref{actionmap}) and (\ref{Kactionmap}).

\noindent $\nu^k_{n+1}$: The transition probability of the discrete-type skeleton $\mathscr{D}$; equation (\ref{form1dis}).

\noindent $e(k,T)$: Equation (\ref{ektdef}).

\noindent $\Xi^{k,g^k}_j$: Equation (\ref{Xioperator}).

\noindent $\mathbf{b}^k_n,\mathbf{o}^k_n$: Equations (\ref{elementsSkn}) and (\ref{elementsSknA}).

\noindent $V(t,u), V^k(T^k_n,u^k)$: Equations (\ref{valuepdef}) and (\ref{discretevalueprocess}).

\noindent $\mathbb{V}^k_j$: Equation (\ref{Vkfunction}).

\section{Controlled stochastic processes}

Throughout this article, we are going to fix a filtered probability space $(\Omega, \mathbb{F},\mathbb{P})$ equipped with a $d$-dimensional Brownian motion $B = \{B^{1},\ldots,B^{d}\}$ where $\mathbb{F}:=(\mathcal{F}_t)_{t\ge 0}$ is the
usual $\mathbb{P}$-augmentation of the filtration generated by $B$ under a fixed probability measure $\mathbb{P}$. For a pair of finite $\mathbb{F}$-stopping times $(M,N)$, we denote

\begin{equation}\label{stocintDEF}
]]M,N ]]: = \{(\omega,t); M(\omega) < t \le N(\omega)\},
\end{equation}
and $]] M, +\infty[[:=\{(\omega,t); M(\omega)< t < +\infty\}.$ The action space is a compact set

$$\mathbb{A}:=\{(x_1, \ldots, x_m)\in \mathbb{R}^m; \max_{1\le i\le m}|x_i|\le \bar{a}\}$$
for some $0 < \bar{a}< +\infty$.
In order to set up the basic structure of our control problem, we first need to define the class of admissible control processes: For each pair $(M,N)$ of a.s finite $\mathbb{F}$-stopping times such that $M < N~a.s$, we denote~\footnote{Whenever necessary, we can always extend a given $u\in U^N_M$ by setting $u = 0$ on the complement of a stochastic set $]]M,N]]$.}

\begin{equation}\label{controlset}
U^{N}_{M}:=\{\text{the set of all}~\mathbb{F}-\text{predictable processes}~u:~]]M,N]]\rightarrow \mathbb{A}; u(M+)~\text{exists}\}.
\end{equation}
For such family of processes, we observe they satisfy the following properties:

\

\begin{itemize}\label{3-proper}
  \item \textbf{Restriction}: $u\in U_M^N\Rightarrow u\mid_{]]M,P]]}\in U_{M}^P$ for $M < P \le N$~a.s.
  \item \textbf{Concatenation}: If $u\in U_{M}^N$ and $v\in U_N^P$ for $M<N< P~a.s$, then $(u\otimes_N v)(\cdot)\in U_M^P$, where
  \begin{equation}\label{concatenation}
  (u\otimes_N v)(r):=\left\{
\begin{array}{rl}
u(r); & \hbox{if} \ M < r \le N \\
v(r);& \hbox{if} \ N < r \le P.
\end{array}
\right.
\end{equation}
  \item \textbf{Finite Mixing}: For every $u,v\in U_M^N$ and $G\in \mathcal{F}_M$, we have
$$u\mathds{1}_G + v\mathds{1}_{G^c}\in U^N_M.$$

\item \textbf{Countable mixing on deterministic times}: Given a sequence of controls $u_1, u_2, \ldots$ in $U_{s}^t$ for $s< t$ and a sequence of disjoint sets $D_1, D_2, \ldots$ in $\mathcal{F}_s$, we have

\[
\sum_{i=1}^\infty u_i 1\!\!1_{ D_i}\in U^t_s.
\]
\end{itemize}
To keep notation simple, we denote $U_M: = \{u:~]]M,+\infty[[\rightarrow \mathbb{A}~\text{is}~\mathbb{F}-\text{predictable and}~u(M+)~\text{exists}\}$ for each finite $\mathbb{F}$-stopping time $M$. Let $\textbf{B}^p(\mathbb{F})$ be the Banach space of all $\mathbb{F}$-adapted real-valued c\`adl\`ag processes $Y$ such that

\begin{equation}\label{bpspace}
\mathbb{E}\sup_{0\le t\le T}|Y(t)|^p < \infty,
\end{equation}
for $1\le p< \infty$ and $ 0 < T < +\infty$ is a fixed terminal time.


\begin{definition}\label{continuouscontrolled}
A continuous \textbf{controlled Wiener functional} is a map $X:U_0\rightarrow\mathbf{B}^p(\mathbb{F})$ for some $p\ge 1$, such that for each $t\ge 0$ and $u\in U_0$, $\{X(s,u); 0\le s\le t\}$ depends on the control $u$ only on $(0,t]$ and $X(\cdot,u)$ has continuous paths for each $u\in U_0$.
\end{definition}


From now on, we are going to fix a controlled Wiener functional $X:U_0\rightarrow\mathbf{B}^2(\mathbb{F})$. In the sequel, $\mathbf{D}^n_T:= \{h:[0,T]\rightarrow \mathbb{R}^n~\text{with c\`adl\`ag paths}\}$ and we equip this linear space with the uniform convergence on $[0,T]$. Throughout this paper, we assume the following regularity properties on the payoff functional:


\

\noindent $\textbf{\text{(A1)}}$: The payoff $\xi:\mathbf{D}^n_T\rightarrow\mathbb{R}$ is bounded and there exists $\gamma \in (0,1]$ and a constant $C>0$ such that

\begin{equation}\label{descA1}
|\xi(f) - \xi(g)|\le C (\sup_{0\le t\le T}\|f(t)-g(t)\|_{\mathbb{R}^n})^\gamma
\end{equation}
for every $f,g\in \mathbf{D}^n_T$.


\begin{remark}
Even though we are only interested in controlled Wiener functionals with continuous paths, we are forced to assume the payoff functional is defined on the space of c\`adl\`ag paths due to a discretization procedure. However, this is not a strong assumption since most of the functionals of interest admits extensions from $\mathbf{C}^n_T$ to $\mathbf{D}^n_T$ preserving property \textbf{(A1)}. The boundedness assumption is not essential but for simplicity of exposition we keep this assumption throughout this work.
\end{remark}

The action of the payoff on a given controlled Wiener functional will be denoted by

\begin{equation}\label{actionmap}
\xi_X(u): = \xi\big(X(\cdot,u)\big); u\in U_0.
\end{equation}

For a given controlled Wiener functional $u\mapsto X(\cdot,u)$, we define
\begin{equation}\label{valuepdef}
V(t,u):=\text{ess}~\sup_{v\in U^T_t}\mathbb{E}\Big[\xi_X(u\otimes_t v)|\mathcal{F}_t\Big];0\le t< T,u \in U_0,
\end{equation}
where $V(T,u) := \xi_X(u)$~a.s and the process $V(\cdot, u)$ has to be viewed backwards. Throughout this paper, in order to keep notation simple, we omit the dependence of the value process in (\ref{valuepdef}) on the controlled Wiener functional $X$ and we write $V$ meaning as a map $V:U_0\rightarrow\mathbf{B}^1(\mathbb{F})$.

Since we are not assuming that $\mathbb{F}$ is the raw filtration generated by the Brownian motion, we can not say that $V(0)$ is deterministic. However, the finite-mixing property on the class of admissible controls implies that $\{\mathbb{E}\big[\xi_X(u\otimes_t \theta)|\mathcal{F}_t\big]; \theta\in U_t^T\}$ has the lattice property (see e.g Def 1.1.2 \cite{lamberton}) for every $t\in [0,T)$ and $u\in U^T_t$. In this case,

\begin{equation}\label{lattice0}
\mathbb{E}\big[V(0)\big] = \sup_{v\in U_0}\mathbb{E}\big[\xi(X(\cdot,v)\big].
\end{equation}
More generally,
\begin{equation}\label{lattice1}
\mathbb{E} \Big[ \esssup_{\theta \in U_t} \mathbb{E}\big[\xi_X (u \otimes_t \theta)|\mathcal{F}_t\big] \big|\mathcal{F}_s\Big] = \esssup_{\theta \in U_t}
\mathbb{E} \big[ \xi_X(u \otimes_t \theta)|\mathcal{F}_s\big]~a.s
\end{equation}
for each $u \in U_0$ and $0\le s\le t\le T$.


\begin{remark} \label{remark_consist}
For any $u\in U_0$, $\{V(s, u); 0\le s\le t\}$ depends only on the control $u$ restricted to the interval $[0,t]$. Hence, $u\mapsto V(\cdot,u)$ is a controlled Wiener functional. Moreover, $V$ is an $U_0$-supermartingale, in the sense that $V(\cdot, u)$ is an $\mathbb{F}$-supermartingale for each $u\in U_0$.
\end{remark}

It is natural to ask when $V(\cdot, u)$ admits adapted modifications with c\`adl\`ag paths for each $u\in U_0$. In order to investigate such property, we shall consider the following assumption:

\

\noindent \textbf{(B1)} There exists a constant $C$ such that

\begin{equation}\label{LipL2}
\|X(\cdot,u)-X(\cdot,\eta)\|^2_{\mathbf{B}^2(\mathbb{F})}\le C \mathbb{E}\int_0^T\|u(s)-\eta(s)\|^2_{\mathbb{R}^m}ds
\end{equation}
for every $u,\eta\in U_0$.

\begin{lemma}\label{continuousV}
If $\xi$ is a bounded pointwise continuous functional and $X$ satisfies \textbf{(B1)}, then for each $u\in U_0$, the supermartingale $V(\cdot, u)$ admits an adapted modification with c\`adl\`ag paths. If \textbf{(A1-B1)} hold true, then for each $u\in U_0$, the supermartingale $V(\cdot, u)$ admits an adapted modification with continuous paths.
\end{lemma}
\begin{proof}
Let us fix $u\in U_0$. In order to prove that $V(\cdot,u)$ admits a c\`adl\`ag (continuous) modification, from the $\mathbb{F}$-supermartingale property and the fact that the augmented Brownian filtration is continuous, it is sufficient to prove that $t\mapsto \mathbb{E}[V(t,u)]$ is right-continuous (continuous) (see Th 2 - page 67 in \cite{dellacherie}) and this verification is a routine exercise by using \textbf{(A1-B1)}, so we omit the details.
\end{proof}

\begin{definition}\label{epsilonDEF}
We say that $u\in U^T_0$ is an $\epsilon$-optimal control if
\begin{equation}\label{optimaldef1}
\mathbb{E}\big[\xi_X(u)\big]\ge \sup_{\eta\in U^T_0}\mathbb{E}\big[\xi_X(\eta)\big]-\epsilon.
\end{equation}
In case, $\epsilon=0$, we say that $u$ realizing (\ref{optimaldef1}) is an optimal control.
\end{definition}

\begin{remark}\label{oprem1}
A classical result (see e.g \cite{Davis_79,striebel}) states that $u^*$ is optimal if, and only if, $V(\cdot, u^*)$ is an $\mathbb{F}$-martingale.
\end{remark}

In the sequel, we introduce the concept of conditional optimality similar to  El Karoui \cite{elkaroui}.

\begin{definition} For a given $\epsilon \geq 0$ and $\pi \in U_{0}^T$, a control $h^\epsilon \in U_{t}^T$ is $(t, \epsilon , \pi)$-optimal if

$$
V(t, \pi) \leq \mathbb{E} \left[\xi_X((\pi \otimes_t h^\epsilon)) \mid \mathcal{F}_t \right] + \epsilon~a.s.
$$
\end{definition}
Of course, an $(0, \epsilon , \pi)$-optimal control is also $\epsilon$-optimal.

\begin{lemma}\label{epsilonrandomop}
For every $t\in[0,T]$, $\pi \in U_{0}^T$ and $\epsilon > 0$, there exist $(t, \epsilon , \pi)$-optimal controls.
\end{lemma}

\begin{proof}
It is well known that there exists a countable subset $J_{t}^T=(u_1, u_2, \cdots)$ of $U_{t}^T$ such that

\[
V(t, \pi) = \text{ess}~\sup_{\theta\in U^T_t} \mathbb{E} \left[\xi_X((\pi \otimes_t \theta)) \mid \mathcal{F}_t \right] = \sup_{i\ge 1} \mathbb{E} \left[\xi_X((\pi \otimes_t u_i)) \mid \mathcal{F}_t \right]~a.s.
\]
It is not difficult to see that the countable set $J_t^T$ allows us to employ the countable mixing property to conclude the $(t, \epsilon , \pi)$-optimality. We omit the details.
\end{proof}
\begin{remark}\label{integralCOST}
One can similarly treat the complete ``standard'' cost function

$$\esssup_{\phi\in U^T_t}\mathbb{E}\Bigg[\int_t^T c\big(s,u\otimes_t\phi,X(u\otimes_t\phi)\big)ds + \xi_X(u\otimes_t\phi)\big|\mathcal{F}_t\Bigg]; 0\le t \le T,
$$
for a non-anticipative function $c:[0,T]\times U^T_0\times \mathbf{C}^n_T\rightarrow \mathbb{R}$. In order to simplify the presentation, we set $c=0$ for the rest of this article.
\end{remark}
In the remainder of this paper, we are going to present an explicit construction of $\epsilon$-optimal controls for

$$\sup_{\phi\in U^T_0}\mathbb{E}\Big[\xi_X(\phi)\Big]$$
and a suitable discrete-type pathwise dynamic programming equation which fully describes a family of approximations for the value process $u\mapsto V(u)$.

\section{Differential skeleton on controlled imbedded discrete structures}\label{CIDSsection}

In this section, we set up the basic differential operators associated with what we will call a \textit{controlled imbedded discrete structure}. It is a natural extension of the differential structure presented in Section 3 in Le\~ao, Ohashi and Simas \cite{LEAO_OHASHI2017.1}. Our philosophy is to view a controlled Wiener functional $u\mapsto Y(\cdot,u)$ as a family of simplified models one has to build in order to extract some information. The extraction of information is made by means of suitable derivative operators which mimic the infinitesimal evolution of $Y$ w.r.t Brownian state. This piece of information is precisely what we need to obtain a concrete description of value processes and the construction of their associated $\epsilon$-optimal controls.

\subsection{The underlying discrete skeleton}
The discretization procedure will be based on a class of pure jump processes driven by suitable waiting times which describe the local behavior of the Brownian motion. We briefly recall the basic properties of this skeleton. For more details, we refer to the work \cite{LEAO_OHASHI2017.1}. We set $T^{k,j}_0:=0$ and

\begin{equation}\label{stopping_times}
T^{k,j}_n := \inf\{T^{k,j}_{n-1}< t <\infty;  |B^{j}(t) - B^{j}(T^{k,j}_{n-1})| = \epsilon_k\}, \quad n \ge 1,
\end{equation}
where $\sum_{k\ge 1}\epsilon_k^2 < \infty$. Then, we define $A^{k,j}; j=1,\ldots, d, k\ge 1$

\begin{equation}\label{rw}
A^{k,j} (t) := \sum_{n=1}^{\infty} \epsilon_k \sigma^{k,j}_n1\!\!1_{\{T^{k,j}_n\leq t \}};~t\ge0.
\end{equation}
The jumps $\{\sigma^{k,j}_n; n\ge 1\}$ are given by

\begin{equation}\label{sigmaknABS}
\sigma^{k,j}_n:=\left\{
\begin{array}{rl}
1; & \hbox{if} \ \Delta A^{k,j}(T^{k,j}_n) > 0 \\
-1;& \hbox{if} \ \Delta A^{k,j}(T^{k,j}_n)< 0, \\
\end{array}
\right.
\end{equation}
By construction

\begin{equation}\label{akb}
\sup_{t\ge 0}|A^{k,j}(t) - B^j(t)|\le \epsilon_k~a.s
\end{equation}
for every $k\ge 1$. Let $\widetilde{\mathbb{F}}^{k,j} := \{ \widetilde{\mathcal{F}}^{k,j}_t; t\ge 0 \} $ be the natural filtration generated by $\{A^{k,j}(t);  t \ge 0\}$. One should notice that $\widetilde{\mathbb{F}}^{k,j}$ satisfies $\widetilde{\mathcal{F}}^{k,j}_0 = \{\Omega, \emptyset \}$ and $\widetilde{\mathcal{F}}^{k,j}_{T^{k,j}_m}=\sigma(T^{k,j}_1, \ldots, T^{k,j}_m, \Delta A^{k,j}(T^{k,j}_1), \ldots, \Delta A^{k,j}(T^{k,j}_m))$ for $m\ge 1$ and $j=1,\ldots, d$. Moreover,

$$\widetilde{\mathcal{F}}^{k,j}_{T^{k,j}_m}\cap \{T^{k,j}_m \le t < T^{k,j}_{m+1}\} = \widetilde{\mathcal{F}}^{k,j}_t\cap \{T^{k,j}_m \le t < T^{k,j}_{m+1}\},$$
for each $m\ge 0$, $j=1,\ldots, d$ and $t\ge 0$. The multi-dimensional filtration generated by $A^{k,j}; j=1,\ldots,d$ is naturally characterized as follows. Let $\widetilde{\mathbb{F}}^k := \{\widetilde{\mathcal{F}}^k_t ; 0 \leq t <\infty\}$ be the product filtration given by $\widetilde{\mathcal{F}}^k_t := \widetilde{\mathcal{F}}^{k,1}_t \otimes\widetilde{\mathcal{F}}^{k,2}_t\otimes\cdots\otimes\widetilde{\mathcal{F}}^{k,d}_t$ for $t\ge 0$. Let $\mathcal{T}:=\{T^k_m; m\ge 0\}$ be the order statistics obtained from the family of random variables $\{T^{k,j}_\ell; \ell\ge 0 ;j=1,\ldots,d\}$. That is, we set $T^k_0:=0$,

$$
T^k_1:= \inf_{1\le j\le d}\Big\{T^{k,j}_1 \Big\},\quad T^k_n:= \inf_{\substack {1\le j\le d\\ m\ge 1} } \Big\{T^{k,j}_m ; T^{k,j}_m \ge  T^k_{n-1}\Big\}
$$
for $n\ge 1$. The filtration $\widetilde{\mathbb{F}}^k$ satisfies

$$
\widetilde{\mathcal{F}}^k_t \cap\{T^k_n \le t < T^k_{n+1}\} = \widetilde{\mathcal{F}}^k_{T^k_n}\cap \{T^k_n \le  t < T^k_{n+1}\}; t\ge 0
$$
where $\widetilde{\mathcal{F}}^k_{T^k_n} = \sigma(A^{k,j}(s\wedge T^k_n); s\ge 0, 1\le j\le d)$ for each $n\ge 0$. Let $\mathcal{F}^{k}_\infty$ be the completion of $\sigma(A^{k,j}(s); s\ge 0; j=1,\ldots,d)$ and let $\mathcal{N}_{k}$ be the $\sigma$-algebra generated by all $\mathbb{P}$-null sets in $\mathcal{F}^{k}_\infty$. We denote $\mathbb{F}^{k} = (\mathcal{F}^{k}_t)_{t\ge 0}$, where $\mathcal{F}^{k}_t$ is the usual $\mathbb{P}$-augmentation (based on $\mathcal{N}_k$) satisfying the usual conditions.

Finally, from (\ref{akb}) and Lemma 2.1 in \cite{LEAO_OHASHI2013}, we do have

\begin{equation}\label{weakfiltration}
\lim_{k\rightarrow \infty}\mathbb{F}^k = \mathbb{F},
\end{equation}
weakly (in the sense of \cite{coquet1}) over $[0,T]$. Moreover, since $\sum_{k\ge 1}\epsilon^2_k < \infty$, then Lemma 2.2 in \cite{koshnevisan} yields

\begin{equation}\label{uniforTk}
\lim_{k\rightarrow+\infty}\sup_{0\le t\le T}|T^{k,j}_{\lceil\epsilon^{-2}_kt\rceil} - t|=0
\end{equation}
almost surely and in $L^2(\mathbb{P})$ for each $j=1,\ldots,d$.

\begin{definition}\label{discreteskeleton}
The structure $\mathscr{D} = \{\mathcal{T}, A^{k,j}; k\ge 1, 1\le j\le d\}$ is called a \textbf{discrete-type skeleton} for the Brownian motion.
\end{definition}
\subsection{Pathwise dynamics of the skeleton}
For a given choice of discrete-type skeleton $\mathscr{D}$, we will construct controlled functionals written on this structure. Before we proceed, it is important to point out that there exists a pathwise description of the dynamics generated by $\mathscr{D}$. Let us define

$$\mathbb{I}_k:=\Big\{ (i^k_1, \ldots, i^k_d); i^k_\ell\in \{-1,0,1\}~\forall \ell \in \{1,\ldots, d\}~\text{and}~\sum_{j=1}^d|i^k_j|=1   \Big\}$$
and $\mathbb{S}_k:=(0,+\infty)\times \mathbb{I}_k$. Let us define $\aleph: \mathbb{I}_k\rightarrow \{1,2,\dots,d\}\times\{-1,1\}$ by

\begin{equation}\label{alephamap}
\aleph(\tilde{i}^{k}):=\big(\aleph_1(\tilde{i}^{k}),\aleph_2(\tilde{i}^{k})\big):=(j,r),
\end{equation}
where $j\in\{1,\dots,d\}$ is the coordinate of $\tilde{i}^k\in \mathbb{I}_k$ which is different from zero and $r\in\{-1,1\}$ is the sign of $\tilde{i}^k$ at the coordinate $j$. The $n$-fold Cartesian product of $\mathbb{S}_k$ is denoted by $\mathbb{S}_k^n$ and a generic element of $\mathbb{S}^n_k$ will be denoted by

\begin{equation}\label{elementsSkn}
\mathbf{b}^k_n := (s^k_1,\tilde{i}^k_1, \ldots, s^k_n, \tilde{i}^k_n)\in \mathbb{S}^n_k
\end{equation}
where $(s^k_r,\tilde{i}^k_r)\in (0,+\infty)\times \mathbb{I}_k$ for $1\le r\le n$. Let us define $\eta^k_n:=(\eta^{k,1}_n, \ldots, \eta^{k,d}_n)$, where

$$
\eta^{k,j}_n:=\left\{
\begin{array}{rl}
1; & \hbox{if} \  \Delta A^{k,j} (T^k_n)>0 \\
-1;& \hbox{if} \  \Delta A^{k,j} (T^k_n)< 0 \\
0;& \hbox{if} \ \Delta A^{k,j} (T^k_n)=0.
\end{array}
\right.
$$
Let us define
\begin{equation}\label{Acaligrafico}
\mathcal{A}^k_n:= \Big(\Delta T^k_1, \eta^k_1, \ldots, \Delta T^k_n, \eta^k_n\Big)\in \mathbb{S}^n_k~a.s.
\end{equation}
One should notice that $$\widetilde{\mathcal{F}}^k_{T^k_n} = (\mathcal{A}^k_n)^{-1}(\mathcal{B}(\mathbb{S}^n_k)),$$
where $\mathcal{B}(\mathbb{S}^k_n)$ is the Borel $\sigma$-algebra generated by $\mathbb{S}^n_k; n\ge 1$.

\

\noindent \textbf{Transition probabilities}. The law of the system will evolve according to the following probability measure defined by

$$\mathbb{P}^k_r(E):=\mathbb{P}\{\mathcal{A}^k_r\in E\}; E\in \mathcal{B}(\mathbb{S}^r_k),$$
for $k,r\ge 1$. By the very definition,

$$\mathbb{P}^k_{n}(\cdot) = \mathbb{P}^k_{r}(\cdot\times \mathbb{S}^{r-n}_k)$$
for any $r> n\ge 1$.
By construction, $\mathbb{P}^k_{r}(\mathbb{S}^{n}_k\times \cdot)$ is a regular measure and $\mathcal{B}(\mathbb{S}_k)$ is countably generated, then it is known (see e.g III. 70-73 in~\cite{dellacherie2}) there exists ($\mathbb{P}^k_{n}$-a.s unique) a disintegration $\nu^k_{n,r}: \mathcal{B}(\mathbb{S}^{r-n}_k)\times\mathbb{S}^{n}_k\rightarrow[0,1]$ which realizes

$$\mathbb{P}^k_{r}(D) = \int_{\mathbb{S}^{n}_k}\int_{\mathbb{S}^{r-n}_k} 1\!\!1_{D}(\textbf{b}^k_{n},q^k_{n,r})\nu^k_{n,r} (dq^k_{n,r}|\textbf{b}^k_{n})\mathbb{P}^k_{n}(d\textbf{b}^k_{n})$$
for every $D\in \mathcal{B}(\mathbb{S}^{r}_k)$, where $q^k_{n,r}$ is the projection of $\textbf{b}^k_r$ onto the last $(r-n)$ components, i.e., $q^k_{n,r} = (s^k_{n+1},\tilde{i}^k_{n+1}, \ldots,s^k_{r},\tilde{i}^k_{r} )$ for a list $\textbf{b}^k_r = (s^k_1,\tilde{i}^k_1, \ldots, s^k_r,\tilde{i}^k_r)\in \mathbb{S}^r_k$. If $r=n+1$, we denote $\nu^k_{n+1}:=\nu^k_{n,n+1}$. By the very definition, for each $E\in \mathcal{B}(\mathbb{S}_k)$ and $\mathbf{b}^k_{n}\in \mathbb{S}_k^{n}$, we have

\begin{equation}\label{form1dis}
\nu^k_{n+1}(E|\mathbf{b}^k_{n})= \mathbb{P}\Big\{(\Delta T^k_{n+1}, \eta^k_{n+1})\in E|\mathcal{A}^k_{n} = \mathbf{b}^k_{n}\Big\}; n\ge 1.
\end{equation}
The explicit expression of the transition kernel (\ref{form1dis}) is derived as follows. For a given $\mathbf{b}^k_n = (s^k_1,\tilde{i}^k_1,\ldots, s^k_n, \tilde{i}^k_n)$, we define

\begin{equation}\label{pfunction}
\wp_\lambda(\textbf{b}^k_{n}):=\max\{1\le j\le n; \aleph_1(\tilde{i}^k_j)=\lambda\},
\end{equation}
where in (\ref{pfunction}), we make the convention that $\max\{\emptyset\}=0$. For each $\textbf{b}^k_n\in \mathbb{S}^n_k$, we set

\begin{equation}\label{tknfunction}
t^k_n(\textbf{b}^k_n) : = \sum_{\beta=1}^n s^k_\beta.
\end{equation}
We then define

\begin{equation}\label{tkmodfunction}
t^{k,\lambda}_{\mathbb{j}_{\lambda}}(\textbf{b}^k_{n}): = \sum_{\beta=1}^{\wp_\lambda(\textbf{b}^k_{n})}s^k_\beta
\end{equation}
for $\lambda\in \{1,\ldots, d\}$ and $\textbf{b}^k_n\in \mathbb{S}^n_k$. We set $t^{k,\lambda}_0 = t^k_0 = 0$ and

$$\Delta^{k,\lambda}_n(\mathbf{b}^k_n):=t^k_n(\mathbf{b}^k_n)  - t^{k,\lambda}_{\mathbb{j}_{\lambda}}(\mathbf{b}^k_n).$$

When no confusion arises, we omit the dependence on the variable $\textbf{b}^k_n$ in $t^k_n$, $t^{k,\lambda}_{\mathbb{j}_{\lambda}}$ and $\Delta^{k,\lambda}_n$. Let $f_k$ be the density of the hitting time $T^{k,1}_1$ (see e.g Section 5.3 in \cite{milstein}). We make use of the information set described in (\ref{pfunction}), (\ref{tknfunction}) and (\ref{tkmodfunction}). We define

$$f^k_{min}(\textbf{b}^k_n,j,t): = \prod_{\lambda\neq j}^{d}f_k\big(t+\Delta^{k,\lambda}_n(\textbf{b}^k_n) \big)$$
for $(\textbf{b}^k_n,j,t)\in \mathbb{S}^n_k\times \{1,\ldots, d\}\times \mathbb{R}_+.$

\begin{proposition}\label{disinteresult}
For each $\mathbf{b}^k_{n}\in\mathbb{S}^{n}_k$, $(j,\ell)\in \{1,\ldots, d\}\times \{-1,1\}$ and $-\infty< a < b < +\infty$, we have

\begin{equation}\label{disintegrationformula}
\begin{split}
& \mathbb{P}\left\{\Delta T^k_{n+1}\in (a,b); \aleph(\eta^k_{n+1})= (j,\ell)\big|\mathcal{A}^k_{n}=\mathbf{b}^k_{n}\right\}\\
&=\displaystyle \frac{1}{2}\left\{\frac{\displaystyle \int_{a+\Delta^{k,j}_n}^{b+\Delta^{k,j}_n}f_{k}\left(x\right)dx}{\displaystyle \int_{\Delta^{k,j}_n}^{+\infty}f_{k}\left(x\right)dx }\right\} \left\{\frac{\displaystyle\int_{-\infty}^{0}\int_{-s}^\infty f_{k}\big(s+t+\Delta^{k,j}_n\big)f^k_{min}(\mathbf{b}^k_n,j,t)dtds}{\displaystyle\prod_{\lambda=1}^d \int_{\Delta^{k,\lambda}_n}^{+\infty} f_{k}(t)dt}\right\};~\text{if}~d>1\\
&= \frac{1}{2}\int_a^b f_k(s)ds;~\text{if}~d=1.
\end{split}
\end{equation}
\end{proposition}
The proof of this formula is presented in \cite{LEAO_OHASHI2017.3}.

\subsection{Controlled imbedded discrete structures}
In this section, we present the differential operators acting on functionals of $\mathscr{D}$ which will constitute the basic pillars for analysing fully non-Markovian control problems. For this purpose, it will be important to enlarge $\widetilde{\mathcal{F}}^k_{T^k_n}$ and $\widetilde{\mathcal{F}}^k_{T^k_{n+1}-}$ by means of universally measurable sets. For readers who are not familiar with this class of sets, we refer to e.g \cite{bertsekas}. If $R$ is a Borel space, let $P(R)$ be the space of all probability measures defined on the Borel $\sigma$-algebra $\mathcal{B}(R)$ generated by $R$. We denote

$$\mathcal{E}(R):=\bigcap_{p\in P(R)}\mathcal{B}(R,p)$$
where $\mathcal{B}(R,p)$ is the $p$-completion of $\mathcal{B}(R)$ w.r.t $p\in P(R)$.

Let $\mathcal{G}^k_0$ be the trivial $\sigma$-algebra and for $n\ge 1$, we set

$$\mathcal{G}^k_n:= \{\mathcal{A}^k_n\in D; D~\in\mathcal{E}(\mathbb{S}^n_k) \},$$
$$\mathcal{G}^k_{n+1-} := \{(\mathcal{A}^k_n,\Delta T^k_{n+1})\in D; D\in\mathcal{E}(\mathbb{S}^n_k\times\mathbb{R}_+)\}.$$
One should notice that

$$\widetilde{\mathcal{F}}^k_{T^k_n}\subset \mathcal{G}^k_n \subset \mathcal{F}^k_{T^k_n}; n\ge 0.$$
Moreover, the following remark holds.
\begin{remark}\label{equalitycond}
If $Y\in L^1(\mathbb{P})$, then $\mathbb{E}\Big[Y|\mathcal{G}^k_n\Big] =\mathbb{E}\Big[Y|\mathcal{F}^k_{T^k_n}\Big]$~a.s and $\mathbb{E}\Big[Y|\mathcal{G}^k_{n+1-}\Big] =\mathbb{E}\Big[Y|\mathcal{F}^k_{T^k_{n+1}-}\Big]$ a.s for every $n\ge 0$.
\end{remark}

Let us start to introduce a subclass $U^{k,T^k_n}_{T^k_m}\subset U_{T^k_m}^{T^k_n}; 0\le m < n < \infty$. For $m < n$, let $U^{k,T^k_n}_{T^k_m}$ be the set of $\mathbb{F}^k$-predictable processes of the form

\begin{equation}\label{controlform}
v^k(t) = \sum_{j=m+1}^{n}v^{k}_{j-1}1\!\!1_{\{T^k_{j-1}< t\le T^k_j\}}; \quad T^k_m < t \le T^k_n,
\end{equation}
where for each $j=m+1, \ldots, n$, $v^k_{j-1}$ is an $\mathbb{A}$-valued $\mathcal{G}^k_{{j-1}}$-measurable random variable. To keep notation simple, we use the shorthand notations

\begin{equation}\label{uknm}
U^{k,n}_m: = U^{k,T^k_n}_{T^k_m}; 0\le m < n
\end{equation}
and $U^k_m$ as the set of all controls $v^k:~]]T^k_m, +\infty[[\rightarrow\mathbb{A}$ of the form

$$
v^k(t) = \sum_{j\ge m+1}v^{k}_{j-1}1\!\!1_{\{T^k_{j-1}< t\le T^k_j\}}; \quad T^k_m < t,
$$
where $v^k_{j-1}$ is an $\mathbb{A}$-valued $\mathcal{G}^k_{{j-1}}$-measurable random variable for every $j\ge m+1$ for an integer $m\ge 0$. We also use a shorthand notation for $u^k\otimes_{T^k_n} v^k$: With a slight abuse of notation, for $u^k\in U^{k,m}_0$ and $v^k\in U^{k,m}_{n}$ with $n < m$, we write

\begin{equation}\label{Kconcatenation}
(u^k\otimes_n v^k): = (u^k_0, \ldots, u^k_{n-1}, v^k_n, \ldots, v^k_{m-1} ).
\end{equation}
This notation is consistent since $u^k\otimes_{T^k_n} v^k$ only depends on the list of variables $(u^k_0, \ldots, u^k_{n-1}, v^k_n, \ldots, v^k_{m-1} )$ whenever $u^k:~]]0,T^k_n]]\rightarrow\mathbb{A}$ and $v^k:~]]T^k_n,T^k_m]]\rightarrow\mathbb{A}$ are controls of the form (\ref{controlform}) for $n < m$. With a slight abuse of notation, in order to alleviate notation we also write

\begin{equation}\label{abbreviatedU}
U_{\ell}:=U_{T^k_{\ell}}, U^m_\ell := U^{T^k_m}_{T^k_\ell}
\end{equation}
and we set

\begin{equation}\label{abuseconcatenation}
u\otimes_{\ell}\phi:=u\otimes_{T^k_{\ell}}\phi\quad \text{if}~\phi\in U_{\ell}, u\in U_0
\end{equation}
for integers $\ell\ge 0$.
\begin{remark}
It is important to observe that any control $u^k\in U^{k,q}_0$ is completely determined by a list of universally measurable functions $g^k_j:\mathbb{S}^j_k\rightarrow \mathbb{A}; 0\le j\le q-1$ in the sense that

$$u^k_j = g^k_j (\mathcal{A}^k_j); j=0, \ldots, q-1,$$
where $g^k_0$ is constant a.s.
\end{remark}


Let us now introduce the analogous concept of controlled Wiener functional but based on the filtration $\mathbb{F}^k$. For this purpose, we need to introduce some further notations. Let us define

\begin{equation}\label{ektdef}
e(k,T):=d \lceil \epsilon^{-2}_kT\rceil,
\end{equation}
where $\lceil x\rceil$ is the smallest integer greater or equal to $x\ge 0$. From (\ref{uniforTk}) and Lemma 3.1 in \cite{LEAO_OHASHI2017.2}, the authors show that

$$T^k_{e(k,t)}\rightarrow t~\text{as}~k\rightarrow+\infty$$
a.s and $L^2(\mathbb{P})$ for each $t\ge 0$. Let $O_T(\mathbb{F}^k)$ be the set of all stepwise constant $\mathbb{F}^k$-optional processes of the form

$$Z^k(t) = \sum_{n=0}^\infty Z^k(T^k_n)\mathds{1}_{\{T^k_n\le t\wedge T^k_{e(k,T)} < T^k_{n+1}\}}; 0\le t\le T,$$
where $Z^k(T^k_n)\in \mathcal{G}^k_{n}; n\ge 0$ and $\mathbb{E}[Z^k,Z^k](T) < \infty$ for every $k\ge 1$.

\begin{definition}\label{GASdef}
A \textbf{weak controlled imbedded discrete structure} $\mathcal{Y} = \big((Y^k)_{k\ge 1},\mathscr{D}\big)$ associated with a controlled Wiener functional $Y$ consists of the following objects: a discrete-type skeleton $\mathscr{D}$ and a map $u^k\mapsto Y^{k}(\cdot,u^k)$ from $U^{k,e(k,T)}_0$ to $O_T(\mathbb{F}^k)$ such that

\begin{equation}\label{antiprop}
Y^{k}(T^k_{n+1},u^k)~\text{depends on the control only at}~(u^k_0, \ldots, u^k_n)
\end{equation}
for each integer $n\in \{0,\ldots,e(k,T)-1\}$, and for each $t\in[0,T]$ and $u\in U_0$,
\begin{equation}\label{ucovprop}
\lim_{k\rightarrow+\infty}\mathbb{E}|Y^{k}(T^k_{e(k,t)},u^k) - Y(t,u)|=0,
\end{equation}
whenever $u^k\in U^{k,e(k,T)}_0$ satisfies $\lim_{k\rightarrow+\infty}u^k=u$ in $L^2_a(\mathbb{P}\times Leb)$.
\end{definition}

\begin{remark}
We will show (see Theorem \ref{density}) that for \textit{every} control $u\in  U^T_0$, one can explicitly construct a sequence $u^k\in U^{k,e(k,T)}_0; k\ge 1$ such that $\lim_{k\rightarrow+\infty}u^k=u$ in $L^2_a(\mathbb{P}\times Leb)$. Therefore, the above definition is not void. In this case, condition (\ref{ucovprop}) can be interpreted as a rather weak property of continuity.
\end{remark}

In the sequel, we are going to fix a weak controlled imbedded discrete structure $\mathcal{Y} = \big((Y^k)_{k\ge 1},\mathscr{D}\big)$ for a controlled Wiener functional $Y$. For a given $u^k\in U^{k,e(k,T)}_0$, we clearly observe that we shall apply the same arguments presented in Section 3 in \cite{LEAO_OHASHI2017.1} to obtain a differential form for $Y^k(\cdot,u^k)$.

\begin{remark}
In contrast to the framework of one fixed probability measure in \cite{LEAO_OHASHI2017.1}, in the present context it is essential to work path wisely on the level of weak controlled imbedded structures. In other words, we need to \textit{aggregrate} the structure into a single deterministic finite sequence of maps due to a possible appearance of mutually singular measures induced by $Y^{k}(\cdot,u^k)$ as $u^k$ varies over the set of controls $U^{k,e(k,T)}_0$.
\end{remark}

Let us now start the pathwise description. The whole dynamics will take place in the history space $\mathbb{H}^{k}:= \mathbb{A}\times \mathbb{S}_k$. We denote $\mathbb{H}^{k,n}$ and $\mathbb{I}^n_k$ as the $n$-fold Cartesian product of $\mathbb{H}^{k}$ and $\mathbb{I}_k$, respectively. The elements of $\mathbb{H}^{k,n}$ will be denoted by

\begin{equation}\label{elementsSknA}
\textbf{o}^{k}_n := \Big( (a^k_0,s^k_1,\tilde{i}^k_1), \ldots, (a^k_{n-1},s^k_{n}, \tilde{i}^k_n) \Big)
\end{equation}
where $(a^k_0, \ldots, a^k_{n-1})\in \mathbb{A}^n$, $(s^k_1, \ldots, s^k_n)\in (0,+\infty)^n$ and $(\tilde{i}^k_1, \ldots, \tilde{i}^k_n)\in \mathbb{I}^n_k$. In the remainder of this article, for any $(r,n)$ such that $1\le r\le n$ and $\textbf{o}^{k}_n = ( (a^k_0,s^k_1,\tilde{i}^k_1), \ldots, (a^k_{n-1},s^k_{n}, \tilde{i}^k_n))$, we denote
$$\pi_{r}(\mathbf{o}^k_n):=\Big( (a^k_0,s^k_1,\tilde{i}^k_1), \ldots, (a^k_{r-1},s^k_{r}, \tilde{i}^k_r)\Big)$$
as the projection of $\mathbf{o}^k_n\in \mathbb{H}^{k,n}$ onto the first $r$ coordinates.

If $F^k_\ell:\mathbb{H}^{k,\ell}\rightarrow\mathbb{R}; \ell=0,\ldots, e(k,T)$ is a list of universally measurable functions ($F^k_0$ is a constant), we then define

\begin{equation}\label{pathwiseVERT}
\nabla_j F^k(\mathbf{o}^k_n):=\frac{F^k_n (\mathbf{o}^k_n) - F^k_{n-1}(\pi_{n-1}(\mathbf{o}^k_{n}))}{\epsilon_k\aleph_2(\tilde{i}^k_n)}\mathds{1}_{\{\aleph_1(\textbf{b}^k_n) = j\}},
\end{equation}
for $\textbf{o}^k_n=( (a^k_0,s^k_1,\tilde{i}^k_1), \ldots, (a^k_{n-1},s^k_{n}, \tilde{i}^k_n)\in \mathbb{H}^n_k, 1\le n\le e(k,T), j=1,\ldots, d$ and  $\mathbf{b}^k_n =( (s^k_1,\tilde{i}^k_1), \ldots, (s^k_{n}, \tilde{i}^k_n))$. We also define

\begin{equation}\label{pathwiseUcond}
\mathscr{U}F^k(\mathbf{o}^k_n,a^k_n):=\int_{\mathbb{S}_k} \frac{F^k_{n+1}(\mathbf{o}^k_n, a^k_n, s^k_{n+1},\tilde{i}^k_{n+1}) - F^k_{n}(\mathbf{o}^k_n)}{\epsilon^2_k}\nu^k_{n+1}(ds^k_{n+1}d\tilde{i}^k_{n+1}|\textbf{b}^k_n),
\end{equation}
for $\textbf{o}^k_n\in \mathbb{H}^n_k,0\le n \le (k,T)-1$, where $\mathbf{b}^k_n$ are the elements of $\mathbf{o}^k_n$ which belong to $\mathbb{S}^n_k$.

The operators $(\nabla_j,\mathscr{U}; j=1,\ldots,d)$ will describe the differential form associated with the controlled structure $\mathcal{Y} = \big((Y^k)_{k\ge 1},\mathscr{D}\big)$. In particular, (\ref{pathwiseUcond}) will play the role of a Hamiltonian in the context of the control problem. Let us now make a connection of (\ref{pathwiseVERT}) and (\ref{pathwiseUcond}) to a differential form composed with the noise $\mathcal{A}^k$. Recall that any control $u^k\in U^{k,e(k,T)}_0$ is completely determined by a list of universally measurable functions $g^k_i:\mathbb{S}^i_k\rightarrow \mathbb{A}; 0\le i\le (e(k,T)-1)$. For a given list of functions $g^k$ representing a control, we define $\Xi^{k,g^k}_j:\mathbb{S}^{j}_k\rightarrow \mathbb{H}^{k,j}$ as follows

\begin{equation}\label{Xioperator}
\Xi^{k,g^k}_j \big(s^k_1, \tilde{i}^k_1, \ldots, s^k_j,\tilde{i}^k_j\big):=\Big((g^k_0,s^k_1,\tilde{i}^k_1), \ldots, (g^k_{j-1}(s^k_1,\tilde{i}^k_1, \ldots, s^k_{j-1},\tilde{i}^k_{j-1}),s^k_j, \tilde{i}^k_j)\Big)
\end{equation}
where $1\le j\le e(k,T)$. We identify $\Xi^{k,g^k}_0$ as a constant (in the action space $\mathbb{A}$) which does not necessarily depend on a list of controls $g^k = (g^k_{n-1})_{n=1}^{e(k,T)}$. The importance of working with those objects relies on the following fact: For a given list of controls $(u^k_j)_{j=0}^{e(k,T)-1}$ based on $(g^k_j)_{j=0}^{e(k,T)-1}$, the Doob-Dynkin's theorem yields the existence of a list of universally measurable functions $F^k_j:\mathbb{H}^j_k\rightarrow\mathbb{R}$ such that

\begin{equation}\label{funcREP}
Y^k(T^k_n,u^k) = F^k_n\big( \Xi^{k,g^k}_n(\mathcal{A}^k_n)  \big)~a.s,~n=0,\ldots, e(k,T).
\end{equation}
Let us know use semimartingale theory to find a differential structure for $\mathcal{Y} = \big((Y^k)_{k\ge 1},\mathscr{D}\big)$. In the sequel, we denote $\mathcal{P}^k$ as the $\mathbb{F}^k$-predictable $\sigma$-algebra over $[0,T]\times \Omega$, $(\cdot)^{p,k}$ is the $\mathbb{F}^k$-dual predictable projection operator and let $\mu_{[A^{k,j}]}$ be the Dol\'eans measure (see e.g Chap.5 in \cite{he}) generated by the point process $[A^{k,j},A^{k,j}]; 1\le j\le d, k\ge 1$. Let us denote

$$
\mathcal{D}^{\mathcal{Y},k,j}Y^k(s,u^k) :=  \sum_{\ell=1}^{\infty} \frac{\Delta Y^{k} (T^{k,j}_\ell,u^k)}{\Delta A^{k,j}(T^{k,j}_\ell)} 1\!\!1_{\{T^{k,j}_\ell=s\}}; 0\le s\le T,\quad \mathcal{Y} = \big((Y^k)_{k\ge 1},\mathscr{D}\big),
$$

\begin{equation}\label{weakinfG}
U^{\mathcal{Y},k,j}Y^k(s,u^k):=\mathbb{E}_{\mu_{[A^{k,j}]}}\Bigg[\frac{\mathcal{D}^{\mathcal{Y},k,j}Y^k(\cdot,u^k)}{\Delta A^{k,j}}\Big|\mathcal{P}^k\Bigg](s);~ 0\le s\le T, k\ge 1, 1\le j\le d,
\end{equation}
where $U^{\mathcal{Y},k,j}Y^k(\cdot,u^k)$ is the unique (up to sets of $\mu_{[A^{k,j}]}$-measure zero) $\mathbb{F}^k$-predictable process such that

$$\Bigg(\int_0^\cdot\frac{\mathcal{D}^{\mathcal{Y},k,j}Y^k(\cdot,u^k)}{\Delta A^{k,j}}d [A^{k,j},A^{k,j}]\Bigg)^{p,k} = \int_0^\cdot \mathbb{E}_{\mu_{[A^{k,j}]}}\Big[\frac{\mathcal{D}^{\mathcal{Y},k,j}Y^k(\cdot,u^k)}{\Delta A^{k,j}}\big|\mathcal{P}^k\Big] d\langle A^{k,j},A^{k,j}\rangle.$$
Here, the stochastic process $\mathcal{D}^{\mathcal{Y},k,j}Y^k(\cdot,u^k)/\Delta A^{k,j}$ is null on the complement of $\cup_{n=1}^\infty \{(\omega,t);T^{k,j}_n(\omega)=t\}.$ Let us denote

$$
\mathbb{D}^{\mathcal{Y},k,j} Y^k(s,u^k) :=  \sum_{\ell=1}^{\infty} \mathcal{D}^{\mathcal{Y},k,j}Y^k(s,u^k) 1\!\!1_{\{T^{k}_\ell\le s< T^{k}_{\ell+1}\}}
$$
and
\begin{equation}\label{uncOPERATORS}
\mathbb{U}^{\mathcal{Y},k,j}Y^k(s,u^k) := U^{\mathcal{Y},k,j}Y^k(s,u^k) \frac{d\langle A^{k,j}, A^{k,j}\rangle}{ds},
\end{equation}
for $0\le s\le T$ and $u^k\in U^k_0$. A direct application of Proposition 3.1 in \cite{LEAO_OHASHI2017.1} to $Y^k(\cdot,u^k)\in O_T(\mathbb{F}^k)$ yields

\begin{equation}\label{nonlinearDFORM}
Y^{k}(t,u^k) = Y^k(0,u^k) + \sum_{j=1}^d \oint_0^t \mathbb{D}^{\mathcal{Y},k,j}Y^k(s,u^k)dA^{k,j}(s) + \sum_{j =1}^d \int_0^t\mathbb{U}^{\mathcal{Y},k,j} Y^k(s,u^k) ds,
\end{equation}
for $0\le t \le T$. For the purpose of this article, the most important aspect of this differential representation is revealed on the time scale $T^k_0, \ldots, T^k_{e(k,T)}$. By applying Lemma 3.2 in \cite{LEAO_OHASHI2017.1}, we have

\begin{equation}\label{ger1}
\sum_{j=1}^d U^{\mathcal{Y},k,j}Y^k(T^k_{n+1},u^k) = \mathbb{E}\Bigg[\frac{\Delta Y^k(T^k_{n+1},u^k)}{\epsilon_k^2}\Big|\mathcal{G}^k_{n+1-}\Bigg]~a.s.
\end{equation}
In other words,

\begin{equation}\label{ger2}
\mathbb{E}\Bigg[\sum_{j=1}^d U^{\mathcal{Y},k,j}Y^k(T^k_{n+1},u^k)\Big|\mathcal{G}^k_n\Bigg] = \mathbb{E}\Bigg[\frac{\Delta Y^k(T^k_{n+1},u^k)}{\epsilon_k^2}\Big|\mathcal{G}^k_{n}\Bigg]~a.s,
\end{equation}
for each $0\le n\le e(k,T)-1$ and $k\ge 1$. By construction, if $Y^k$ is given by the functional representation (\ref{funcREP}), the functions $(\nabla_jF^k,\mathscr{U}F^k; j=1,\ldots, d)$ realize the following identities: For every $u^k\in U^{k,e(k,T)}_0$ associated with $(g^k_j)_{j=0}^{e(k,T)-1}$,

\begin{equation}\label{pathwiseVERTPROB}
\nabla_j F^k(\Xi^{k,g^k}_n(\mathcal{A}^k_n))=\mathcal{D}^{\mathcal{Y},k,j}Y^k(T^k_n,u^k)\mathds{1}_{\{\aleph_1(\mathcal{A}^k_n) = j\}}~a.s,
\end{equation}
for each $j=1,\ldots, d, n=1, \ldots, e(k,T)$ and

\begin{equation}\label{pathwiseUcondPROB}
\mathscr{U}F^k(\Xi^{k,g^k}_n(\mathcal{A}^k_n),g^k_n(\mathcal{A}^k_n))=\mathbb{E}\Bigg[\sum_{j=1}^d U^{\mathcal{Y},k,j}Y^k(T^k_{n+1},u^k)\Big|\mathcal{G}^k_n\Bigg]~a.s\ \
\end{equation}
for $n=0,\ldots, e(k,T)-1$.


The differential structure summarized in this section will play a key role in the obtention of $\epsilon$-optimal controls in a given non-Markovian control problem. It is not obvious that maximizing the function $\mathscr{U}F^k$ (for a suitable $F^k$) path wisely over the action space will provide such optimal objects. Next, we are going to start to explain how to achieve this.

\section{The controlled imbedded discrete structure for the value process}\label{CIDSVALUEsection}
In this section, we are going to describe controlled imbedded structures associated with an arbitrary value process

$$V(t,u) = \esssup_{v\in U_t^T}\mathbb{E}\big[\xi_X(u\otimes_t v)|\mathcal{F}_t\big]; u\in U^T_0, 0\le t\le T,$$
where the payoff $\xi$ is a bounded Borel functional and $X$ is an arbitrary controlled Wiener functional admitting a controlled structure $\big((X^k)_{k\ge 1},\mathscr{D}\big)$. Throughout this section, we are going to fix a structure

\begin{equation}\label{controlledstate}
u^k\mapsto X^{k}(\cdot,u^k)
\end{equation}
associated with $X$ such that (\ref{antiprop}) holds and we define

\begin{equation}\label{Kactionmap}
\xi_{X^k}(u^k):=\xi\big(X^k(\cdot, u^k)\big)
\end{equation}
for $u^k\in U^{e(k,T)}_0$. We then

\begin{equation}\label{discretevalueprocess}
V^{k}(T^k_n, u^k):=\esssup_{\phi^k\in U^{k,e(k,T)}_n}\mathbb{E}\Big[\xi_{X^k}(u^k\otimes_n\phi^k)\big|\mathcal{G}^k_{n}\Big]; n=1,\ldots, e(k,T)-1
\end{equation}
with boundary conditions

$$V^k(0):=V^k(0,u^k):=\sup_{\phi^k\in U^{k,e(k,T)}_0}\mathbb{E}\big[\xi_{X^k}(\phi^k)\big],\quad V^{k}(T^k_{e(k,T)}, u^k): = \xi_{X^k}(u^k).$$

One should notice that $V^k(T^k_n, u^k)$ only depends on $u^{k,n-1}:=(u^k_0, \ldots, u^k_{n-1})$ so it is natural to write

$$V^k(T^k_n, u^{k,n-1}) :=V^k(T^k_n, u^k); u^k\in U^{k,e(k,T)}_0, 0\le n\le e(k,T)$$
with the convention that $u^{k,-1}:=\mathbf{0}$. By construction, $V^k$ satisfies (\ref{antiprop}) in Definition \ref{GASdef}.

Similar to the value process $V$, we can write a dynamic programming principle for $V^k$ where the Brownian filtration is replaced by the discrete-time filtration $\mathcal{G}^k_n; n=e(k,T)-1,\ldots, 0$.

\begin{lemma}\label{lattice2}
Let $0\le n\le e(k,T)-1$. For each $\phi^k$ and $\eta^k$ in $U^{k,e(k,T)}_n$, there exists $\theta^k\in U^{k,e(k,T)}_n$ such that

$$\mathbb{E}\Big[\xi_{X^k}(\pi^k\otimes_n\theta^k)|\mathcal{G}^k_{n}\Big] =\mathbb{E}\Big[\xi_{X^k}(\pi^k\otimes_n\phi^k)|\mathcal{G}^k_{n}\Big]\vee \mathbb{E}\Big[\xi_{X^k}(\pi^k\otimes_n\eta^k)|\mathcal{G}^k_{n}\Big]~a.s$$
for every $\pi^k\in U^{k,n}_0$. Therefore, for each $\pi^k\in U^{k,n}_0$

$$\mathbb{E}\Bigg[\esssup_{\theta^k\in U^{k,e(k,T)}_n}\mathbb{E}\Big[\xi_{X^k}(\pi^k\otimes_n\theta^k)|\mathcal{G}^k_{n}\Big]\Big|\mathcal{G}^k_{j}\Bigg] = \esssup_{\theta^k\in U^{k,e(k,T)}_n}\mathbb{E}\Big[\xi_{X^k}(k,\pi^k\otimes_n\theta^k)|\mathcal{G}^k_{j}\Big]~a.s$$
if $0\le j\le n$ and $0\le n\le e(k,T)-1$.
\end{lemma}
\begin{proof}
Let $G=\Big\{\mathbb{E}\Big[\xi_{X^k}(k,\pi^k\otimes_n\phi^k)|\mathcal{G}^k_{n}\Big] > \mathbb{E}\Big[\xi_{X^k}(\pi^k\otimes_n\eta^k)|\mathcal{G}^k_{n}\Big] \Big\}$. Choose $\theta^k = \phi^k1\!\!1_{G} + \eta^k1\!\!1_{G^c}$ and apply the finite mixing property to exchange the esssup into the conditional expectation (see e.g Prop 1.1.4 in \cite{lamberton}) to conclude the proof.
\end{proof}

\begin{proposition}\label{DPprop}
For each $u^k\in U^{k,e(k,T)}_0$, the discrete-time value process $V^k(\cdot, u^k)$ satisfies

\begin{equation} \label{DPE}
\begin{split}
&V^{k} (T^k_n , u^k) = \esssup_{\theta^{k}_n \in U^{k,n+1}_n}
\mathbb{E} \Bigg[ V^{k} \left(T^k_{n+1}, u^{k,n-1} \otimes_n \theta^k_n \right)    \mid \mathcal{G}^k_{n}\Bigg];~0\le n\le e(k,T)-1\\
&V^{k} (T^k_{e(k,T)} , u^k) = \xi_{X^k} (u^k)~a.s.
\end{split}
\end{equation}
On the other hand, if a
class of processes $\{Z^{k} (T^k_n, u^k); u^k \in U^{k,e(k,T)}_0; 0\le n\le e(k,T)\}$ satisfies the dynamic programming equation
(\ref{DPE}) for every $u^k\in U^{k,e(k,T)}_0$, then $Z^{k} (T^k_n, u^k)$ coincides with $V^{k}(T^k_n , u^k)~a.s$ for every $0\le n\le e(k,T)$ and for every $u^k \in U^{k,e(k,T)}_0$.
\end{proposition}
\begin{proof}
Fix $u^k\in U^{k,e(k,T)}_0$. By using Lemma~\ref{lattice2} and the identity

\begin{small}
$$\esssup_{\phi^k\in U^{k,e(k,T)}_n}\mathbb{E}\Big[\xi_{X^k}(u^k\otimes_n\phi^k)|\mathcal{G}^k_{n}\Big] = \esssup_{\theta^k_n\in U^{k,n+1}_n}\esssup_{\phi^k\in U^{k,e(k,T)}_{n+1}}\mathbb{E}\Big[\xi_{X^k}(u^k\otimes_n(\theta^k_n\otimes_{n+1}\phi^k))|\mathcal{G}^k_{n}\Big]$$
\end{small}
a.s for each $0\le n\le e(k,T)-1$, the proof is straightforward, so we omit the details.
\end{proof}

\subsection{Measurable selection and $\epsilon$-controls}\label{constructionVALUE}
Let us now present a selection measurable theorem which will allow us to \textit{aggregate} the map $u^k\mapsto V^k(\cdot,u^k)$ into a single list of upper semi-analytic functions $F^k_m:\mathbb{H}^{k,m}\rightarrow \mathbb{R}; m=0, \ldots, e(k,T)$. As a by product, we also construct $\epsilon$-optimal controls at the level of the optimization problem

$$\sup_{\phi^k\in U^{k,e(k,T)}_0}\mathbb{E}\big[\xi_{X^k}(\phi^k)\big]; k\ge 1.$$
At first, we observe that for a given control $u^k\in U^{k,e(k,T)}_0$ associated with $\{g^k_{\ell-1}\}_{\ell=1}^{e(k,T)}$ and a given $x\in \mathbb{R}^n$, we can easily construct a Borel function $\gamma^k_{e(k,T)}:\mathbb{H}^{k,e(k,T)}\rightarrow \mathbf{D}_T^n$ such that $\gamma^k_{e(k,T)}(\textbf{o}^{k,e(k,T)})(0)=x = X^{k}(0,u^k)$ and

\begin{equation}\label{compoID}
\gamma^k_{e(k,T)}\Big(\Xi^{k,g^k}_{e(k,T)}(\mathcal{A}^k_{e(k,T)}(\omega))\Big)(t) = X^{k}\big(t,\omega,u^k(\omega)\big)
\end{equation}
for a.a $\omega$ and for every $t\in [0,T]$. For concrete examples of these constructions, we refer to Section \ref{APPLICATIONsection}.

Let us now present the selection measurable theorem which will play a key role in our methodology. For this purpose, we will make a backward argument. To keep notation simple, in the sequel we set $m=e(k,T)$. Recall that a structure of the form (\ref{controlledstate}) is fixed and it is equipped with a Borel function $\gamma^k_m:\mathbb{H}^{k,m}\rightarrow \mathbf{D}_T^n$ realizing (\ref{compoID}) with a given initial condition $x\in \mathbb{R}^n$. For such structure, we write $V^k$ as the associated value process given by (\ref{discretevalueprocess}).

We start with the map $\mathbb{V}^{k}_m:\mathbb{H}^{k,m}\rightarrow\mathbb{R}$ defined by

$$\mathbb{V}^{k}_m(\textbf{o}^{k}_m):=\xi(\gamma^k_m(\textbf{o}^{k}_m)); \textbf{o}^{k}_m\in \mathbb{H}^{k,m}.$$
By construction, $\mathbb{V}^{k}_m$ is a Borel function.

\begin{lemma}\label{UMdisint}
The probability measure $\mathbb{P}^k_{n+1}$ on $\mathcal{E}(\mathbb{S}^{n+1}_k)$ can be disintegrated as

$$\mathbb{P}^k_{n+1}(D) = \int_{\mathbb{S}^{n}_k}\int_{\mathbb{S}_k} 1\!\!1_{D}(\mathbf{b}^k_{n},s^k_{n+1},\tilde{i}^k_{n+1})\nu^k_{n+1} (ds^k_{n+1}d\tilde{i}^k_{n+1}|\mathbf{b}^k_{n})\mathbb{P}^k_{n}(d\mathbf{b}^k_{n})$$
for every $D\in \mathcal{E}(\mathbb{S}^{n+1}_k)$ and $n\ge 0$, where $E\mapsto \nu^k_{n+1}(E|\mathbf{b}^k_n)$ is the canonical extension from $\mathcal{B}(\mathbb{S}_k)$ to $\mathcal{E}(\mathbb{S}_k)$. Moreover this extension can be chosen to be a Borel function $\mathbf{b}^k_n\mapsto \nu^k_{n+1}(E|\mathbf{b}^k_n)$ from $\mathbb{S}^n_k$ to $[0,1]$ for each $E\in \mathcal{E}(\mathbb{S}_k)$.
\end{lemma}
\begin{proof}
By using the fact the Lebesgue $\sigma$-algebra contains universally measurable sets, this fact easily follows from the disintegration formula (\ref{disintegrationformula}) for $\nu ^k_{n+1}$ given by Proposition \ref{disinteresult}.
\end{proof}
An immediate consequence of Lemma \ref{UMdisint} is the following elementary result.
\begin{lemma}\label{conditionalexprep}
Let $u^k \in U^{k,m}_0$ be a control associated with universally measurable functions $(g^k_{n-1})^m_{n=1}$, where $m=e(k,T)$. Then,

\begin{equation}\label{1oiter}
\mathbb{E}\big[\xi_{X^k}(u^k)|\mathcal{G}^k_{{m-1}}\big] = \int_{\mathbb{S}_k}\mathbb{V}^{k}_m\big(\Xi^{k,g^k}_{m}(\mathcal{A}^k_{m-1}, s^k_m, \tilde{i}^k_m) \big)\nu^k_m(ds^k_m, d\tilde{i}^k_m|\mathcal{A}^k_{m-1})~a.s
\end{equation}
and

\begin{equation}\label{firstiter}
\mathbb{E} \big[\xi_{X^k}(u^k)\big] = \int_{\mathbb{S}^m_k}\mathbb{V}^{k}_m (\Xi^{k,g^k}_m(\mathbf{b}^k_m))\mathbb{P}^{k}_m(d\mathbf{b}^k_m).
\end{equation}
\end{lemma}

\begin{lemma}\label{measurabilityissue1}
The map
$$
(\mathbf{o}^k_{m-1},a^k_{m-1})\mapsto \int_{\mathbb{S}_k}\mathbb{V}^{k}_m\big(\mathbf{o}^{k}_{m-1},a^k_{m-1},s^k_m,\tilde{i}^k_{m}  \big)\nu^k_m(ds^k_m,d\tilde{i}^k_m|\mathbf{b}^k_{m-1})
$$
is a Borel function from $\mathbb{H}^{k,m-1}\times\mathbb{A}$ to $\mathbb{R}$, where $\mathbf{b}^k_{m-1}$ are the elements of $\mathbf{o}^k_{m-1}$ which belong to $\mathbb{S}^{m-1}_k$.
\end{lemma}
\begin{proof}
We shall imitate the proof of Prop. 7.29 in (\cite{bertsekas}) due to Lemma \ref{UMdisint} which says that $\mathbf{b}^k_{m-1}\mapsto\nu^k_{m}(E|\mathbf{b}^k_{m-1})$ is Borel measurable for each $E\in \mathcal{E}(\mathbb{S}_k)$.
\end{proof}

The boundedness assumption on $\xi$ yields

\begin{equation}\label{cotafund}
\Bigg|\int_{\mathbb{S}_k}\mathbb{V}^{k}_m\big(\mathbf{o}^{k}_{m-1},a^k_{m-1},s^k_m,\tilde{i}^k_{m}  \big)\nu^k_m(ds^k_m,d\tilde{i}^k_m|\mathbf{b}^k_{m-1})\Bigg|\le \sup_{\eta\in \mathbf{D}^n_T}|\xi(\eta)|< \infty
\end{equation}
for every $\mathbf{o}^{k}_{m-1} = ( (a^k_0,s^k_1,\tilde{i}^k_1), \ldots, (a^k_{m-2},s^k_{m-1}, \tilde{i}^k_{m-1})) \in \mathbb{H}^{k,m-1}, \mathbf{b}^k_{m-1} = (s^k_1,\tilde{i}^k_1), \ldots, (s^k_{m-1}, \tilde{i}^k_{m-1}))\in \mathbb{S}^{m-1}_k$ and $a^k_{m-1}\in \mathbb{A}$.

\begin{lemma}\label{iterUM}
Let $\mathbb{V}^{k}_{m-1}:\mathbb{H}^{k,m-1}\rightarrow\mathbb{R}$ be the function defined by

$$\mathbb{V}^{k}_{m-1}(\mathbf{o}^{k}_{m-1}):=\sup_{a^k_{m-1}\in \mathbb{A}}\int_{\mathbb{S}_k}\mathbb{V}^{k}_m\big(\mathbf{o}^{k}_{m-1},a^k_{m-1},s^k_m,\tilde{i}^k_{m}  \big)\nu^k_m(ds^k_m,d\tilde{i}^k_m|\mathbf{b}^k_{m-1})$$
for $\mathbf{o}^{k}_{m-1}\in \mathbb{H}^{k,m-1}$ where $\mathbf{b}^k_{m-1}$ are the elements of $\mathbf{o}^k_{m-1}$ which belong to $\mathbb{S}^{m-1}_k$. Then, $\mathbb{V}^{k}_{m-1}$ is upper semianalytic and for every $\epsilon> 0$, there exists an analytically measurable function $C^\epsilon_{k,m-1}:\mathbb{H}^{k,m-1}\rightarrow\mathbb{A}$ which realizes

\begin{equation}\label{1oiter1}
\mathbb{V}^{k}_{m-1}(\mathbf{o}^{k}_{m-1})\le \int_{\mathbb{S}_k}\mathbb{V}^{k}_m\big(\mathbf{o}^{k}_{m-1},C^\epsilon_{k,m-1}(\mathbf{o}^{k}_{m-1}),s^k_m,\tilde{i}^k_{m}  \big)\nu^k_m(ds^k_m,d\tilde{i}^k_m|\mathbf{b}^k_{m-1}) +\epsilon
\end{equation}
for every $\mathbf{o}^{k}_{m-1}\in \mathbb{H}^{k,m-1}$.
\end{lemma}
\begin{proof}
The fact that $\mathbb{V}^{k}_{m-1}$ is upper semianalytic follows from Prop 7.47 in \cite{bertsekas} and Lemma \ref{measurabilityissue1} which says the map given by
$$f(\mathbf{o}^{k}_{m-1},a^k_{m-1})= \int_{\mathbb{S}_k}\mathbb{V}^{k}_m\big(\mathbf{o}^{k}_{m-1},a^k_{m-1},s^k_m,\tilde{i}^k_{m}  \big)\nu^k_m(ds^k_m,d\tilde{i}^k_m|\mathbf{b}^k_{m-1}) $$
($\mathbf{b}^k_{m-1}$ being the $\mathbb{S}^k_{m-1}$-elements of $\mathbf{o}^k_{m-1}$) is a Borel function (hence upper semianalytic). Moreover, by construction $\mathbb{H}^{k,m-1}\times\mathbb{A}$ is a Borel set. Let

$$\mathbb{V}^{k}_{m-1}(\mathbf{o}^{k}_{m-1})= \sup_{a^k_{m-1}\in \mathbb{A}} f(\mathbf{o}^{k}_{m-1},a^k_{m-1}); \mathbf{o}^{k}_{m-1}\in \mathbb{H}^{k,m-1}.$$
The bound (\ref{cotafund}) and Prop 7.50 in \cite{bertsekas} yield the existence of an analytically measurable function $C^\epsilon_{k,m-1}:\mathbb{H}^{k,m-1}\rightarrow\mathbb{A}$ such that

$$f\big(\mathbf{o}^{k}_{m-1},C^\epsilon_{k,m-1}(\mathbf{o}^{k}_{m-1})\big)\ge \mathbb{V}^{k}_{m-1}(\mathbf{o}^{k}_{m-1})-\epsilon$$
for every $\mathbf{o}^{k}_{m-1}\in \{\mathbb{V}^{k}_{m-1} < +\infty\} = \mathbb{H}^{k,m-1}$.
\end{proof}

\begin{lemma}\label{iterDOIS}
For every $\epsilon>0$ and $u^k\in U^{k,m}_0$, there exists a control $\phi^{k,\epsilon}_{m-1}\in U^{k,m}_{m-1}$ such that

\begin{equation}\label{iterdois3}
V^k(T^k_{m-1}, u^k) \le \mathbb{E}\big[V^k(T^k_m, u^k\otimes_{m-1}\phi^{k,\epsilon}_{m-1}) |\mathcal{G}^k_{{m-1}}\big] + \epsilon~a.s.
\end{equation}
\end{lemma}

\begin{proof}
Let $(g^k_{n-1})_{n=1}^m$ be a list of universally measurable functions associated with the control $(u^k_{n-1})_{n=1}^m$. For $\epsilon>0$, let $C^\epsilon_{k,m-1}:\mathbb{H}^{k,m-1}\rightarrow\mathbb{A}$ be the analytically measurable function which realizes (\ref{1oiter1}).
We claim that

\begin{eqnarray}\label{iterdois1}
\mathbb{V}^{k}_{m-1}(\Xi^{k,g^k}_{m-1}(\mathcal{A}^k_{m-1})) &=& \esssup_{\phi^k\in U^{k,m}_{m-1}}\mathbb{E}\big[\xi_{X^k}(u^k\otimes_{m-1}\phi^k)|\mathcal{G}^k_{{m-1}}\big]~a.s\\
\nonumber& &\\
\nonumber&=&V^k(T^k_{m-1}, u^k)~a.s
\end{eqnarray}
and $\mathbb{E}\big[\xi_{X^k}(u^k\otimes_{m-1}\phi^{k,\epsilon}_{m-1})|\mathcal{G}^k_{{m-1}}\big]$ equals (a.s) to

\begin{equation}\label{iterdois2}
\int_{\mathbb{S}_k}\mathbb{V}^{k}_{m}\big(\Xi^{k,g^k,C^\epsilon_{k,m-1}}_{m}(\mathcal{A}^k_{m-1}, s^k_{m},\tilde{i}^k_{m})\big)\nu^k_{m}(ds^k_{m},d\tilde{i}^{k}_m|\mathcal{A}^k_{m-1})
\end{equation}
where $\phi^{k,\epsilon}_{m-1}:=C^{\epsilon}_{k,m-1}(\Xi^{k,g^k}_{m-1}(\mathcal{A}^k_{m-1}))$ is the composition of an analytically measurable function with an universally measurable one. In this case, it is known that $C^{\epsilon}_{k,m-1}\circ \Xi^{k,g^k}_{m-1}$ is universally measurable (see e.g Prop 7.44 in \cite{bertsekas}). This shows that $\phi^{k,\epsilon}_{m-1}$ is a control.

At this point, it is convenient to introduce the mapping $\Xi_{j}^{k,g^k,z^k}(\textbf{b}^k_j)$ given by

\begin{equation}\label{Xifuncdouble}
\Big( (g^k_0,s^k_1,\tilde{i}^k_1),\ldots, (g^k_{j-2}(\pi_{j-2}(\textbf{b}^k_{j})),s^k_{j-1},\tilde{i}^k_{j-1}), (z^k_{j-1}(\pi_{j-1}(\textbf{b}^k_{j})),s^k_{j},\tilde{i}^k_{j})\Big)
\end{equation}
where $\pi_{r}(\mathbf{b}^k_j)$ is the projection map onto the first $r$ coordinates of $\mathbf{b}^k_j$ for a given $\textbf{b}^k_j = (s^k_1,\tilde{i}^k_1, \ldots, s^k_j,\tilde{i}^k_j)$ and $j=m, \ldots, 1$. The assertion (\ref{iterdois2}) is a direct application of (\ref{1oiter}) which also yields

\begin{equation}\label{iterdois4}
V^k(T^k_{m-1}, u^k) = \esssup_{z^k\in U^{k,m}_{m-1}}\int_{\mathbb{S}_k}\mathbb{V}^{k}_{m}\big(\Xi^{k,g^k,z^k}_{m}(\mathcal{A}^k_{m-1}, s^k_{m},\tilde{i}^k_{m})\big)\nu^k_{m}(ds^k_{m},d\tilde{i}^{k}_m|\mathcal{A}^k_{m-1})
\end{equation}
almost surely. Clearly,

\begin{equation}\label{iterdois5}
\mathbb{V}^{k}_{m-1}(\Xi^{k,g^k}_{m-1}(\mathcal{A}^k_{m-1}))\ge \int_{\mathbb{S}_k}\mathbb{V}^{k}_{m}\big(\Xi^{k,g^k,z^k}_{m}(\mathcal{A}^k_{m-1}, s^k_{m},\tilde{i}^k_{m})\big)\nu^k_{m}(ds^k_{m},d\tilde{i}^{k}_m|\mathcal{A}^k_{m-1})
\end{equation}
a.s for every $z^k:\mathbb{S}_k^{m-1}\rightarrow \mathbb{A}$ universally measurable function. By considering single elements $a^k_{m-1}$ as constant controls of the form $z^k(\textbf{b}^k_{m-1}) = a^k_{m-1}; \textbf{b}^k_{m-1}\in \mathbb{S}_k^{m-1}$, the right-hand side of (\ref{iterdois4}) is greater than or equals to
\begin{equation}\label{iterdois6}
\int_{\mathbb{S}_k}\mathbb{V}^{k}_{m}\big(\Xi^{k,g^k}_{m-1}(\mathcal{A}^k_{m-1}),a^k_{m-1}, s^k_{m},\tilde{i}^k_{m}\big)
\nu^k_{m}(ds^k_{m},d\tilde{i}^{k}_m|\mathcal{A}^k_{m-1})~a.s,
\end{equation}
for every $a^k_{m-1}\in \mathbb{A}$. Summing up (\ref{iterdois4}), (\ref{iterdois5})and (\ref{iterdois6}), we conclude that (\ref{iterdois1}) holds true. By composing $\mathbb{V}^{k}_{m-1}$ with $\Xi^{k,g^k}_{m-1}(\mathcal{A}^k_{m-1})$ in (\ref{1oiter1}) and using (\ref{iterdois1}) and (\ref{iterdois2}), we conclude that (\ref{iterdois3}) holds true.
\end{proof}

We are now able to iterate the argument as follows. From (\ref{cotafund}) and a backward argument, we are able to define the sequence of functions $\mathbb{V}^{k}_\ell:\mathbb{H}^{k,\ell}\rightarrow\mathbb{R}$

\begin{equation}\label{Vkfunction}
\mathbb{V}^{k}_{\ell}(\mathbf{o}^{k}_{\ell}):=\sup_{a^k_\ell\in \mathbb{A}}\int_{\mathbb{S}_k}\mathbb{V}^{k}_{\ell+1}(\mathbf{o}^{k}_{\ell},a^k_\ell,s^k_{\ell+1},\tilde{i}^k_{\ell+1})
\nu^k_{\ell+1}(ds^k_{\ell+1},d\tilde{i}^k_{\ell+1}|\mathbf{b}^k_\ell)
\end{equation}
for $\mathbf{o}^k_{\ell}\in \mathbb{H}^{k,\ell}$ ($\mathbf{b}^k_{\ell}$ being the $\mathbb{S}^\ell_k$-elements of $\mathbf{o}^k_\ell$) and $\ell=m-1,\ldots, 1$.

\begin{lemma}\label{measurabilityissue2}
For each $j=m-1, \ldots, 1$, the map

$$(\mathbf{o}^k_{j},a^k_{j})\mapsto\int_{\mathbb{S}_k}\mathbb{V}^{k}_{j+1}(\mathbf{o}^k_{j},a^k_j,s^k_{j+1},\tilde{i}^k_{j+1})\nu^k_{j+1}(ds^k_{j+1}, d\tilde{i}^k_{j+1}|\mathbf{b}^k_{j})$$
is upper semianalytic from $\mathbb{H}^{k,j}\times \mathbb{A}$ to $\mathbb{R}$, where $\mathbf{b}^k_{j}$ are the elements of $\mathbf{o}^k_{j}$ which belong to $\mathbb{S}^k_{j}$.
\end{lemma}
\begin{proof}
The same argument used in the proof of Lemma \ref{measurabilityissue1} applies here. We omit the details.
\end{proof}

\begin{proposition}\label{detarg}
The function $\mathbb{V}^{k}_j:\mathbb{H}^{k,j}\rightarrow\mathbb{R}$ is upper semianalytic for each $j=m-1,\ldots, 1$. Moreover, for every $\epsilon>0$, there exists an analytically measurable function $C^\epsilon_{k,j}:\mathbb{H}^{k,j}\rightarrow\mathbb{A}$ such that

\begin{equation}\label{keyrecursion}
\mathbb{V}^{k}_j(\mathbf{o}^k_{j})\le \int_{\mathbb{S}_k}\mathbb{V}^{k}_{j+1}\big(\mathbf{o}^k_{j},C^\epsilon_{k,j}(\mathbf{o}^k_{j}),s^k_{j+1},\tilde{i}^k_{j+1}\big)
\nu^k_{j+1}(ds^k_{j+1},d\tilde{i}^k_{j+1}|\mathbf{b}^k_j) + \epsilon
\end{equation}
for every $\mathbf{o}^k_{j}\in \mathbb{H}^{k,j}$ ($\mathbf{b}^k_{j}$ being the $\mathbb{S}^j_k$-elements of $\mathbf{o}^k_{j}$), where $j=m-1,\ldots, 1$.
\end{proposition}
\begin{proof}
From (\ref{cotafund}), the following estimate holds true

\begin{equation}\label{cotafund1}
\sup_{\mathbf{o}^k_{j}\in \mathbb{H}^{k,j}}|\mathbb{V}^{k}_{j}(\mathbf{o}^k_{j})|\le \sup_{\eta\in \mathbf{D}^n_T}|\xi(\eta)| < + \infty,
\end{equation}
for every $j=m-1, \ldots, 1$. Now, we just repeat the argument of the proof of Lemma \ref{iterUM} jointly with Lemma \ref{measurabilityissue2}.



\end{proof}

We are now able to define the value function at step $j=0$ as follows

\begin{equation}\label{detarg1}
\mathbb{V}^{k}_0:=\sup_{a^k_0\in \mathbb{A}}\int_{\mathbb{S}_k}\mathbb{V}^{k}_{1}(a^k_0,s^k_{1},\tilde{i}^k_{1})\mathbb{P}^{k}_1(ds^k_{1},d\tilde{i}^k_{1}).
\end{equation}
Therefore, by definition of the supremum, for $\epsilon>0$, there exists $C^\epsilon_{k,0}\in \mathbb{A}$ which realizes

$$
\mathbb{V}^{k}_0 < \int_{\mathbb{S}_k}\mathbb{V}^{k}_{1}\big(C^\epsilon_{k,0},s^k_{1},\tilde{i}^k_{1}\big)
\mathbb{P}^k_{1}(ds^k_{1},d\tilde{i}^k_{1}) + \epsilon.
$$

\begin{proposition}\label{AGREGATION}
For each $j=m-1, \ldots, 0$ and a control $u^k\in U^{k,m}_0$ associated with a list of universally measurable functions $(g^k_j)_{j=0}^{m-1}$, we have

\begin{equation}\label{vkrep}
V^k(T^k_j,u^k) = \mathbb{V}^{k}_j(\Xi^{k,g^k}_j(\mathcal{A}^k_j))~a.s.
\end{equation}
Moreover, for every $\epsilon>0$, there exists a control $u^{k,\epsilon}_j$ defined by

\begin{equation}\label{explicitcontrol}
u^{k,\epsilon}_j := C^\epsilon_{k,j}(\Xi^{k,g^k}_j(\mathcal{A}^k_j)); j=m-1, \ldots, 0
\end{equation}
which realizes

\begin{equation}\label{vkineq}
V^k(T^k_j,u^k)\le \mathbb{E}\big[V^k(T^k_{j+1}, u^k\otimes_{j}u^{k,\epsilon}_j)|\mathcal{G}^k_{j}  \big] + \epsilon~a.s
\end{equation}
for every $j=m-1, \ldots, 0$.
\end{proposition}
\begin{proof}
The statements for $j=m-1$ hold true due to (\ref{iterdois3}) and (\ref{iterdois1}) in Lemma \ref{iterDOIS}. Now, by using Proposition \ref{detarg} and a backward induction argument, we shall conclude the proof.
\end{proof}
We are now able to construct an $\epsilon$-optimal control in this discrete level.

\begin{proposition}\label{epsiloncTH}
For every $\epsilon> 0$ and $k\ge 1$, there exists a control $\phi^{*,k,\epsilon}\in U^{k,m}_0$ such that

\begin{equation}\label{epsiloncTHzero}
\sup_{u^k\in U^{k,m}_0}\mathbb{E}\big[\xi_{X^k}(u^k)\big]\le \mathbb{E}\big[\xi_{X^k}(\phi^{*,k,\epsilon})\big]+ \epsilon.
\end{equation}
\end{proposition}
\begin{proof}
Fix $\epsilon>0$ and let $\eta_k(\epsilon) = \frac{\epsilon}{m}$, where we recall $m=e(k,T)$. The candidate for an $\epsilon$-optimal control is

$$\phi^{*,k,\epsilon} = (\phi^{k,\eta_k(\epsilon)}_0, \phi^{k,\eta_k(\epsilon)}_1, \ldots, \phi^{k,\eta_k(\epsilon)}_{m-1}),$$
where $\phi^{k,\eta_k(\epsilon)}_i; i=m-1, \ldots, 0$ are constructed via (\ref{explicitcontrol}). Let us check it is indeed $\epsilon$-optimal. From (\ref{vkineq}), we know that

\begin{equation}\label{epsiloncTH1}
\sup_{u^k\in U^{k,e(k,T)}_0}\mathbb{E}\big[ \xi_{X^k}(u^k)\big]\le \mathbb{E}\big[V^k(T^k_1, \phi^{k,\eta_k(\epsilon)}_0 )\big] + \eta_k(\epsilon)
\end{equation}
and

\begin{equation}\label{epsiloncTH2}
V^k(T^k_1, \phi^{k,\eta_k(\epsilon)}_0 )\le \mathbb{E}\big[V^k(T^k_2,\phi^{k,\eta_k(\epsilon)}_0\otimes_1 \phi^{k,\eta_k(\epsilon)}_1)|\mathcal{G}^k_{1}\big] + \eta_k(\epsilon)~a.s.
\end{equation}
Inequalities (\ref{epsiloncTH1}) and (\ref{epsiloncTH2}) yield

\begin{equation}\label{epsiloncTH3}
\sup_{u^k\in U^{k,e(k,T)}_0}\mathbb{E}\big[ \xi_{X^k}(u^k)\big]\le \mathbb{E}\big[V^k(T^k_2,\phi^{k,\eta_k(\epsilon)}_0\otimes_1 \phi^{k,\eta_k(\epsilon)}_1)\big] + 2\eta_k(\epsilon),
\end{equation}
where (\ref{vkineq}) implies $V^k(T^k_j,\phi^{k,\eta_k(\epsilon)}_0\otimes_1 \phi^{k,\eta_k(\epsilon)}_1\otimes_2 \ldots, \otimes_{j-1} \phi^{k,\eta_k(\epsilon)}_{j-1})$ less than or equals to

\begin{equation}\label{epsiloncTH4}
\mathbb{E}\big[ V^k(T^k_{j+1},\phi^{k,\eta_k(\epsilon)}_0\otimes_1\phi^{k,\eta_k(\epsilon)}_1\otimes \ldots, \otimes_{j} \phi^{k,\eta_k(\epsilon)}_{j})|\mathcal{G}^k_{j}\big]+\eta_k(\epsilon)
\end{equation}
a.s for $j=1, \ldots, m-1$. By iterating the argument starting from (\ref{epsiloncTH3}) and using (\ref{epsiloncTH4}), we conclude (\ref{epsiloncTHzero}).
\end{proof}

\noindent $\bullet$ \textbf{Construction of $\epsilon$-optimal control for the controlled imbedded discrete structure}: From the proof of Proposition \ref{epsiloncTH}, we can actually construct an $\epsilon$-optimal control for $\sup_{u^k\in U^{k,e(k,T)}_0}\mathbb{E}[\xi_{X^k}(u^k)]$. For $m=e(k,T)$, let us assume one has constructed the functions $C^{\frac{\epsilon}{m}}_{k,j}; j=m-1, \ldots, 0$ backwards as described in Proposition \ref{detarg} and (\ref{detarg1}). Then, we set $\phi^{k,\frac{\epsilon}{m}}_0:=C^{\frac{\epsilon}{m}}_{k,0}$ and

$$\phi^{k,\frac{\epsilon}{m}}_1:=C^{\frac{\epsilon}{m}}_{k,1}\big(\Xi^{k,g^k_0}_1(\mathcal{A}^k_1)\big)$$
where $g^k_0:=C^{\frac{\epsilon}{m}}_{k,0}$. The next step is

$$\phi^{k,\frac{\epsilon}{m}}_2:=C^{\frac{\epsilon}{m}}_{k,2}\big(\Xi^{k,g^k_0,g^k_1}_2(\mathcal{A}^k_2)\big)$$
where $g^k_1(\mathbf{b}^k_1):=C^{\frac{\epsilon}{m}}_{k,1}\big(\Xi^{k,g^k_0}_1(\mathbf{b}^k_1)\big)$ and

$$\Xi^{k,g^k_0,g^k_1}_2(\mathbf{b}^k_2):=\Big(g^k_0, \mathbf{s}^k_1,\tilde{i}^k_1, g^k_1(\mathbf{s}^k_1,\tilde{i}^k_1),\mathbf{s}^k_2,\tilde{i}^k_2 \Big)$$
for $\mathbf{b}^k_2 = (\mathbf{s}^k_1,\tilde{i}^k_1,\mathbf{s}^k_2,\tilde{i}^k_2)$. We then proceed as follows
$$\phi^{k,\frac{\epsilon}{m}}_j:=C^{\frac{\epsilon}{m}}_{k,j}\big(\Xi^{k,g^k_0,\ldots,g^k_{j-1}}_j(\mathcal{A}^k_j)\big)$$
where

$$g^k_{j-1}(\mathbf{b}^k_{j-1}):=C^{\frac{\epsilon}{m}}_{k,j-1}\big(\Xi^{k,g^k_0,\ldots,g^k_{j-2}}_{j-1}(\mathbf{b}^k_{j-1})\big),$$
and
$$\Xi^{k,g^k_0,\ldots,g^k_{j-1}}_j(\mathbf{b}^k_j):=\Big((g^k_0, \mathbf{s}^k_1,\tilde{i}^k_1),\ldots, (g^k_{j-1}(\mathbf{s}^k_1,\tilde{i}^k_1,\ldots,\mathbf{s}^k_{j-1},\tilde{i}^k_{j-1}),\mathbf{s}^k_j,\tilde{i}^k_j) \Big)$$
for $\mathbf{b}^k_j = (\mathbf{s}^k_1,\tilde{i}^k_1,\ldots,\mathbf{s}^k_j,\tilde{i}^k_j); j=1,\ldots, m-1.$ The $\epsilon$-optimal control for (\ref{epsiloncTHzero}) is then

\begin{equation}\label{constructionoptimalcontrol}
\phi^{*,k,\epsilon} = (\phi^{k,\frac{\epsilon}{m}}_0, \ldots, \phi^{k,\frac{\epsilon}{m}}_{m-1}).
\end{equation}

\subsection{Pathwise dynamic programming equation in the embedded structure}
Let us now illustrate the fact that the dynamic programming principle given by Proposition \ref{DPprop}, can be formulated in terms of an optimization procedure based on the operator (\ref{pathwiseUcondPROB}) which plays the role of a Hamiltonian. In the dynamic programming equation below, there are several concatenations of different controls at each time step. Then, if $\big((Z^k)_{k\ge1},\mathscr{D}\big)$ is a weak controlled imbedded discrete structure for $V$, we set

$$Z^{k,(n)}(\cdot, u^{k,n-1}\otimes_n\theta^k_n):= Z^k\big(\cdot, (u^k_0,\ldots,u^k_{n-1},\theta^k_n, \ldots, \theta^k_n)\big)$$
for $0\le n\le e(k,T)-1$. By construction,

$$Z^{k,(n)}(t, u^{k,n-1}\otimes_n\theta^k_n)=Z^{k,(n)}(T^k_{n+1}, u^{k,n-1}\otimes_n\theta^k_n)$$
for every $t\ge T^k_{n+1}$, i.e., it is stopped after the stopping time $T^k_{n+1}$.

For such $\big((Z^k)_{k\ge 1},\mathscr{D}\big)$, we write $\mathcal{X}_n = \big( (Z^{k,(n)})_{k\ge 1},\mathscr{D}\big)$ for $0\le n \le e(k,T)-1$. For each $n=0,\ldots, e(k,T)-1$, $Z^{k,(n)}(\cdot, u^{k,n-1}\otimes_n\theta^k_n)$ is an element of $O_T(\mathbb{F}^k)$ and hence we are able to apply decomposition (\ref{nonlinearDFORM}) to get

$$Z^{k,(n)}(t, u^{k,n-1}\otimes_n\theta^k_n) = Z^{k,(n)}(0,u^{k,n-1}\otimes_n\theta^k_n ) +$$
$$ \sum_{j=1}^d\oint_0^t \mathbb{D}^{\mathcal{X}_n,k,j}Z^{k,(n)}(s,u^{k,n-1}\otimes_n\theta^k_n)dA^{k,j}(s)+  \sum_{j=1}^d \int_0^t\mathbb{U}^{\mathcal{X}_n,k,j}Z^{k,(n)}(s,u^{k,n-1}\otimes_n\theta^k_n)ds$$
for $0\le t\le T$. Here, following the arguments which describes (\ref{uncOPERATORS}), we have

$$\mathbb{D}^{\mathcal{X}_n,k,j}Z^{k,(n)}(s, u^{k,n-1}\otimes_n\theta^k_n)(s) = \sum_{\ell=1}^{\infty} \mathcal{D}^{\mathcal{X}_n,k,j}Z^{k,(n)}(s, u^{k,n-1}\otimes_n\theta^k_n ) \mathds{1}_{\{T^k_\ell\le s < T^k_{\ell+1}\}}$$
where

$$\mathcal{D}^{\mathcal{X}_n,k,j}Z^{k,(n)}(s, u^{k,n-1}\otimes_n\theta^k_n ) = \sum_{r=1}^{\infty}\frac{\Delta Z^{k,(n)}(T^{k,j}_r,u^{k,n-1}\otimes_n\theta^k_n )}{\Delta A^{k,j}(T^{k,j}_r)}\mathds{1}_{\{T^{k,j}_r=s\}},$$
and

$$\mathbb{U}^{\mathcal{X}_n,k,j}Z^{k,(n)}(s,u^{k,n-1}\otimes_n\theta^k_n) = U^{\mathcal{X}_n,k,j}Z^{k,(n)}(s,u^{k,n-1}\otimes_n\theta^k_n)\frac{d\langle A^{k,j}\rangle }{ds},$$
where

\begin{equation}\label{openU}
\sum_{j=1}^dU^{\mathcal{X}_n,k,j}Z^{k,(n)}(T^k_{n+1}, u^{k,n-1}\otimes_n\theta^k_n) = \mathbb{E}\Bigg[\frac{\Delta Z^k(T^k_{n+1}, u^{k,n-1}\otimes_n\theta^k_n)}{\epsilon^2_k}\Big|\mathcal{G}^k_{n+1-}\Bigg]~a.s.
\end{equation}

\begin{proposition}\label{HJBprob}
For a given controlled structure $\big((X^k)_{k\ge1},\mathscr{D}\big)$ satisfying (\ref{antiprop}), the associated value process $u^k\mapsto V^k(\cdot, u^k)$ is the unique map from $U^{k,e(k,T)}_0$ to $O_T(\mathbb{F}^k)$ which satisfies

\begin{equation}\label{HJB1}
\esssup_{\theta^k_n\in U^{k,n+1}_n}\mathbb{E}\Bigg[\sum_{j=1}^dU^{\mathcal{X}_n,k,j}V^k(T^k_{n+1}, u^{k,n-1}\otimes_n\theta^k_n)\Big|\mathcal{G}^k_n\Bigg] =0~a.s;\quad 0\le n\le e(k,T)-1
\end{equation}
for every $u^k\in U^{k,e(k,T)}_0$ with boundary condition $V^k(T^k_{e(k,T)},u^k) = \xi_{X^k}(u^k)$~a.s.
\end{proposition}
\begin{proof}
It is an immediate application of Proposition \ref{DPprop}. Equation (\ref{DPE}) is equivalent to (\ref{HJB1}) due to (\ref{openU}).
\end{proof}
The results presented in Section \ref{constructionVALUE} combined with Proposition \ref{HJBprob} allow us to state the following result which summarizes the pathwise aspect of our methodology based on the operator (\ref{pathwiseUcondPROB}). It is a pathwise version of Proposition \ref{HJBprob} and it reveals that the operator $\mathscr{U}$ plays the role of a \textit{Hamiltonian}-type operator.

\begin{corollary}\label{PDECHARACTERIZATION}
For a given function $\gamma^k_{m}:\mathbb{H}^{k,m}\rightarrow\mathbb{R}$ satisfying (\ref{compoID}) where $m=e(k,T)$, the value function $(\mathbb{V}^{k}_n)_{n=0}^{m}$ associated with $V^k$ is the unique solution of

\begin{eqnarray}\label{PDEvalue}
\sup_{a^k_n\in \mathbb{A}}\mathscr{U}\mathbb{V}^k_n(\pi_n(\mathbf{o}^{k}_m),a^k_n) & = & 0;\quad n=m-1, \ldots, 0, \\
\nonumber\mathbb{V}^k_{m} (\mathbf{o}^k_{m}) &=&\xi\big(\gamma^k_{m}(\mathbf{o}^k_m)\big);~\mathbf{o}^k_{m}\in\mathbb{H}^{k,m}.
\end{eqnarray}
By composing with the state driving noise, Proposition \ref{HJBprob} can be rewritten as

\begin{eqnarray*}
\sup_{a^k_n\in \mathbb{A}}\mathscr{U}\mathbb{V}^k_n(\Xi^{k,g^k}_n(\mathcal{A}^k_n),a^k_n) & = & 0;\quad n=m-1, \ldots, 0, \\
\mathbb{V}^k_{m} (\Xi^{k,g^k}_{m}(\mathcal{A}^k_{m})) &=&\xi\big(\gamma^k_{m}(\Xi^{k,g^k}_{m}(\mathcal{A}^k_{m}))\big)~a.s,
\end{eqnarray*}
for every control $(u^k_n)_{n=0}^{m-1}$ associated with a list of universally measurable functions $(g^k_\ell)_{\ell=0}^{m-1}$.
\end{corollary}




\section{Convergence of value processes and $\epsilon$-optimal controls}\label{CONVERGENCEsection}
Throughout this section, we assume that $u\mapsto X(\cdot,u)$ is a controlled Wiener functional. The goal of this section is twofold: (i) We aim to prove that $\big((V^k)_{k\ge 1},\mathscr{D}\big)$ is a weak controlled imbedded discrete structure (Definition \ref{GASdef}) associated with the value process $V$. (ii) We want to show it is possible to construct $\epsilon$-optimal controls (see Definition \ref{epsilonDEF}) for (\ref{optimaldef1}) by analyzing the backward optimization problem

$$\argmax_{a^k_n\in \mathbb{A}} \int_{\mathbb{S}_k}\mathbb{V}^k_{n+1}(\mathbf{o}^k_n, a^k_n, s^k_{n+1},\tilde{i}^k_{n+1})\nu^k_{n+1}(ds^k_{n+1}d\tilde{i}^k_{n+1}|\mathbf{b}^k_n); \quad n=e(k,T)-1, \ldots, 0,$$
where $\mathbf{b}^k_n$ above are the elements of $\mathbf{o}^k_n$ which belong to $\mathbb{S}^k_n$.
\subsection{Approximation of controls}
In this section, we present a density result which will play a key role in this article: We want to approximate any control $u\in U^T_0$ by means of controls in the sets $U^{k,e(k,T)}_0$. For this purpose, we make use of the stochastic derivative introduced in the works \cite{LEAO_OHASHI2013, LEAO_OHASHI2017.1}. A given control $u = (u_1, \ldots, u_m)\in U^T_0$ has $m$-components and it is adapted w.r.t the filtration $\mathbb{F}$ generated by the $d$-dimensional Brownian motion $B^1, \ldots, B^d$. The key point is the identification of any control $u\in U^T_0$ with $\mathbb{F}$-martingales. In order to shorten notation, without any loss of generality, we will assume that $m=d$. The key point is the identification of any control $u\in U^T_0$ with its associated martingale

\begin{equation}\label{l2}
W(t)=\sum_{j=1}^d\int_0^t u_j(s)dB^j(s);0\le t\le T,
\end{equation}
where the control $u=(u_1, \ldots, u_d)$ is identified as the stochastic derivative operator $\mathcal{D}W=(\mathcal{D}_1W, \ldots, \mathcal{D}_dW)$ as described in Def. 4.4 in \cite{LEAO_OHASHI2017.1}. In this section, we make use of this operator computed on the subset of $\mathbb{F}$-martingales such that

\begin{equation}\label{l1}
\sup_{0\le t\le T}\|\mathcal{D}W(t)\|_{\mathbb{R}^d}\le \bar{a}~a.s,
\end{equation}
where $\bar{a}$ is the constant which describes the compact action space $\mathbb{A}$. In the sequel, we denote $\mathcal{X}_{\bar{a}} = \{W; \mathcal{D}W~\text{satisfies}~(\ref{l1})\}$ and observe that $\mathcal{X}_{\bar{a}}$ is isomorphic to $U^T_0$. For each $W\in \mathcal{X}_{\bar{a}}$, we set
$$M^{\mathcal{Y},k}(t):=\mathbb{E}[W(T)|\mathcal{F}^k_t];\quad J^k(t):=\sum_{n=1}^\infty \Delta M^{\mathcal{Y},k}(T^k_n) 1\!\!1_{ \{T^k_n\le t\}};0\le t\le T,$$
and we write $\mathcal{Y} = \big((J^k)_{k\ge 1},\mathscr{D}\big)$. The special $\mathbb{F}^k$-semimartingale decomposition is

\begin{equation}\label{stW}
J^k =M^{\mathcal{Y},k} + N^{\mathcal{Y},k}
\end{equation}
where we shall write $J^k(t)=\sum_{n=1}^\infty J^k(T^k_n)1\!\!1_{\{T^k_n \le t < T^k_{n+1}\}}; t\ge 0$, $J^k(T^k_n):=\sum_{\ell=1}^n\Delta M^{\mathcal{Y},k}(T^k_\ell); n\ge 1$. Moreover, $N^{\mathcal{Y},k}$ is the $\mathbb{F}^k$-dual predictable projection of $J^k$ which has continuous paths.

\begin{lemma}\label{severalsteps}
For every $W\in \mathcal{X}_{\bar{a}}$, the following limits hold true:

\begin{equation}\label{as1}
\lim_{k\rightarrow \infty}\mathbb{E}\sup_{0\le t\le T}|M^{\mathcal{Y},k}(t)- W(t)|^p=0
\end{equation}
for every $p>1$,

\begin{equation}\label{as2}
\lim_{k\rightarrow \infty}[M^{\mathcal{Y},k},M^{\mathcal{Y},k}](T) = [W,W](T)
\end{equation}
strongly in $L^1(\mathbb{P})$
and
\begin{equation}\label{as3}
\lim_{k\rightarrow \infty}[M^{\mathcal{Y},k},A^{k,j}](t) = [W,B^j](t)
\end{equation}
weakly in $L^1(\mathbb{P})$ for every $t\in [0,T]$ and $1\le j \le d$.
\end{lemma}
\begin{proof}
Throughout this proof, $C$ is a generic constant which may differ from line to line. At first, we observe that $\mathbb{E}|W(T)|^p<\infty$ for every $p>1$ and $W\in \mathcal{X}_{\bar{a}}$. From Lemma 2.2 in \cite{LEAO_OHASHI2013}, we know that $\lim_{k\rightarrow \infty}\mathbb{F}^k=\mathbb{F}$ weakly so that $\lim_{k\rightarrow \infty}M^{\mathcal{Y},k}= W$ uniformly in probability. Burkholder-Davis-Gundy and Jensen inequality yield

$$
\mathbb{E}[M^{\mathcal{Y},k},M^{\mathcal{Y},k} ]^{\frac{p}{2}}(T)\le C\mathbb{E}\sup_{0\le t\le T}|M^{\mathcal{Y},k}(t)|^p\le C\mathbb{E}|W(T)|^p
$$
so that
\begin{equation}\label{bound1}
\sup_{k\ge 1}\mathbb{E}[M^{\mathcal{Y},k},M^{\mathcal{Y},k} ]^{\frac{p}{2}}(T)\le C\mathbb{E}|W(T)|^p < \infty, p> 1.
\end{equation}
The bound (\ref{bound1}) implies that (\ref{as1}) holds true. Corollary 12 and Remark 6 in \cite{memin} yield
$$\lim_{k\rightarrow \infty}[M^{\mathcal{Y},k}, M^{\mathcal{Y},k}](\cdot)= [W,W](\cdot)$$
uniformly in probability so by taking $p>2$ in (\ref{bound1}), we then conclude that (\ref{as2}) holds true. We claim that

\begin{equation}\label{as4}
\lim_{k\rightarrow \infty}[M^{\mathcal{Y},k},A^{k,j}](\cdot) = [W,B^j](\cdot)
\end{equation}
uniformly in probability. From Th. 6. 22 in \cite{he}, we know that

\begin{equation}\label{bound2}
\mathbb{E}\sup_{0\le t\le T}|\Delta M^{\mathcal{Y},k}(s)|^2\le \mathbb{E}\sum_{n=1}^\infty|\Delta M^{\mathcal{Y},k}(T^k_n)|^21\!\!1_{\{T^k_n\le T\}}\le \mathbb{E}|M^{\mathcal{Y},k}(T)|^2\le C \mathbb{E}|W(T)|^2
\end{equation}
so that $\sup_{k\ge 1}\mathbb{E}\sup_{0\le t\le T}|\Delta M^{\mathcal{Y},k}(s)| < \infty$. Hence, we shall apply Prop. UT2 and Th UT3 in \cite{memin} to conclude that (\ref{as4}) holds true. Let us fix $t\in [0,T]$. Kunita-Watanabe and Burkholder-Davis-Gundy inequalities yield

\begin{equation}\label{bound3}
\mathbb{E}|[M^{\mathcal{Y},k},A^{k,j}](t)|^2\le C\big(\mathbb{E}|B^j(T)|^2\big)^{1/2} \times \big(\mathbb{E}|W(T)|^2\big)^{1/2}
\end{equation}
so that $\sup_{k\ge 1}\mathbb{E}|[M^{\mathcal{Y},k},A^{k,j}](t)|^2< \infty$ which implies that $\{[M^{\mathcal{Y},k},A^k](t); k\ge 1\}$ is uniformly integrable for every $t\in [0,T]$. From (\ref{as4}), we then conclude (\ref{as3}) holds true.
\end{proof}

\begin{lemma}\label{taulemma}
Let $\tilde{\eta}^{k,j}(t+):=\min \{T^{k,j}_n; t < T^{k,j}_n\}$ and $\tilde{\eta}^{k.j}(t):=\max\{T^{k,j}_n; T^{k,j}_n \le t \}$ for $t\in[0,T]; 1\le j\le d, k\ge 1$. Let $\tau=\inf\{t>0; |Y(t)|=1\}$ for a standard real-valued Brownian motion $Y$. Then, for any $q\ge 1$, we have
$$\mathbb{E}|\tilde{\eta}^{k,j}(t+) - \tilde{\eta}^{k,j}(t)|^q =\epsilon_k^{2q}\mathbb{E}\tau^q;~1\le j\le d, 0\le t\le T, k\ge 1.$$
\end{lemma}
\begin{proof}
Let us fix $t\in [0,T]$, $1\le j\le d$ and $q\ge 1$. Let $N^{k,j}(t) = \max\{n ;T^{k,j}_n\le t\}; t\ge 0$. By the very definition, $\tilde{\eta}^{k,j}(t+) -\tilde{\eta}^{k,j}(t) = T^{k,j}_{N^{k,j}(t)+1} -T^{k,j}_{N^{k,j}(t)}$. By construction, $\Delta T^{k,j}_{m+1}$ and $1\!\!1_{\{N^{k,j}(t)=m\}}$ are independent and $\mathbb{E}\tau^q < \infty$ for $q\ge 1$ (see Section 5.3.2 in \cite{milstein}). In this case, one can easily check

$$\mathbb{E}\Big[|\tilde{\eta}^{k,j}(t+) - \tilde{\eta}^{k,j}(t)|^q\big|N^{k,j}(t)=m\Big] =\mathbb{E}|T^{k,j}_{m+1}-T^{k,j}_m |^q = \epsilon^{2q}_k\mathbb{E}\tau^q
$$
for every $m\ge 1$ and this allows us to conclude the proof.
\end{proof}

Let us define,

$$\mathbb{D}^{\mathcal{Y},k,j}W(t):=\sum_{\ell=1}^\infty \mathcal{D}^{\mathcal{Y},k,j}W(T^{k}_\ell)1\!\!1_{\{T^k_\ell \le t< T^k_{\ell+1}\}}$$
where

$$\mathcal{D}^{\mathcal{Y},k,j}W(t):=\sum_{r=1}^\infty\frac{\Delta M^{\mathcal{Y},k}(T^{k,j}_r)}{\Delta A^{k,j}(T^{k,j}_r)}1\!\!1_{\{T^{k,j}_r=t\}}; 0\le t\le T; j=1,\ldots,d.$$
By Lemma \ref{severalsteps}, we know that $\lim_{k\rightarrow +\infty}M^{\mathcal{Y},k}= W$ in $\mathbf{B}^p(\mathbb{F})$ for every $p\ge 1$. Then, the sequence $\mathcal{Y} = \big((J^k)_{k\ge 1},\mathscr{D}\big)$ is an almost stable imbedded discrete structure for $W$ in the sense of Definition 4.3 in \cite{LEAO_OHASHI2017.1}. The only difference is that $(J^k)_{k\ge 1}$ is not a good approximating sequence for $W$ in the sense of Definition 3.2 in \cite{LEAO_OHASHI2017.1}. In case $(J^k)_{k\ge 1}$ is a good approximating sequence for $W$, then  we would just need to use Theorem 4.1 in \cite{LEAO_OHASHI2017.1} to conclude

$$
\lim_{k\rightarrow+\infty}\mathbb{D}^{\mathcal{Y},k,j}W=u_j~\text{weakly in}~L^2_a(\mathbb{P}\times Leb).
$$
Even though $(J^k)_{k\ge 1}$ is not a good approximating sequence for $W$ in the sense of Definition 3.2 in \cite{LEAO_OHASHI2017.1}, one should observe that

$$
\Delta J^k(T^k_n) = \Delta M^{\mathcal{Y},k}(T^k_n); n\ge 1,
$$
and from (\ref{as3}), we have
$$
[J^k,A^{k,j}](t) = [M^{\mathcal{Y},k},A^{k,j}](t)\rightarrow [W,B^j](t)~\text{weakly in}~L^1(\mathbb{P})
$$
for every $t\in [0,T]$. Then, we can follow the same steps in the proof of Theorem 4.1 in \cite{LEAO_OHASHI2017.1} to safely state the following lemma.

\begin{lemma}
For each $W\in \mathcal{X}_{\bar{a}}$, let $\mathcal{Y}$ be the sequence of pure jump processes given by (\ref{stW}). Then,
\begin{equation}\label{as6}
\lim_{k\rightarrow\infty}\mathbb{D}^{\mathcal{Y},k,j}W=\mathcal{D}_jW~\text{weakly in}~L^2_a(\mathbb{P}\times Leb)
\end{equation}
for each $1\le j\le d$.
\end{lemma}

We are now able to define the approximation for a given control $u\in U^T_0$. At first, we consider the predictable version of $\mathbb{D}^{\mathcal{Y},k,j}W$ as follows

\begin{equation}\label{Dbold}
\mathbf{D}^{\mathcal{Y},k,j}W(t): = \mathbb{D}^{\mathcal{Y},k,j}W(t-); 0\le t\le T,
\end{equation}
with the usual convention that $\mathbb{D}^{\mathcal{Y},k,j}W(0-)=0$.

\begin{proposition}\label{asy}
For each $W\in \mathcal{X}_{\bar{a}}$ associated with a control $\mathcal{D}W = (\mathcal{D}_1W, \ldots, \mathcal{D}_dW)$ satisfying (\ref{l1}), we have for each $1\le j\le d$,
\begin{equation}\label{pr2}
\lim_{k\rightarrow \infty}\mathbf{D}^{\mathcal{Y},k,j}W=\mathcal{D}_jW\quad \text{strongly in}~L^2_a(\mathbb{P}\times Leb).
\end{equation}
\end{proposition}
\begin{proof}
By the very definition,

\begin{eqnarray}
\nonumber\mathbb{E}[M^{\mathcal{Y},k},M^{\mathcal{Y},k} ](T) &=& \mathbb{E}\int_0^T\big\|\mathbb{D}^{\mathcal{Y},k}W(s)\big\|^2_{\mathbb{R}^d} ds\\
\label{as8}&+&\sum_{j=1}^d\mathbb{E}\sum_{n=0}^\infty|\Delta M^{\mathcal{Y},k}(T^{k,j}_n)|^2 \epsilon^{-2}_k(T^{k,j}_{n+1}-T)1\!\!1_{\{T^{k,j}_{n}\le T < T^{k,j}_{n+1}\}}
\end{eqnarray}
Triangle inequality yields $|\Delta M^{\mathcal{Y},k}(t)|\le 2 \sup_{0\le u\le T}|M^{\mathcal{Y},k}(u) - W(u)|$ a.s for every $k\ge 1$ and $t\in [0,T]$. Therefore, Cauchy-Schwartz's inequality and Lemma \ref{taulemma} yield

\begin{equation}\label{as9}
\mathbb{E}\sum_{j=1}^d|\Delta M^{\mathcal{Y},k}(\tilde{\eta}^{k,j}(T))|^2(\tilde{\eta}^{k,j}(T+)-T)\epsilon^{-2}_k\le\sum_{j=1}^d2 \epsilon^{-2}_k\Big(\mathbb{E}\sup_{0\le u\le T}|M^{\mathcal{Y},k}(u) - W(u)|^4\Big)^{1/2}
\end{equation}
$$\times \Big(\mathbb{E}|\tilde{\eta}^{k,j}(T+) - \tilde{\eta}^{k,j}(T)|^2\Big)^{1/2}\le d2C\Big(\mathbb{E}\sup_{0\le u\le T}|M^{\mathcal{Y},k}(u) - W(u)|^4\Big)^{1/2}\rightarrow 0
$$
as $k\rightarrow\infty$, for a constant $C = (\mathbb{E}\tau^2)^{1/2}$ where $\tau$ is given in Lemma \ref{taulemma}. From (\ref{as2}) in Lemma \ref{severalsteps}, (\ref{as8}) and (\ref{as9}), we have

$$\lim_{k\rightarrow \infty}\mathbb{E}\int_0^T\big\|\mathbb{D}^{\mathcal{Y},k}W(s)\big\|^2_{\mathbb{R}^d} ds = \mathbb{E}\int_0^T\big\|\mathcal{D}W(s)\big\|^2_{\mathbb{R}^d}ds  = \mathbb{E}[W,W](T)$$
so we shall apply Radon-Riesz Theorem to conclude that $\lim_{k\rightarrow \infty}\mathbb{D}^{\mathcal{Y},k}W = \mathcal{D}W$ strongly in $L^2_a(\mathbb{P}\times Leb)$. Since $\mathbb{D}^{\mathcal{Y},k}W =\mathbf{D}^{\mathcal{Y},k}W$ for $\mathbb{P}\times Leb$-a.s, we then have $\lim_{k\rightarrow \infty}\mathbf{D}^{\mathcal{Y},k}W = \mathcal{D}W$ strongly in $L^2_a(\mathbb{P}\times Leb)$.
\end{proof}
We are now able to present the main result of this section.
\begin{theorem}\label{density}
The subset $\cup_{k\ge 1} U^{k,e(k,T)}_0$ is dense in $U^T_0$ w.r.t the $L^2_a(\mathbb{P}\times Leb)$-strong topology.
\end{theorem}
\begin{proof}
In the sequel, $C$ is a constant which may defer form line to line. For a given $u\in U^T_0$, let us associate $W(\cdot) = \sum_{j=1}^d\int_0^\cdot u_j(s)dB_j(s)$ and let $\mathcal{Y} = \big((J_k)_{k\ge 1},\mathscr{D}\big)$ as given by (\ref{stW}). From Proposition \ref{asy}, we know that $\lim_{k\rightarrow \infty}\mathbf{D}^{\mathcal{Y},k,j}W=u_j\quad \text{strongly in}~L^2_a(\mathbb{P}\times Leb)$ for each $1\le j\le d$. However, there is no guarantee that $\mathbf{D}^{\mathcal{Y},k,j}W$ is essentially bounded by the constant $\bar{a}$. For simplicity of notation, let us denote $u^{k,j}(t) =\mathbf{D}^{\mathcal{Y},k,j}W(t); 0\le t\le T; j=1,\ldots, d$. We may assume (if necessary) that $\lim_{k\rightarrow+\infty}u^{k,j} = u^j$ a.s w.r.t the measure $\mathbb{P}\times Leb$. Let us define $\mathbf{d}^ku: = \big(\mathbf{d}^{k,1}u, \ldots, \mathbf{d}^{j,d}u\big)$, where

\begin{equation}\label{contructionCONTROL}
\mathbf{d}^{k,j}u:= u^{k,j}\mathds{1}_{E^c_j(k)} + \bar{a}\mathds{1}_{E_j(k)\cap H_j(k)} -\bar{a}\mathds{1}_{E_j(k) \cap H^c_j(k)},
\end{equation}
$E_j(k): = \{(\omega,t); |u^{k,j}(\omega,t)| > \bar{a}\}$ and $H_j(k) := \{u^{k,j} > 0\}$. Since, $u^{k,j}$ is $\mathbb{F}^k$-predictable, then the processes $\mathds{1}_{E^c_j(k)}$, $\mathds{1}_{E_j(k)\cap H_j(k)}$ and $\mathds{1}_{E_j(k)\cap H^c_j(k)}$ are $\mathbb{F}^k$-predictable so that $\mathbf{d}^{k,j}u$ is $\mathbb{F}^k$-predictable as well. By Theorem 5.55 in \cite{he}, the fact that $u^{k,j}$ is stepwise constant and $\mathcal{G}^k_n =\mathcal{F}^k_{T^k_n}$ (up to null sets in $\mathcal{N}_k$), we may choose (if necessary) a version of $\mathbf{d}^{k,j}u$ in such way that $(\mathbf{d}^{k,j}u)(T^k_{n+1})$ is $\mathcal{G}^k_n$-measurable for each $n\ge 0$. Therefore, $\mathbf{d}^{k}u \in U^{k,e(k,T)}_0; k\ge 1$. We fix $j=1,\ldots, d$. Now,

$$
\mathbb{E}\int_0^T |\mathbf{d}^{k,j}u(s) - u_j(s)|^2ds\le C\mathbb{E}\int_0^T \big|u^{k,j}(t) - u_j(t)\big|^2 dt$$
\begin{equation}\label{divcorreta}
+C\int_{E_j(k) \cap H_j(k)}\big| (\bar{a} - u_j(t))\big|^2 d(\mathbb{P}\times Leb)
+ C\int_{E_j(k) \cap H^c_j(k)}\big| (-\bar{a} - u_j(t))\big|^2 d(\mathbb{P}\times Leb)
\end{equation}
for $k\ge 1$. At this point, we observe that the $\mathbb{P}\times Leb$-almost sure convergence $\lim_{k\rightarrow+\infty}u^{k,j} = u_j$ implies

\begin{equation}\label{div1}
\lim_{k\rightarrow+\infty}\mathds{1}_{E_j(k)\cap H_j(k)} = \mathds{1}_{\big\{|u_j| \ge \bar{a}, u_j \ge 0\big\}}~\mathbb{P}\times Leb-a.s
\end{equation}
and

\begin{equation}\label{div2}
\lim_{k\rightarrow+\infty}\mathds{1}_{E_j(k)\cap H^c_j(k)} = \mathds{1}_{\big\{|u_j| \ge \bar{a}, u_j \le 0\big\}}~\mathbb{P}\times Leb-a.s.
\end{equation}
Therefore, from (\ref{divcorreta}), (\ref{div1}) and (\ref{div2}), we have

$$
\limsup_{k\rightarrow+\infty}\nonumber\mathbb{E}\int_0^T |\mathbf{d}^{k,j}u(s) - u_j(s)|^2ds\le C\int_{\big\{|u_j| \ge \bar{a}, u_j \ge 0\big\}}\big| (\bar{a} - u_j(t))\big|^2 d(\mathbb{P}\times Leb)
$$
\begin{equation}\label{div3}
+ C\int_{\big\{|u_j| \ge \bar{a}, u_j \le 0\big\}}\big| (-\bar{a} - u_j(t))\big|^2 d(\mathbb{P}\times Leb)=0,
\end{equation}
where (\ref{div3}) holds because $\sup_{0\le t\le T}|u_j(t)|\le\bar{a}$~a.s. This concludes the proof.
\end{proof}

\subsection{Main results}
In the sequel, it is desirable to recall the set $U^{k,n}_m$ given by (\ref{uknm}) and the concatenations (\ref{concatenation}) and (\ref{Kconcatenation}). The goal of this section is to prove the following results:

\begin{theorem}\label{VALUEconv}
Let $V(t, u) = \esssup_{\theta\in U^T_t}\mathbb{E}\big[\xi_X(u\otimes_t\theta)|\mathcal{F}_t\big]; 0\le t\le T$ be the value process associated with a payoff $\xi$ satisfying \textbf{(A1)}. Assume that $V$ and $X$ are continuous controlled Wiener functionals. Let $\big((V^k)_{k\ge 1},\mathscr{D}\big)$ be the value process (\ref{discretevalueprocess}) associated with a controlled imbedded structure $\big((X^k)_{k\ge 1},\mathscr{D}\big)$ w.r.t $X$. Assume that for every sequence $u^k\in U^{k,e(k,T)}_0$ such that $\lim_{k\rightarrow+\infty}u^k=u$ in $L^2_a(\mathbb{P}\times Leb)$ and $t\in [0,T]$

\begin{equation}\label{keyassepsilon}
\lim_{k\rightarrow +\infty}\sup_{\phi\in U^{k,e(k,T)}_{e(k,t)}}\mathbb{E}\sup_{0\le s\le T}\|X^{k}(s,u^k\otimes_{e(k,t)}\phi) - X(s,u\otimes_t \phi)\|^p_{\mathbb{R}^n}=0,
\end{equation}
for $p\ge 1$. Then,

\begin{equation}\label{STvalueconv}
\lim_{k\rightarrow +\infty}\mathbb{E}\big| V^k(T^k_{e(k,t)},u^k)  - V(t,u)\big|^p=0,~0\le t \le T,
\end{equation}
for every sequence $u^k\in U^{k,e(k,T)}_0$ such that $\lim_{k\rightarrow+\infty}u^k=u$ in $L^2_a(\mathbb{P}\times Leb)$. In particular, $\big((V^k)_{k\ge 1},\mathscr{D}\big)$ is a weak controlled imbedded discrete structure for $V$.
\end{theorem}

\begin{remark}
Recall that Lemma \ref{continuousV} states that if $\xi$ and $X$ satisfy \textbf{(A1-B1)}, then $V$ is a continuous controlled Wiener functional.
It is natural to ask if it is possible to state a stronger result
\begin{equation}\label{global}
\lim_{k\rightarrow +\infty}\sup_{u\in U_0^T}\mathbb{E}| V^k(T^k_{e(k,t)},\mathbf{d}^k(u))  - V(t,u)|=0,~0\le t\le T,
\end{equation}
where $\mathbf{d}^k(u)$ is the approximating sequence given by (\ref{contructionCONTROL}) in Theorem \ref{density}. This would produce a global approximation result over the set of controls. For the purpose of this article, the stronger convergence (\ref{global}) is not necessary so we leave this question to a further investigation.
\end{remark}

An important consequence of Theorem \ref{VALUEconv} is the next result which states that if $\big((X^k)_{k\ge1},\mathscr{D}\big)$ is a controlled imbedded discrete structure w.r.t $X$ and (\ref{keyassepsilon}) holds, then the control (\ref{constructionoptimalcontrol}) is an $\epsilon$-optimal control (see Definition \ref{epsilonDEF}) for the control problem $\sup_{u\in U^T_0}\mathbb{E}[\xi_X(u)]$.

\begin{theorem}\label{VALUEconvcontrol}
Assume the payoff $\xi$ satisfies \textbf{(A1)} and $V$ and $X$ are continuous controlled Wiener functionals. If (\ref{keyassepsilon}) holds true for $t=0$, then for any $\epsilon>0$, the $\epsilon$-optimal control $\phi^{*,k,\epsilon}$ constructed via $V^k$ in Proposition \ref{epsiloncTH} (see (\ref{constructionoptimalcontrol})) satisfies: $\phi^{*,k,\epsilon}\in U_0^T$ and
\begin{equation}\label{nearOC}
\mathbb{E}\Big[ \xi_X\big(\phi^{*,k,\epsilon}\big)\Big] \ge \sup_{u\in U^T_0}\mathbb{E}\Big[ \xi_X\big(u\big)\Big]-\epsilon
\end{equation}
for every $k$ sufficiently large.
\end{theorem}
\begin{proof}
Let us fix $\frac{\epsilon}{3} >0$. For each positive integer $k\ge 1$, let $\phi^{*,k,\epsilon}$ be the control constructed in Proposition \ref{epsiloncTH}, i.e.,

\begin{equation}\label{epso1}
\mathbb{E}\big[ \xi_{X^k}(\phi^{*,k,\epsilon})\big]\ge \sup_{\theta\in U^{k,e(k,T)}_0}\mathbb{E}\big[\xi_{X^k}(\theta)\big]  - \frac{\epsilon}{3};~k\ge 1.
\end{equation}
From Theorem \ref{VALUEconv}, we know that

\begin{equation}\label{epso2}
\Big|\sup_{\theta\in U^{k,e(k,T)}_0}\mathbb{E}\big[\xi_{X^k}(\theta)\big] - \sup_{v\in U^{T}_0}\mathbb{E}\big[\xi_{X}(v)\big] \Big| < \frac{\epsilon}{3}
\end{equation}
for every $k$ sufficiently large. By using assumptions (\ref{keyassepsilon}) and \textbf{(A1)}, we also know there exists a positive constant $C$ such that

\begin{equation}\label{epso3}
\Big| \mathbb{E}\big[\xi_{X^k}(\phi^{*,k,\epsilon})\big]   - \mathbb{E}\big[\xi_{X}(\phi^{*,k,\epsilon})\big] \Big|\le C \big(\mathbb{E}\sup_{0\le t\le T}\| X^k(t,\phi^{*,k,\epsilon}) - X(t,\phi^{*,k,\epsilon})\|^p_{\mathbb{R}^n}\big)^{\frac{1}{\alpha}} < \frac{\epsilon}{3}
\end{equation}
for every $k$ sufficiently large and $\alpha=p/\gamma$. Summing up inequalities (\ref{epso1}), (\ref{epso2}) and (\ref{epso3}), we then have

\begin{equation}\label{epso4}
\mathbb{E}\big[\xi_{X}(\phi^{*,k,\epsilon})\big] + \frac{\epsilon}{3}\ge \mathbb{E}\big[\xi_{X^k}(\phi^{*,k,\epsilon})\big]\ge \sup_{\theta\in U^{k,e(k,T)}_0}\mathbb{E}\big[\xi_{X^k}(\theta)\big]-\frac{\epsilon}{3} \ge \sup_{v\in U^{T}_0}\mathbb{E}\big[\xi_{X}(v)\big] -\frac{2\epsilon}{3}
\end{equation}
for every $k$ sufficiently large.
 \end{proof}
In the sequel, it is desirable to recall the sets $U_\ell$ and $U^m_\ell$ (see (\ref{abbreviatedU})) and the concatenation (\ref{abuseconcatenation}). In the remainder of this section, the assumptions of Theorem \ref{VALUEconv} will be in force. In what follows, we are going to fix a controlled imbedded structure $\big((X^k)_{k\ge 1},\mathscr{D}\big)$ satisfying (\ref{keyassepsilon}) and $u\in U_0$. If $\phi\in U_{n}$, we define

$$\tilde{\phi}^k(t):=\sum_{i=n+1}^\infty \mathbb{E}\big[\phi(T^k_{i-1}+)|\mathcal{G}^k_{i-1}\big]1\!\!1_{\{T^k_{i-1}< t\le T^k_i\}}; t\ge 0.
$$
It is immediate that the following lemma holds true.

\begin{lemma}\label{fundlemmaconv1}
For each $t\in [0,T)$ and $u^k = (u^k_0, \ldots, u^k_{e(k,T)-1})\in U^{k,e(k,T)}_0$, we have
$$\esssup_{\phi\in U^{k,e(k,T)}_{e(k,t)}}\mathbb{E}\Big[ \xi_{X^k}(u^k\otimes_{e(k,t)}\phi)|\mathcal{G}^k_{e(k,t)}  \Big] =\esssup_{\phi\in U^{e(k,T)}_{e(k,t)}}\mathbb{E}\Big[ \xi_{X^k}(u^k\otimes_{e(k,t)}\tilde{\phi}^k)|\mathcal{G}^k_{e(k,t)}  \Big]~a.s.
$$
\end{lemma}
\begin{lemma}
For each $u\in U_0$,
\begin{equation}\label{deltaconv2}
\esssup_{\phi\in U_{e(k,\cdot)}} \mathbb{E}\Big[\xi_X\big(u\otimes_{e(k,\cdot)}\phi\big)\big|\mathcal{G}^k_{e(k,\cdot)}\Big]\rightarrow V(\cdot,u)
\end{equation}
in $\mathbf{B}^p(\mathbb{F})$ as $k\rightarrow+\infty$ for $p\ge 1$.
\end{lemma}
\begin{proof}
The lattice property of $\big\{\mathbb{E}\big[  \xi_X(u\otimes_{e(k,t)}\phi)|\mathcal{F}_{T^k_{e(k,t)}}\big] ; \phi\in U_{e(k,t)}\big\}$ for every $t\in [0,T]$ yields

\begin{equation}\label{deltaconv1}
\mathbb{E}\Big[V(u,T^k_{e(k,t)})\big|\mathcal{F}^k_{T^k_{e(k,t)}}\Big] = \esssup_{\phi\in U_{e(k,t)}} \mathbb{E}\Big[\xi_X\big(u\otimes_{e(k,t)}\phi\big)\big|\mathcal{F}^k_{T^k_{e(k,t)}}\Big]~a.s
\end{equation}
for each $t\in [0,T]$ and $k\ge 1$. Jensen's inequality, the weak convergence $\lim_{k\rightarrow+\infty}\mathbb{F}^k=\mathbb{F}, \mathbb{F}^k\subset \mathbb{F}$ and the fact that $V(\cdot,u)$ has continuous paths allow us to apply Th.1 in \cite{coquet1} to get


$$\Big\|\mathbb{E}\Big[V(u,T^k_{e(k,\cdot)})\big|\mathcal{F}^k_{T^k_{e(k,\cdot)}}\Big] -V(u,T^k_{e(k,\cdot)}) \Big\|^p_{\mathbf{B}^p}\le \Big\| \mathbb{E}\big[V(u,\cdot)|\mathcal{F}^k_\cdot\big] -V(u,\cdot)\Big\|^p_{\mathbf{B}^p} \rightarrow 0$$
as $k\rightarrow+\infty$. By the pathwise uniform continuity of $t\mapsto V(t,u)$ on $[0,T]$, we have: For any $\epsilon>0$, there exists a $\delta = \delta(\omega,\epsilon)$ such that

$$|T^k_{e(k,t)}(\omega) - t|< \delta \Longrightarrow |V(t,u,\omega) - V(T^k_{e(k,t)},u,\omega) |< \epsilon.$$
By Lemma 3.1 (inequality (3.2)) in \cite{LEAO_OHASHI2017.2} and Lemma 2.2 in \cite{koshnevisan}, we have

$$\sup_{0\le t\le T}|T^k_{e(k,t)} - t|\rightarrow 0$$
as $k\rightarrow+\infty$ a.s. Since $\mathcal{G}^k_{e(k,\cdot)}  =\mathcal{F}^k_{T^k_{e(k,t)}}$ up to null sets in $\mathcal{N}_k$, then, we can safely state that (\ref{deltaconv2}) holds true.
\end{proof}

\begin{lemma}\label{lemmaCV1}
For every $t\in [0,T]$, and $u\in U_0$,
\begin{equation}\label{asconvSk}
S^k(t,u):=\esssup_{\phi\in U_{e(k,t)}} \mathbb{E}\Big[\xi_X\big(u\otimes_{e(k,t)}\tilde{\phi}^k\big)\big|\mathcal{F}^k_{T^k_{e(k,t)}}\Big]\rightarrow V(t,u)
\end{equation}
$a.s$ as $k\rightarrow+\infty$.
\end{lemma}
\begin{proof}
We fix $t\in [0,T)$ and $u\in U^T_0$. The following inequalities hold:

\begin{equation}\label{sub1}
S^k(t,u)\le\esssup_{\phi\in U_{e(k,t)}} \mathbb{E}\Big[\xi_X\big(u\otimes_{e(k,t)}\phi\big)\big|\mathcal{F}^k_{T^k_{e(k,t)}}\Big] ~a.s~\forall k\ge 1,
\end{equation}
and by Lemma \ref{epsilonrandomop}, for $\epsilon>0$ we know there exists a control $\eta\in U^T_t$ such that

\begin{equation}\label{sub2}
V(t,u) < \epsilon + \mathbb{E}\big[\xi_X(u\otimes_t \eta)|\mathcal{F}_t\big]~a.s.
\end{equation}
Choose $\eta^k\in U^{k,e(k,T)}_0$ such that $\eta^k\rightarrow \eta$ in $L^2_a(\mathbb{P}\times Leb)$ as $k\rightarrow+\infty$. Assumptions \textbf{(A1-B1)} yields

$$\lim_{k\rightarrow+\infty}\mathbb{E}\big[\xi_X(u\otimes_{e(k,t)} \eta^k)|\mathcal{F}_t\big] =\mathbb{E}\big[\xi_X(u\otimes_t \eta)|\mathcal{F}_t\big]$$
in $L^2(\mathbb{P})$, but since $T^k_{e(k,t)}\rightarrow t$ a.s and $\mathbb{F}^k$ converges weakly to $
\mathbb{F}$, we actually have

\begin{equation}\label{littlelemma1}
\lim_{k\rightarrow+\infty}\mathbb{E}\big[\xi_X(u\otimes_{e(k,t)} \eta^k)|\mathcal{F}^k_{T^k_{e(k,t)}}\big] =\mathbb{E}\big[\xi_X(u\otimes_t \eta)|\mathcal{F}_t\big]\quad \text{in}~L^2(\mathbb{P}).
\end{equation}
From (\ref{sub2}), (\ref{littlelemma1}) and the definition of $\esssup$, we can find a subsequence $\gamma(k)$ such that

\begin{equation}\label{sub3}
V(t,u)\le \epsilon + \liminf_{k\rightarrow+\infty}S^{\gamma(k)}(t,u)~a.s.
\end{equation}

For this subsequence, we make use of (\ref{deltaconv2}) to extract a further subsequence $\{v(k)\} \subset \{\gamma(k)\}$ such that
$$\lim_{k\rightarrow+\infty}\esssup_{\phi\in U_{e(v(k),t)}} \mathbb{E}\Big[\xi_X\big(u\otimes_{e(v(k),t)}\phi\big)\big|\mathcal{F}^{v(k)}_{T^k_{e(v(k),t)}}\Big] = V(t,u)~a.s.$$
From (\ref{sub1}), we have $\limsup_{k}S^{v(k)}(t,u)\le V(t,u)~a.s$ and (\ref{sub3}) allows us to conclude

$$\lim_{k}S^{v(k)}(t,u) = V(t,u)~a.s.$$
The above argument shows that every subsequence of $\{S^k(t,u); k\ge 1\}$ has a further convergent subsequence which converges almost surely to the same limit $V(t,u)$. This shows the entire sequence converges and (\ref{asconvSk}) holds true.

\end{proof}

\begin{lemma}\label{lemmaskconv}

\begin{equation}
\mathbb{E}\Big|V^k(T^k_{e(k,t)},u^k) - S^k(t,u)\Big|^p\le \sup_{\phi\in U^{k,e(k,T)}_{e(k,t)}}\mathbb{E}\Big|\xi_{X^k}\big(k,u^k\otimes_{e(k,t)}\phi\big)  - \xi_X\big(u\otimes_{e(k,t)}\phi\big) \Big|^p
\end{equation}
for every $u\in U_0, u^k\in U^{k,e(k,T)}_0 ; k\ge 1$ and $t\in [0,T]$.
\end{lemma}
\begin{proof}
We fix $u\in U_0, u^k\in U^{k,e(k,T)}_0, t\in [0,T]$ and $k\ge 1$. Clearly, if $\phi\in U_{e(k,t)}$, then

$$\Bigg|\esssup_{\phi\in U_{e(k,t)}}\mathbb{E}\Big[\xi_X\big(u\otimes_{e(k,t)}\tilde{\phi}^k\big)\big|\mathcal{F}^k_{T^k_{e(k,t)}}\Big]- \esssup_{\phi\in U_{e(k,t)}}\mathbb{E}\Big[\xi_{X^k}\big(k,u^k\otimes_{e(k,t)}\tilde{\phi}^k\big)\big|\mathcal{F}^k_{T^k_{e(k,t)}}\Big]\Bigg|
$$
$$\le\esssup_{\phi\in U_{e(k,t)}}\mathbb{E}\Big[|\xi_{X^k}\big(u^k\otimes_{e(k,t)}\tilde{\phi}^k\big)  - \xi_X\big(u\otimes_{e(k,t)}\tilde{\phi}^k\big)|\big|\mathcal{F}^k_{T^k_{e(k,t)}}\Big]~a.s.
$$
By applying Lemma \ref{fundlemmaconv1}, we then arrive at

$$\Big|S^k(u,t)- V^k(T^k_{e(k,t)},u^k)\Big|
$$
\begin{equation}\label{sub4}
\le\esssup_{\phi\in U_{e(k,t)}}\mathbb{E}\Big[|\xi_{X^k}\big(u^k\otimes_{e(k,t)}\tilde{\phi}^k\big)  - \xi_X\big(u\otimes_{e(k,t)}\tilde{\phi}^k\big)|\big|\mathcal{F}^k_{T^k_{e(k,t)}}\Big]=:J^k_t~a.s.
\end{equation}
The set $\big\{\mathbb{E}\big[|\xi_{X^k}\big(u^k\otimes_{e(k,t)}\tilde{\phi}^k\big)  - \xi_X\big(u\otimes_{e(k,t)}\tilde{\phi}^k\big)|\big|\mathcal{F}^k_{T^k_{e(k,t)}}\big]; \phi\in U_{e(k,t)}\Big\}
 $ has the lattice property and hence (see e.g Prop 1.1.3 in \cite{lamberton}), there exists a sequence $\{\phi_i; i\ge 1\}\subset U_{e(k,t)}$ such that

$$\mathbb{E}\Big[|\xi_{X^k}\big(u^k\otimes_{e(k,t)}\tilde{\phi_i}^k\big)  - \xi_X\big(u\otimes_{e(k,t)}\tilde{\phi_i}^k\big)|\big|\mathcal{F}^k_{T^k_{e(k,t)}}\Big]\uparrow J^k_t
~a.s $$
as $i\rightarrow +\infty$. The estimate (\ref{sub4}), Jensen's inequality and monotone convergence theorem yield

\begin{eqnarray*}
\mathbb{E}\Big|S^k(t,u)- V^k(T^k_{e(k,t)},u^k)\Big|^p &\le& \limsup_{i\rightarrow+\infty}\mathbb{E}\Big|\xi_{X^k}\big(u^k\otimes_{e(k,t)}\tilde{\phi_i}^k\big)  - \xi_X\big(u\otimes_{e(k,t)}\tilde{\phi_i}^k\big)\Big|^p\\
& &\\
&\le&\sup_{\phi\in U^{k,e(k,T)}_{e(k,t)}}\mathbb{E}\Big|\xi_{X^k}\big(u^k\otimes_{e(k,t)}\phi\big)  - \xi_X\big(u\otimes_{e(k,t)}\phi\big)\Big|^p
\end{eqnarray*}
for $k\ge 1$.
\end{proof}

\

\noindent \textbf{Proof of Theorem \ref{VALUEconv}:} Lemma \ref{lemmaskconv}, \textbf{(A1)}, (\ref{keyassepsilon}) and H\"{o}lder's inequality yield

$$
\mathbb{E}\Big|S^k(t,u)- V^k(T^k_{e(k,t)},u^k)\Big|^p \le\sup_{\phi\in U^{k,e(k,T)}_{e(k,t)}}\mathbb{E}\Big|\xi_{X^k}\big(u^k\otimes_{e(k,t)}\phi\big)  - \xi_X\big(u\otimes_{e(k,t)}\phi\big)\Big|^p$$
$$\le C\Big(\sup_{\phi\in U^{k,e(k,T)}_{e(k,t)}}\mathbb{E}\sup_{0\le s\le T}\Big|X^k\big(s,u^k\otimes_{e(k,t)}\phi\big)  - X\big(s,u\otimes_{e(k,t)}\phi\big)\Big|^p\Big)^\gamma \rightarrow 0
$$
as $k\rightarrow+\infty$. Lemma \ref{lemmaCV1} and triangle inequality allow us to conclude the proof.

\section{Applications}\label{APPLICATIONsection}

We now show that the abstract results obtained in this article can be applied to the
concrete examples mentioned in the Introduction. We will treat two cases: The state controlled process is a path-dependent SDE driven by Brownian motion and a SDE driven by fractional Brownian motion with additive noise. Section \ref{portfolioOPT} illustrates the method with a portfolio optimization problem based on a risky asset process driven by path-dependent coefficients.
\subsection{Path-dependent controlled SDEs}
In the sequel, we make use of the following notation

$$\omega_t: = \omega(t\wedge \cdot); \omega \in \mathbf{D}^n_T.$$

This notation is naturally extended to processes. We say that $F$ is a \textit{non-anticipative} functional if it is a Borel mapping and

$$F(t,\omega) = F(t,\omega_t); (t,\omega)\in[0,T]\times \mathbf{D}^n_T.$$

The underlying state process is the following $n$-dimensional controlled SDE

\begin{equation}\label{pdsdeBM}
dX^u(t) = \alpha(t,X^u_t,u(t))dt + \sigma(t,X^u_t,u(t))dB(t); 0\le t\le T,
\end{equation}
with a given initial condition $X^u(0)=x\in \mathbb{R}^n$. We define $\Lambda_T:=\{(t,\omega_t); t\in [0,T]; \omega\in \mathbf{D}^n_T\}$ and we endow this set with the metric

$$d_{1/2}((t,\omega); (t',\omega')): = \sup_{0\le u\le T}\|\omega(u\wedge t) - \omega'(u\wedge t')\|_{\mathbb{R}^n} + |t-t'|^{1/2}.$$
Then, $(\Lambda_T,d_{1/2})$ is a complete metric space equipped with the Borel $\sigma$-algebra. The coefficients of the SDE will satisfy the following regularity conditions:

\

\noindent \textbf{Assumption (C1)}: The non-anticipative mappings $\alpha: \Lambda_T\times \mathbb{A}\rightarrow \mathbb{R}^n$ and $\sigma:\Lambda_T\times\mathbb{A}\rightarrow \mathbb{R}^{n\times d}$ are Lipschitz continuous, i.e., there exists a pair of constants $K_{Lip}=(K_{1,Lip}, K_{2,Lip})$ such that

$$\|\alpha(t,\omega, a) - \alpha(t',\omega',b)\|_{\mathbb{R}^n} + \|\sigma(t,\omega,a) - \sigma(t',\omega',b)\|_{\mathbb{R}^{n\times d}}
$$
$$\le K_{1,Lip}d_{1/2} \big((t,\omega); (t',\omega')\big) + K_{2,Lip}\|a-b\|_{\mathbb{R}^m}$$
for every $t,t'\in [0,T]$ and $\omega,\omega'\in \mathbf{D}^n_T$ and $a,b\in \mathbb{A}$. One can easily check by routine arguments that the SDE (\ref{pdsdeBM}) admits a strong solution such that

\begin{equation}\label{integraPDSDE}
\sup_{u\in U^T_0}\mathbb{E}\sup_{0\le t\le T}\|X^u(t)\|^{2p}_{\mathbb{R}^n}\le C(1+\|x_0\|^{2p}_{\mathbb{R}^n})\exp(CT),
\end{equation}
where $X(0)=x_0$, $C$ is a constant depending on $T>0,p\ge 1$, $K_{Lip}$ and the compact set $\mathbb{A}$.

\begin{remark}
Due to Assumption \textbf{(C1)}, it is a routine exercise to check that the controlled SDE satisfies Assumption \textbf{(B1)}. Therefore, Lemma \ref{continuousV} implies that the associated value process
$$V(t,u) = \esssup_{\theta\in U_t^T}\mathbb{E}\big[\xi_X(u\otimes_t\phi)|\mathcal{F}_t\big]; 0\le t\le T$$
has continuous paths for each $u\in U_0$.
\end{remark}

\

\noindent \textbullet~\textbf{Definition of the} \textbf{controlled imbedded structure} for (\ref{pdsdeBM}): Let us now construct a controlled imbedded structure $\big((X^k)_{k\ge 1},\mathscr{D}\big)$ associated with (\ref{pdsdeBM}). In the sequel, in order to alleviate notation, we are going to write controlled processes as

$$X^{k,u^k},~X^u$$
rather than $X^k(\cdot,u^k)$ and $X(\cdot,u)$, respectively, as in previous sections. Let us fix a control $u^k= (u^k_0, \ldots, u^k_{n-1}, \ldots)$ based on a collection of universally measurable functions $g^k_\ell:\mathbb{S}^\ell_k\rightarrow\mathbb{A}$ realizing $u^k_{\ell} = g^k_\ell(\mathcal{A}^k_\ell)$ a.s. At first, we construct an Euler-Maruyama-type scheme based on the random partition $(T^k_n)_{n\ge 0}$ as follows: Let

$$S^{k,j}_n:=\max\{T^{k,j}_p; T^{k,j}_p\le T^{k}_{n-1}\}; 1\le j\le d, k,n\ge 1.$$
It is important to notice that $S^{k,j}_n$ is \textit{not} a stopping time, but it is $\mathcal{G}^k_{n-1}$-measurable for every $n\ge 1$. Moreover, $S^{k,j}_1=0~a.s, 1\le j\le d$. For a given information set $\textbf{b}^k_{\ell}\in \mathbb{S}^{\ell}_k$, we recall (see (\ref{pfunction}))

$$\wp_\lambda(\textbf{b}^k_\ell) =\max\Big\{1\leq j\leq \ell, \quad \tilde{i}^k_j=(\underbrace{0,\dots,0}_{\lambda-1},r,\underbrace{0,\dots,0}_{d-\lambda}),~ r \in \{-1,1\}\Big\},$$
for $\lambda\in \{1, \ldots, d\}$. We observe that $\wp_\lambda(\mathbf{b}^k_\ell)$ only depends on $(\tilde{i}^k_1, \ldots, \tilde{i}^k_\ell)$ for a given information set $\mathbf{b}^k_\ell\in \mathbb{S}^\ell_k$.

Let $\pi_\ell:\mathbb{S}_k^{\infty}\rightarrow\mathbb{S}_k^\ell$ be the standard projection onto the $\ell$-th coordinate. For each $1\le j\le d$, let us define $g^{k,j}_\ell:\mathbb{S}^{\ell}_k\rightarrow\mathbb{A}$ given by

$$g^{k,j}_\ell(\textbf{b}^k_\ell):= g^k_{\wp_j(\textbf{b}^k_{\ell})}\big( \pi_{\wp_j(\textbf{b}^k_\ell)}(\textbf{b}^k_\ell) \big); \textbf{b}^k_\ell\in \mathbb{S}^\ell_k; \ell\ge 1,$$
where we set $g^{k,j}_0 := g^k_0, 1\le j\le d$.
\begin{remark}
For every sequence of universally measurable functions $g^k_\ell:\mathbb{S}^\ell_k\rightarrow\mathbb{A}$, one can easily check that $g^{k,j}_\ell:\mathbb{S}^\ell_k\rightarrow\mathbb{A}$ is universally measurable for each $\ell\ge 1$ and $1\le j\le d$. In this case,

$$g^{k,j}_{n-1}(\mathcal{A}^k_{n-1})= g^k_{\wp_j(\eta^k_1,\ldots,\eta^k_{{n-1}})}\big(\mathcal{A}^k_{\wp_j(\eta^k_1,\ldots,\eta^k_{{n-1}})}\big)$$
is $\mathcal{G}^k_{n-1}$-measurable for each $n\ge 1$ and $1\le j\le d$.
\end{remark}

Let us define $\mathbb{X}^{k,u^k}(0):=x$ and $\mathbb{X}^{k,u^k}_0:=\mathbb{X}^{k,u^k}_{T^k_0} := x$ (hence $\mathbb{X}^{k,u^k}_{S^{k,j}_1}=\mathbb{X}^{k,u_k}_{T^k_0}$) is the constant function $x$ over $[0,T]$. Let us define

\begin{eqnarray*}
\mathbb{X}^{i,k,u^k}(T^k_q)&:=&\mathbb{X}^{i,k,u^k,}(T^k_{q-1}) + \alpha^i\big(T^k_{q-1},\mathbb{X}^{k,u^k}_{T^k_{q-1}},g^k_{q-1}(\mathcal{A}^k_{q-1})\big)\Delta T^k_{q}\\
& &\\
&+& \sum_{j=1}^d \sigma^{ij}(S^{k,j}_{q},\mathbb{X}^{k,u^k}_{S^{k,j}_{q}},g^{k,j}_{q-1}(\mathcal{A}^k_{q-1})\big)\Delta A^{k,j}(T^k_{q})
\end{eqnarray*}
for $q\ge 1$ and $1\le i\le n$, where

$$\mathbb{X}^{k,u^k}_{T^k_{q}}(t):= \sum_{\ell=0}^{q-1}\mathbb{X}^{k,u^k}(T^k_\ell) 1\!\!1_{\{T^k_\ell\le t < T^k_{\ell+1}\}} + \mathbb{X}^{k,u^k}(T^k_{q})1\!\!1_{\{T^k_{q}\le t \}}; 0\le t\le T.$$

By construction the following relation holds true

$$
\mathbb{X}^{k,u^k}_{S^{k,j}_q}=\left\{
\begin{array}{rl}
\mathbb{X}^{k,u^k}_{T^k_{q-1}}; & \hbox{if} \ S^{k,j}_q =T^k_{q-1} \\
\mathbb{X}^{k,u^k}_{S^{k,j}_{q-1}};& \hbox{if} \ S^{k,j}_q< T^k_{q-1},  \\
\end{array}
\right.
$$
for $q\ge2$. The controlled structure is then naturally defined by

\begin{equation}\label{xksde}
X^{k,u^k}(t):= \mathbb{X}^{k,u^k}(t\wedge T^k_{e(k,T)}),~\text{where}~\mathbb{X}^{k,u^k}(t) = \sum_{m=0}^\infty \mathbb{X}^{k,u^k}(T^k_m)1\!\!1_{\{T^k_m \le t < T^k_{m+1}\}}
\end{equation}
for $t\in [0,T]$ and $u^k\in U^{k,e(k,T)}_0$.

\


\noindent \textbullet~\textbf{Pathwise description:}
In the sequel, we make use of the information set described in (\ref{pfunction}), (\ref{tknfunction}) and (\ref{tkmodfunction}). Let us fix a final step $q\ge 1$ and an information set $\textbf{b}^k_{q} = \big( s^k_1, \tilde{i}^k_1,\ldots, s^k_q,\tilde{i}^k_q \big)$. For a given information set

$$\Big\{ t^{k,\lambda}_{\mathbb{j}_{\lambda}}(\pi_\ell(\textbf{b}^k_q)); 1\le \ell \le q,1\le \lambda\le d \Big\},$$
we observe

$$t^k_1 = \min_{\substack {1\le \lambda\le d}}\Big\{t^{k,\lambda}_{\mathbb{j}_{\lambda}}(\pi_1(\textbf{b}^k_{q})) \Big\},\quad t^k_r = \min_{\substack {1\le \ell\le r\\ 1\le \lambda\le d}}\Big\{t^{k,\lambda}_{\mathbb{j}_{\lambda}}(\pi_{\ell}(\textbf{b}^k_{q})); t^{k,\lambda}_{\mathbb{j}_{\lambda}}(\pi_\ell(\textbf{b}^k_{q})) > t^k_{r-1}\Big\}; 2\le r \le q.$$
Let us define

$$\vartheta^{k,\lambda}_r:=\max_{\substack {\ell\le r-1}}\Big\{t^{k,\lambda}_{\mathbb{j}_{\lambda}}(\pi_{\ell}(\textbf{b}^k_{q}));  t^{k,\lambda}_{\mathbb{j}_{\lambda}}(\pi_\ell(\textbf{b}^k_q))\le t^k_{r-1}\Big\}; 2\le r\le q,$$
where $\vartheta^{k,\lambda}_1:=0$ for $1\le\lambda\le d$.

In the sequel, for each subset $\{a^k_0, \ldots, a^k_{q}\}\subset \mathbb{A}$, we define

$$a^k_{\wp_j}(\textbf{b}^k_\ell):=a^k_{\wp_j(\textbf{b}^k_\ell)};~\mathbf{b}^k_\ell\in \mathbb{S}^k_\ell,\ell\ge 1.$$
For a given $\mathbf{o}^k_{q-1} = \big((a^k_0,s^k_1,\tilde{i}^k_1),\ldots, (a^k_{q-2},s^k_{q-1},\tilde{i}^k_{q-1}) \big)$, we denote $\mathbf{b}^k_{q-1} = (s^k_1, \tilde{i}^k_1,\ldots, s^k_{q-1},\tilde{i}^k_{q-1})$ and we define $h^k_q$ by forward induction as follows: We set $h^k_0:=x$ and $\bar{\gamma}^k_0:=x$ (and hence $\bar{\gamma}_{\vartheta^{k,j}_1}=x,1\le j\le d$) is the constant function over $[0,T]$. For $\mathbf{o}^k_q = (\mathbf{o}^k_{q-1}, a_{q-1},s^k_q,\tilde{i}^k_q)\in \mathbb{H}^{k,q}$, we set

\begin{eqnarray*}
h^{i,k}_{q}(\textbf{o}^k_{q})&:=&h^{i,k}_{q-1}(\mathbf{o}^k_{q-1}) + \alpha^i\big(t^k_{q-1},\bar{\gamma}^k_{q-1}(\textbf{o}^k_{q-1}),a^k_{q-1}\big)s^k_{q}\\
& &\\
&+& \sum_{j=1}^d \sigma^{ij}\Big(\vartheta^{k,j}_{q}(\textbf{b}^k_{q-1}),\bar{\gamma}^k_{q-1,\vartheta^{k,j}_{q}}(\textbf{o}^k_{q-1}),a^k_{\wp_j}(\textbf{b}^k_{q-1})\Big)
\epsilon_k\tilde{i}^{k,j}_{q},
\end{eqnarray*}
for $1\le i\le n$, where $\bar{\gamma}^k_{q-1,\vartheta^{k,j}_{q}}(\textbf{o}^k_{q-1}): = \bar{\gamma}^k_{q-1}(\textbf{o}^k_{q-1})(\cdot\wedge \vartheta^{k,j}_{q})$ and

$$\bar{\gamma}^k_{q-1}(\textbf{o}^k_{q-1})(t):= \sum_{\ell=0}^{q-2}h^k_{\ell}(\pi_{\ell}(\textbf{o}^k_{q-1})) 1\!\!1_{\{t^k_\ell\le t < t^k_{\ell+1}\}} + h^k_{q-1}
(\textbf{o}^k_{q-1})1\!\!1_{\{t^k_{q-1}\le t \}}; 0\le t\le T.$$

By construction the following relation holds true

$$
\bar{\gamma}^k_{q-1,\vartheta^{k,j}_{q}}(\textbf{o}^k_{q-1})=\left\{
\begin{array}{rl}
\bar{\gamma}^k_{q-1}(\textbf{o}^k_{q-1}); & \hbox{if} \ \vartheta^{k,j}_q =t^k_{q-1} \\
\bar{\gamma}^k_{q-2,\vartheta^{k,j}_{q-1}}(\pi_{q-2}(\textbf{o}^k_{q-1}));& \hbox{if} \ \vartheta^{k,j}_q < t^k_{q-1}  \\
\end{array}
\right.
$$
for $q\ge 2$ and $1\le j\le d$. We then define

$$
\bar{\gamma}^k(\textbf{o}^k_{\infty})(t)= \sum_{n=0}^\infty h^k_{n}(\textbf{o}^k_n)1\!\!1_{\{t^k_n \le t< t^k_{n+1}\}}
$$
for $\textbf{o}^k_{\infty}\in \mathbb{H}^{k,\infty}$ and we set
\begin{equation}\label{gammapath}
\gamma^k_{e(k,T)}(\textbf{o}^k_{e(k,T)})(t)= \bar{\gamma}^k(\textbf{o}^k_{\infty})(t\wedge t^k_{e(k,T)}); 0\le t\le T.
\end{equation}
By construction, $\gamma^k_{e(k,T)}\Big(\Xi^{k,g^k}_{e(k,T)}(\mathcal{A}^k_{e(k,T)}(\omega))\Big)(t) = X^{k,u^k}(t,\omega)$ for a.a $\omega$ and for each $t\in [0,T]$ where $(g^k_{\ell})_{\ell=0}^{e(k,T)}$ is the list of functions associated with $u^k$.

\

\textbullet~\textbf{Checking that} $\big((X^k)_{k\ge 1},\mathscr{D}\big)$ \textbf{is a controlled structure for the SDE}. Let us now check that $\big((X^k)_{k\ge1},\mathscr{D}\big)$ satisfies the assumptions given in Theorem \ref{VALUEconv}. In the sequel, we re going to fix a control $\eta\in U_0$ and a sequence $\eta^k\in U^{k,e(k,T)}_0; k\ge 1$. It is necessary to introduce the following objects:

$$S^{k,j,+}_n:=\min\{T^k_p; T^k_p > S^{k,j}_n\},~\bar{t}_k:=\sum_{n=0}^\infty T^k_n 1\!\!1_{\{T^{k}_n \le t <T^{k}_{n+1}\}}$$
$$\bar{t}^k_j:=\sum_{n=1}^\infty T^{k,j}_{n-1}1\!\!1_{\{T^{k,j}_{n-1}\le t < T^{k,j}_{n}\}},$$

$$\widetilde{\Sigma}^{ij,k,\eta^k}(t):=0 1\!\!1_{\{t=0\}} + \sum_{\ell=1}^\infty \sigma^{ij}\big(S^{k,j}_{\ell}, \mathbb{X}^{k,\eta^k}_{S^{k,j}_{\ell}},\eta^k(S^{k,j,+}_{\ell})\big)1\!\!1_{\{T^{k}_{\ell-1} < t \le T^{k}_\ell\}},
$$
and
$$
\Sigma^{ij,k,\eta^k}(t):= 0 1\!\!1_{\{t=0\}} + \sum_{n=1}^\infty \sigma^{ij}\big(T^{k,j}_{n-1}, \mathbb{X}^{k,\eta^k}_{T^{k,j}_{n-1}},\eta^{k}(T^{k,j}_{n-1}+)\big)1\!\!1_{\{T^{k,j}_{n-1} < t \le T^{k,j}_n\}},
$$
for $0\le t\le T, 1\le i\le n, 1\le j\le d$.
We define

\begin{equation}\label{Xhatmult}
\widehat{\mathbb{X}}^{i,k,\eta^k}(t):=x^i_0 + \int_0^t\alpha^i(\bar{s}_k, \mathbb{X}^{k,\eta^k}_{\bar{s}_k}, \eta^k(\bar{s}_k))ds + \sum_{j=1}^d\int_0^t\widetilde{\Sigma}^{ij,k,\eta^k}(s)dA^{k,j}(s)
\end{equation}
for $0\le t\le T, 1\le i\le n.$ The differential $dA^{k,j}$ in (\ref{Xhatmult}) is interpreted in the Lebesgue-Stieljtes sense. One should notice that

\begin{equation}\label{r1mult}
\mathbb{X}^{k,\eta^k}(t) =\mathbb{X}^{k,\eta^k}(\bar{t}_k) = \widehat{\mathbb{X}}^{k,\eta^k}(\bar{t}_k),~\mathbb{X}^{k,\eta^k}_t = \mathbb{X}^{k,\eta^k}_{\bar{t}_k},
\end{equation}
for every $t\in[0,T]$. To keep notation simple, we set $\|f\|_{\infty} = \sup_{0\le t\le T}|f(t)|$. For a given $1\le i\le n$, the idea is to analyse

$$\mathbb{E}\|\mathbb{X}^{i,k,\eta^k}_T- X^{i,\eta}_T\|^2_\infty\le 2\mathbb{E}\|\mathbb{X}^{i,k,\eta^k}_T- \widehat{\mathbb{X}}^{i,k,\eta^k}_T\|^2_\infty + 2\mathbb{E}\|\widehat{\mathbb{X}}^{i,k,\eta^k}_T- X^{i,\eta}_T\|^2_\infty. $$
The following remark is very simple but very useful to our argument. Recall that $N^{k,j}(t) = \max\{n; T^{k,j}_n\le t\};t\ge 0.$

\begin{lemma}\label{basic1mult}
For every $t\ge 0$ and $1\le i\le n, 1\le j\le d$,

$$
\int_0^t\Sigma^{ij,k,\eta^k}(s)dB^j(s) = \int_0^t\widetilde{\Sigma}^{ij,k,\eta^k}(s)dA^{k,j}(s)$$
$$+\sigma^{ij}\Big( T^{k,j}_{N^{k,j}(t)},\mathbb{X}^{k,\eta^k}_{T^{k,j}_{N^{k,j}(t)}}, \eta^{k}\big(T^{k,j}_{N^{k,j}(t)} + \big) \Big) \big( B^j(t)-B^j(T^{k,j}_{N^{k,j}(t)})\big)~a.s$$
where $dB^j$ is the It\^o integral and $dA^{k,j}$ is the Lebesgue Stieltjes integral.
\end{lemma}
Let us now present a couple of lemmas towards the final estimate. In the remainder of this section, $C$ is a constant which may defer from line to line in the proofs of the lemmas.
\begin{lemma}\label{passo1mult}

$$\sup_{k\ge 1}\mathbb{E}\|\mathbb{X}^{k,\eta^k}_T\|^p_\infty < \infty \quad \forall p> 1.$$
\end{lemma}
\begin{proof}
We fix $1\le i\le n$. At first, it is important to notice that

$$\int_0^T\mathbb{E}\|\mathbb{X}^{i,k,\eta^k}_{\bar{s}_k}\|^p_\infty ds < \infty$$
for every $k\ge 1$ and $p>1$. From Assumption \textbf{(C1)}, there exists a constant $C$ such that

\begin{equation}\label{r2mult}
|\alpha^i(t,\omega,a)| + |\sigma^{ij}(t,\omega,a)|\le C(1+ \|\omega_T\|_\infty)
\end{equation}
for every $(t,\omega,a)\in [0,T]\times \Omega\times \mathbb{A}$ and $1\le i\le n,1\le j\le d$, where $C$ only depends on $T, \alpha(0,0,0)$, $\sigma(0,0,0)$ and the compact set $\mathbb{A}$. Identity (\ref{r1mult}) and Lemma \ref{basic1mult} yield

$$\mathbb{X}^{i,k,\eta^k}(t) = x_0 + \int_0^{\bar{t}_k}\alpha^i(\bar{s}_k,\mathbb{X}^{k,\eta^k}_{\bar{s}_k},\eta^k(\bar{s}_k))ds + \sum_{j=1}^d\int_0^{\bar{t}^k_{j}}\Sigma^{ij,k,\eta^k}(s)dB^j(s).$$

By applying Jensen's ineguality and using (\ref{r2mult}), we get

\begin{equation}\label{r3mult}
\mathbb{E}\sup_{0\le t\le T}\Bigg|\int_0^{\bar{t}_k}\alpha^i(\bar{s}_k,\mathbb{X}^{k,\eta^k}_{\bar{s}_k},\eta^k(\bar{s}_k))ds\Bigg|^p\le TC\Bigg(1+\int_0^T\mathbb{E}\|\mathbb{X}^{k,\eta^k}_{\bar{s}_k}\|^p_\infty ds\Bigg).
\end{equation}
Burkholder-Davis-Gundy and Jensen inequalities jointly with (\ref{r2mult}) yield

$$
\mathbb{E}\sup_{0\le t\le T}\Bigg|\int_0^{t}\Sigma^{ij,k,\eta^k}(s)dB^j(s)\Bigg|^p\le\mathbb{E}\Bigg(\int_0^T|\Sigma^{ij,k,\eta^k}(s)|^2ds\Bigg)^{\frac{p}{2}}
$$
$$\le C\mathbb{E}\int_0^T|\Sigma^{ij,k,\eta^k}(s)|^pds\le C \Big(1+\int_0^T\mathbb{E}\|\mathbb{X}^{k,\eta^k}_{\bar{s}^k_{j}}\|^p_\infty ds\Big).$$

Summing up the above estimates, we have

$$\mathbb{E}\|\mathbb{X}^{k,\eta^k}_T\|^p_\infty\le C\Bigg(1+ \int_0^T\mathbb{E}\|\mathbb{X}^{k,\eta^k}_{s}\|^p_\infty ds \Bigg).$$
Grownall's inequality allows us to conclude the result.
\end{proof}

\begin{lemma}\label{passo2mult}
$$\mathbb{E}\|\widehat{\mathbb{X}}^{k,\eta^k}_T- X^\eta_T\|^2_\infty \le \Bigg\{ \mathbb{E}\vee_{n=1}^\infty|\Delta T^k_n|1\!\!1_{\{T^k_n\le T\}}$$
$$+\mathbb{E}\int_0^{T}\|\eta^k(s) - \eta(s)\|^2_{\mathbb{R}^m}ds + \sum_{j=1}^d\mathbb{E}\int_0^{T} \|\eta^k(\bar{\textbf{s}}^k_j+) - \eta(s)\|^2_{\mathbb{R}^m}ds + \epsilon^2_k + \int_0^T\mathbb{E}\|\mathbb{X}^{k,\eta^k}_{s}- X^{\eta}_s\|^2_\infty ds \Bigg\}.$$
\end{lemma}
\begin{proof}
Let us fix $1\le i\le n$. By the very definition,

$$\widehat{\mathbb{X}}^{i,k,\eta^k}(t) - X^{i,\eta}(t) = \int_0^t\Big[\alpha^i(\bar{s}_k,\mathbb{X}^{k,\eta^k}_{\bar{s}_k},\eta^k(\bar{s}_k)) - \alpha^i(s,X^{\eta}_{s},\eta(s))\big]ds$$
$$+\sum_{j=1}^d\int_0^t\widetilde{\Sigma}^{ij,k,\eta^k}(s)dA^{k,j}(s) - \sum_{j=1}^d\int_0^t\sigma^{ij}(s,X^{\eta}_s,\eta(s))dB^j(s).$$
Lemma \ref{basic1mult} allows us to write

$$\widehat{\mathbb{X}}^{i,k,\eta^k}(t) - X^{\eta}(t) = \int_0^t\Big[\alpha^i(\bar{s}_k,\mathbb{X}^{k,\eta^k}_{\bar{s}_k},\eta^k(\bar{s}_k)) - \alpha^i(s,X^{\eta}_{s},\eta(s))\big]ds$$
$$+\sum_{j=1}^d\int_0^t\Big[\Sigma^{ij,k,\eta^k}(s) - \sigma^{ij}(s,X^{\eta}_s,\eta(s))\Big]dB^j(s) - \sum_{j=1}^d\sigma^{ij}\Big( T^{k,j}_{N^{k,j}(t)},\mathbb{X}^{k,\eta^k}_{T^{k,j}_{N^{k,j}(t)}}, \eta^k(T^{k,j}_{N^{k,j}(t)}+) \Big)$$
$$\times \big( B^j(t) - B^j(T^{k.j}_{N^{k,j}(t)})\big)= I^{k,i}_1(t) + I^{k,i}_2(t) + I^{k,i}_3(t).$$

\textbf{Analysis of}~$I^{k,i}_1$: Assumption \textbf{(C1)} yields

$$\mathbb{E}\sup_{0\le t\le T}|I^{k,i}_1(t)|^2\le C \Big\{\mathbb{E}\vee_{n=1}^\infty\Delta T^k_n1\!\!1_{\{T^k_n\le T\}} + \int_0^T\mathbb{E}\|\mathbb{X}^{k,\eta^k}_s - X^{\eta}_s\|^2_\infty ds$$
 $$+ \mathbb{E}\int_0^{T} \|\eta^k(\bar{s}_k)- \eta(s)\|^2_{\mathbb{R}^m}ds\Big\}.$$

\textbf{Analysis of}~$I^{k,i}_2$: By using Burkholder-Davis-Gundy's inequality, we have

$$\mathbb{E}\sup_{0\le t\le T}|I^{k,i}_2(t)|^2\le C\sum_{j=1}^d\mathbb{E}\int_0^T \big|\Sigma^{ij,k,\eta^k}(s) - \sigma^{ij}(s,X_s^{\eta},\eta(s))   \big|^2ds.$$
Assumption \textbf{(C1)} yields

$$
\Bigg|\Big[\Sigma^{ij,k,\eta^k}(s) - \sigma^{ij}(s,X^{\eta}_{s},\eta(s))\Big]\Bigg|=
\Bigg|\Big[\sigma^{ij}\Big(\bar{s}^k_j,\mathbb{X}^{k,\eta^k}_{\bar{s}^k_j},\eta^{k}(\bar{s}^k_j+)\Big) - \sigma^{ij}(s,X^{\eta}_{s},\eta(s))\Big]\Bigg|$$
$$\le C|\bar{s}^k_j - s|^{1/2} + C\sup_{0\le \ell\le T}\|\mathbb{X}^{k,\eta^k}(\ell\wedge \bar{s}^k_j) - X^{\eta}(\ell\wedge s)\|_{\mathbb{R}^n}+\|\eta^{k}(\bar{s}^k_j+) - \eta(s)\|_{\mathbb{R}^m}
$$
so that

$$\mathbb{E}\sup_{0\le t\le T}|I^{k,i}_2(t)|^2\le C\Big\{\mathbb{E}\vee_{n=1}^\infty\Delta T^k_n 1\!\!1_{\{T^k_n\le T\}}+ \int_0^T\mathbb{E}\|\mathbb{X}^{k,\eta^k}_s - X^{\eta}_s\|^2_\infty ds$$
 $$+ \sum_{j=1}^d\mathbb{E}\int_0^{T} \|\eta^k(\bar{\textbf{s}}^k_j+) - \eta(s)\|^2_{\mathbb{R}^n}ds\Big\}.$$

\textbf{Analysis of}~$I^{k,i}_3$: The estimate (\ref{r2mult}), Lemma~\ref{passo1mult} and the fact that $\big|B^j(t) - B^j(T^{k,j}_{N^{k,j}(t)})\big|\le \epsilon_k~a.s$ for every $t\ge 0$ yield

$$\mathbb{E}\sup_{0\le t\le T}|I^{k,i}_3(t)|^2\le C\epsilon_k^2.$$
Summing up the above estimates, we conclude the proof.
\end{proof}

\begin{lemma}\label{passo3mult}
There exists a constant $C$ which only depends $T,p$ and $\alpha$ such that
$$\mathbb{E}\|\mathbb{X}^{k,\eta^k}_T- \widehat{\mathbb{X}}^{k,\eta^k}_T\|^2_\infty \le C \big(\mathbb{E}\big|\vee_{n=1}^\infty \Delta T^k_n\big|^p1\!\!1_{\{T^k_n\le T\}}\big)^{\frac{1}{p}}$$
for every $p>1$.
\end{lemma}
\begin{proof}
The idea is to use (\ref{r2mult}) and Lemma \ref{passo1mult}. We know that $\widehat{\mathbb{X}}^{k,\eta^k}(\bar{t}_k) = \mathbb{X}^{k,\eta^k}(t); t\ge 0$. We fix $1\le i\le n$ and we see that

$$| \widehat{\mathbb{X}}^{i,k,\eta^k}(t) - \mathbb{X}^{i,k,\eta^k}(t)| = |\widehat{\mathbb{X}}^{i,k,\eta^k}(t) - \widehat{\mathbb{X}}^{i,k,\eta^k}(\bar{t}_k)|$$
$$\le \int_{\bar{t}_k}^t |\alpha^i(\bar{s}_k, \mathbb{X}^{k,\eta^k}_{\bar{s}_k},\eta^k(\bar{s}_k))|ds = \frac{t-\bar{t}_k}{t-\bar{t}_k} \int_{\bar{t}_k}^t |\alpha^i(\bar{s}_k, \mathbb{X}^{k,\eta^k}_{\bar{s}_k},\eta^k(\bar{s}_k))|ds$$
because $t-\bar{t}_k >0$~a.s for every $t>0$ and $\sum_{j=1}^d\int_0^t\widetilde{\Sigma}^{ij,k,\eta^k}(s)dA^{k,j}(s) =\int_0^{\bar{t}_k}\widetilde{\Sigma}^{ij,k,\eta^k}(s)dA^{k,j}(s) ~a.s$ for every $t\ge 0$. Therefore, (\ref{r2mult}) and Jensen's inequality yield

$$\|\mathbb{X}^{i,k,\eta^k}_T- \widehat{\mathbb{X}}^{i,k,\eta^k}_T\|^2_\infty \le C\big(1+ \|\mathbb{X}^{k,\eta^k}_T\|^2_\infty\big)\times \vee_{n=1}^\infty|\Delta T^k_n|1\!\!1_{\{T^k_n\le T\}}~a.s.$$
By using Lemma \ref{passo1mult} and H\"{o}lder inequality, there exists a constant $C$ which only depends on $\alpha,p$ and $T$ such that

$$\mathbb{E}\|\mathbb{X}^{k,\eta^k}_T- \widehat{\mathbb{X}}^{k,\eta^k}_T\|^2_\infty \le C \Big(\mathbb{E}\Big|\vee_{n=1}^\infty|\Delta T^k_n|1\!\!1_{\{T^k_n\le T\}}\Big|^p\Big)^{\frac{1}{p}}$$
for every each $p>1$.
\end{proof}

Summing up Lemmas \ref{passo2mult} and \ref{passo3mult}, we arrive at the following result: There exists a constant $C$ which depends on $\alpha, \sigma, T,p$ such that
$$\mathbb{E}\|\mathbb{X}^{k,\eta^k}_T - X^{\eta}_T\|^2_\infty\le C\Bigg\{\big(\mathbb{E}\big|\vee_{n=1}^\infty \Delta T^k_n\big|^p1\!\!1_{\{T^k_n\le T\}}\big)^{\frac{1}{p}}$$
$$+ \mathbb{E}\int_0^{T}\|\eta^k(s) - \eta(s)\|^2_{\mathbb{R}^m}ds$$
$$\sum_{j=1}^d \mathbb{E}\int_0^{T}\|\eta^k(\bar{s}^k_j+) - \eta(s)\|^2_{\mathbb{R}^m}ds +\epsilon_k^2 + \int_0^T\mathbb{E}\|\mathbb{X}^{k,\eta^k}_s - X^{\eta}_s\|^2_\infty ds\Bigg\},$$
for every $k\ge 1$. By using Grownall's inequality, we then have

$$
\mathbb{E}\|\mathbb{X}^{k,\eta^k}_T - X^{\eta}_T\|^2_\infty\le (C+1)\Bigg\{\big(\mathbb{E}\big|\vee_{n=1}^\infty \Delta T^k_n\big|^p1\!\!1_{\{T^k_n\le T\}}\big)^{\frac{1}{p}}$$
$$+ \mathbb{E}\int_0^{T}\|\eta^k(s) - \eta(s)\|^2_{\mathbb{R}^m}ds + \sum_{j=1}^d \mathbb{E}\int_0^{T}\|\eta^k(\bar{s}^k_j+) - \eta(s)\|^2_{\mathbb{R}^m}ds  +\epsilon^2_k\Bigg\}.
$$
Lemma 2.2 in \cite{LEAO_OHASHI2017.1} yields the existence of a constant $C$ which depends on $T$, $\beta\in (0,1)$ and $p > 1$ such that

\begin{equation}\label{meshrate}
\mathbb{E}\big|\vee_{n=1}^\infty \Delta T^k_n|^p\mathds{1}_{\{T^k_n\le T\}}\le C\max_{1\le j\le d}\mathbb{E}\big|\vee_{n=1}^\infty\Delta T^{k,j}_n|^p 1\!\!1_{\{T^{k}_n \le T\}}\le C\epsilon_k^{2p} \lceil\epsilon^{-2}_kT\rceil^{1-\beta}
\end{equation}
for every $k\ge 1$.

Let us denote

$$\gamma_{k,p,\beta} := \Bigg\{\epsilon_k^2\lceil\epsilon^{-2}_kT\rceil^{\frac{1-\beta}{p}} +\epsilon^2_k\Bigg\}$$
and

\begin{equation}\label{etakj}
\eta^{k,j}(s):=\eta^k(\bar{s}^k_j); s\ge 0,
\end{equation}
for $k\ge 1$, $p>1$ and $\beta\in (0,1)$. We observe that $\eta^{k,j}(s) = \eta^{k}(\bar{s}^k_j+)$ a.s $\mathbb{P}\times Leb$. We then arrive at the following estimate:

\begin{lemma}\label{lemmaXmathk}
Assume the coefficients of the SDE (\ref{pdsdeBM}) satisfy Assumption \textbf{(C1)}. Let $(\eta^k,\eta)\in U^{k,e(k,T)}_0\times U_0$ be an arbitrary pair of controls. Then, there exists a constant $C>0$ which depends on $\alpha, \sigma, T$, $\beta\in (0,1), p > 1$ and $\bar{a}$ such that

\begin{equation}\label{caligXest}
\mathbb{E}\|\mathbb{X}^{k,\eta^k}_T - X^{\eta}_T\|^2_\infty\le (C+1)\Bigg\{\gamma_{k,p,\beta} + \|\eta^k - \eta\|^2_{L^2_a(\mathbb{P}\times Leb)} + \sum_{j=1}^d \|\eta^{k,j} - \eta\|^2_{L^2_a(\mathbb{P}\times Leb)} \Bigg\}
\end{equation}
for every $k\ge 1$.
\end{lemma}

\begin{remark}
Unless the underlying filtration $\mathbb{F}$ is generated by a one-dimensional Brownian motion, we observe that, in general, $\eta^{k,j}$ defined in (\ref{etakj}) is not equal to $\eta^k$ for a given control $\eta^k\in U^{k,e(k,T)}_0$. This remainder term appears due to the fact that $\Delta A^{k,j}(T^k_n)\neq 0~a.s$ if, and only if, $T^k_n$ is realized by the $j$-th component of the Brownian motion. This particular feature makes the analysis trickier than the usual Euler-Maruyama scheme based on deterministic partitions if $d> 1$.
\end{remark}

\begin{proposition}\label{Xkfinalestprop}
Assume the coefficients of the SDE (\ref{pdsdeBM}) satisfy Assumption \textbf{(C1)}. Let $(\eta^k,\eta)\in U^{k,e(k,T)}_0\times U_0$ be an arbitrary pair of controls. Then, there exists a constant $C>0$ which depends on $\alpha, \sigma, T$, $\beta\in (0,1), p > 1$ and $\bar{a}$ such that

\begin{equation}\label{Xkfinalest}
\mathbb{E}\|X^{k,\eta^k}_T - X^{\eta}_T\|^2_\infty\le C\Bigg\{\gamma_{k,p,\beta} + \sum_{j=1}^d \|\eta^{k,j} - \eta\|^2_{L^2_a(\mathbb{P}\times Leb)} + \|\eta^k - \eta\|^2_{L^2_a(\mathbb{P}\times Leb)}
\end{equation}
$$
 +\|T-T^k_{e(k,T)}\|_{L^2(\mathbb{P})} \Bigg\}
$$
for every $k\ge 1$.
\end{proposition}
\begin{proof}
By definition,

$$X^{k,\eta^k}(t) = \mathbb{X}^{k,\eta^k}(t\wedge T^k_{e(k,T)}); 0\le t\le T,$$
where $t\wedge T^k_{e(k,T)}\le t \le T$, then triangle inequality yields

$$\Big \|X^{k,\eta^k}_T - X^{\eta}_T\Big\|^2_\infty \le C\big\|\mathbb{X}^{k,\eta^k}_T- X^{\eta}_T\big\|^2_{\infty} + C \sup_{0\le t\le T}\big\|X^{\eta}(t\wedge T^k_{e(k,T)}) - X^{\eta}(t)\big\|^2_{\mathbb{R}^n}.$$
Therefore, in view of the estimate (\ref{caligXest}), we only need to estimate

\begin{equation}\label{need}
\mathbb{E}\sup_{0\le t\le T}\big\|X^{\eta}(t\wedge T^k_{e(k,T)}) - X^{\eta}(t)\big\|^2_{\mathbb{R}^n}.
\end{equation}
For a given $1\le i\le n$,

$$
\big|X^{i,\eta}(t\wedge T^k_{e(k,T)}) - X^{i,\eta}(t)\big|^2\le C\Bigg|\int_{t\wedge T^k_{e(k,T)}}^t\alpha^i(s,X^{\eta}_s,\eta(s))ds\Bigg|^2$$
$$+ C\sum_{j=1}^d\Bigg|\int_{t\wedge T^k_{e(k,T)}}^t\sigma^{ij}(s,X^{\eta}_s,\eta(s))dB^j(s)\Bigg|^2.
$$


By using H\"{o}lder's inequality, (\ref{r2mult}) and (\ref{integraPDSDE}), we get the existence of a constant $C$ such that
\begin{eqnarray*}
\mathbb{E}\sup_{0\le t\le T}\Bigg|\int_{t\wedge T^k_{e(k,T)}}^t\alpha^i(s,X^{\eta}_s,\eta(s))ds\Bigg|^2&\le&\mathbb{E}\Big(C(1+\|X^{\eta}_T\|^2_\infty)\sup_{0\le t\le T}|t-t\wedge T^k_{e(k,T)}|\Big)\\
& &\\
&\le& C \big(\mathbb{E}|T-T^k_{e(k,T)}|^{2}\big)^{1/2}.
\end{eqnarray*}
By applying Burkholder-Davis-Gundy's inequality and the same argument as above, we get

\begin{eqnarray*}
\mathbb{E}\sup_{0\le t\le T}\Bigg|\int_{t\wedge T^k_{e(k,T)}}^t\sigma^{ij}(s,X^{\eta}_s,\eta(s))dB^j(s)\Bigg|^2&\le&\mathbb{E}\Big(C(1+\|X^{\eta}_T\|^2_\infty)\sup_{0\le t\le T}|t-t\wedge T^k_{e(k,T)}|\Big)\\
& &\\
&\le& C \big(\mathbb{E}|T-T^k_{e(k,T)}|^{2}\big)^{1/2}.
\end{eqnarray*}
This concludes the proof.
\end{proof}

We are now able to prove that $X^k$ satisfies the fundamental condition (\ref{keyassepsilon}) given in Theorem \ref{VALUEconv}. For this purpose, let us introduce some objects: $G^{k,j}_{n-1} := \min\{T^k_\ell; T^k_\ell> T^{k,j}_{n-1}\}; n\ge 1$,

$$D^k_j: = \bigcup_{n=1}^{\lceil\epsilon^{-2}_kT\rceil}\Big\{(\omega,s); T^{k,j}_{n-1}(\omega)\le s < G^{k,j}_{n-1}(\omega)\Big\},$$
for $j=1,\ldots, d, k\ge 1$.

\begin{lemma}\label{KEYLEMMAF}
Let $u^k\in U^{k,e(k,T)}_0$ be an arbitrary sequence of admissible controls. Then,

\begin{equation}\label{q1}
\mathbb{E}\int_0^T\|u^k(\bar{s}^k_j) - u^k(s)\|^2_{\mathbb{R}^m}ds\le \bar{a}^2\mathbb{E}|T^{k,j}_{\lceil\epsilon^{-2}_kT\rceil}\wedge T-T| + \bar{a}^2\mathbb{E}\int_0^T\big(\mathds{1}_{\Omega\times [0,T]} - \mathds{1}_{D^k_j}\big)ds
\end{equation}
\end{lemma}

\begin{proof}
In order to alleviate notation and without any loss of generality, we consider $T=1$, $\epsilon_k = 2^{-k}$ and $d=2$. Let us fix a coordinate $1\le j\le 2$. At first, we observe that

\begin{equation}\label{firOB}
D^k_j \subset D^k_j(u^k):=\Big\{(\omega,s)\in [[0,T^{k,j}_{2^{2k}}]]; u^k\big(\omega,\bar{s}^k_j(\omega)\big)=u^k\big(\omega,s\big)\Big\}\subset \big[\big[0,T^{k,j}_{2^{2k}}\big]\big]
\end{equation}
Recall that we use the notation $\big[\big[0,T^{k,j}_{2^{2k}}\big]\big] = \{(\omega,t); 0 \le t \le T^{k,j}_{2^{2k}}(\omega)\}$.
Then,

$$\mathds{1}_{D^k_j}\le \mathds{1}_{\big[\big[0,T^{k,j}_{2^{2k}}\big]\big]}~a.s~-\mathbb{P}\times Leb,$$
for every $k\ge 1$. Since $T^{k,j}_{2^{2k}}\rightarrow 1$ a.s, then

\begin{equation}\label{lims}
\limsup_{k\rightarrow+\infty}\mathds{1}_{D^k_j}\le \mathds{1}_{\Omega\times [0,1]}~a.s~-\mathbb{P}\times Leb.
\end{equation}
We claim that any $(\omega,t)\in \Omega\times (0,1)$ belongs to $D^k_j$ for infinitely many $k\ge 1$ with the possible exception of a finite number of them. In fact, it is known (see Lemma 2.2 in \cite{koshnevisan}) that

\begin{equation}\label{q3}
\max_{1\le \ell \le 2}\sup_{0\le s\le1}|T^{k,\ell}_{\lceil2^{2k}s\rceil}-s|\rightarrow 0
\end{equation}
holds true a.s say in a set $\Omega^*$ of full probability. Let us fix $(\omega,t)\in \Omega^*\times (0,1)$ where $\mathbb{P}(\Omega^*)=1$. Let us take a sequence of positive numbers $\{r_n; n\ge 1\}$ such that $r_n < t $ and $r_n\uparrow t$ as $n\rightarrow +\infty$. By writing

$$|T^{k,j}_{\lceil 2^{2k}r_n\rceil}(\omega) - t|\le|T^{k,j}_{\lceil 2^{2k}r_n\rceil}(\omega) - r_n| + |r_n-t|$$
and using the fundamental convergence (\ref{q3}), we observe that for $\epsilon,\eta>0$ with $\epsilon-\eta>0$, there exists $k_0(\omega,\epsilon,\eta)$ such that

$$t - (\epsilon-\eta) < T^{k,j}_{\lceil 2^{2k}r_n\rceil}(\omega)< t + (\epsilon-\eta) $$
for every $k\ge k_0(\omega,\epsilon,\eta)$ and for every $n\ge N$ sufficiently large where $N$ does not depend on $k$ and $\omega$. Now, we observe the important property:

$$\lim_{n\rightarrow+\infty}\lim_{k\rightarrow+\infty}T^{k,j}_{\lceil 2^{2k}(r_n+\epsilon)\rceil}(\omega)=t+\epsilon\Longrightarrow \lim_{n\rightarrow+\infty}\lim_{k\rightarrow+\infty}G^{k,j}_{\lceil 2^{2k}(r_n+\epsilon)\rceil}(\omega)=t+\epsilon$$
so that

$$(t+\epsilon) - \eta < G^{k,j}_{\lceil 2^{2k}(r_n+\epsilon)\rceil}(\omega) <(t + \epsilon)+\eta$$
for every $k\ge k_0(\omega,\epsilon,\eta)$ and for every $n\ge N$.
Therefore,

\begin{equation}\label{q5}
T^{k,j}_{\lceil 2^{2k}r_n\rceil}(\omega) < t+(\epsilon-\eta) <  G^{k,j}_{\lceil 2^{2k}(r_n+\epsilon)\rceil}(\omega)
\end{equation}
for every $k\ge k_0(\omega,\epsilon,\eta)$ and for every $n\ge N$. Since $\lim_{n\rightarrow +\infty}\lim_{k\rightarrow+\infty}T^{k,j}_{\lceil 2^{2k}r_n\rceil}(\omega)=t$, we may assume (take a subsequence if necessary) that

\begin{equation}\label{q6}
T^{k,j}_{\lceil 2^{2k}r_n\rceil}(\omega)< t
\end{equation}
for every $k\ge k_0(\omega,\epsilon,\eta)$ and for every $n\ge N$. From (\ref{q5}) and (\ref{q6}), we then have

\begin{equation}\label{q7}
T^{k,j}_{\lceil 2^{2k}r_n\rceil}(\omega)< t < G^{k,j}_{\lceil 2^{2k}(r_n+\epsilon)\rceil}(\omega)
\end{equation}
for every $k\ge k_0(\omega,\epsilon,\eta)$ and for every $n\ge N$. This shows that

\begin{equation}\label{q8}
\Omega\times [0,1] = \liminf_{k\rightarrow+\infty} D^k_j\quad a.s~-\mathbb{P}\times Leb
\end{equation}
and summing up with (\ref{lims}), we conclude that

\begin{equation}\label{q9}
\lim_{k\rightarrow \infty}\mathds{1}_{D^k_j} =\mathds{1}_{\Omega\times [0,1]}~a.s~-\mathbb{P}\times Leb.
\end{equation}
Now, we observe that

$$\mathbb{E}\int_0^1\|u^k(\bar{s}^k_j) - u^k(s)\|^2_{\mathbb{R}^m}ds =\mathbb{E}\int_0^{1\wedge T^{k,j}_{2^{2k}}}\|u^k(\bar{s}^k_j) - u^k(s)\|^2_{\mathbb{R}^m}ds$$
$$+\mathbb{E}\int_{T^{k,j}_{2^{2k}}\wedge 1}^1\|u^k(\bar{s}^k_j) - u^k(s)\|^2_{\mathbb{R}^m}ds.$$
Obviously, $ \mathbb{E}\int_{T^{k,j}_{2^{2k}}\wedge 1}^1\|u^k(\bar{s}^k_j) - u^k(s)\|^2_{\mathbb{R}^m}ds\le \bar{a}^2\mathbb{E}|T^{k,j}_{2^{2k}}\wedge 1 -1|\rightarrow 0$. Moreover, by (\ref{firOB}) and using (\ref{q9}) jointly with bounded convergence theorem, we conclude that

$$
\mathbb{E}\int_0^{1\wedge T^{k,j}_{2^{2k}}}\|u^k(\bar{s}^k_j) - u^k(s)\|^2_{\mathbb{R}^m}ds \le \mathbb{E}\int_0^{1\wedge T^{k,j}_{2^{2k}}}\|u^k(\bar{s}^k_j) - u^k(s)\|^2_{\mathbb{R}^m}\Big(\mathds{1}_{\Omega\times [0,1]} -\mathds{1}_{D^k_j(u^k)}\Big) ds
$$
$$+ \mathbb{E}\int_0^{1\wedge T^{k,j}_{2^{2k}}}\|u^k(\bar{s}^k_j) - u^k(s)\|^2_{\mathbb{R}^m}\mathds{1}_{D^k_j(u^k)}ds\le \mathbb{E}\int_0^{1\wedge T^{k,j}_{2^{2k}}}\|u^k(\bar{s}^k_j) - u^k(s)\|^2_{\mathbb{R}^m}\Big(\mathds{1}_{\Omega\times [0,1]} -\mathds{1}_{D^k_j}\Big) ds$$
$$
\le\bar{a}^2\mathbb{E}\int_0^1 \Big( \mathds{1}_{\Omega\times [0,1]} -\mathds{1}_{D^k_j} \Big)ds\rightarrow 0~\text{as}~k\rightarrow+\infty.
$$
\end{proof}
As a by product of Lemma \ref{KEYLEMMAF} and (\ref{Xkfinalest}) in Proposition \ref{Xkfinalestprop}, we arrive at the main result of this section: In the sequel, we set

$$\ell_k:=\sum_{j=1}^d \Bigg(\bar{a}^2\mathbb{E}|T^{k,j}_{\lceil\epsilon^{-2}_kT\rceil}\wedge T-T| + \bar{a}^2\mathbb{E}\int_0^T\big(\mathds{1}_{\Omega\times [0,T]} - \mathds{1}_{D^k_j}\big)ds\Bigg); k\ge 1.
$$

\begin{theorem}\label{MAINTHEOREMSDE}
Assume the coefficients of the controlled SDE (\ref{pdsdeBM}) satisfy Assumption \textbf{(C1)}. Then, there exists a constant $C>0$ which depends on $\alpha, \sigma, T,\beta\in (0,1), p > 1$ and $\bar{a}$ such that

$$
\sup_{\phi\in U^{k,e(k,T)}_{e(k,t)}}\mathbb{E}\|X^{k,u^k\otimes_{e(k,t)}\phi}_T - X^{u\otimes_t\phi}_T\|^2_\infty \le C\Bigg\{\gamma_{k,p,\beta} + \ell_k + \|u^k - u\|^2_{L^2_a(\mathbb{P}\times Leb)}$$
$$+ \|T^k_{e(k,t)}\wedge t - T^k_{e(k,t)}\vee t\|_{L^1} + \|T-T^k_{e(k,T)}\|_{L^2} \Bigg\},$$
for every $k\ge 1, t\in [0,T], u\in U^T_0$ and $u^k\in U^{k,e(k,T)}_0$. The controlled imbedded structure $\big((X^k)_{k\ge 1},\mathscr{D}\big)$ satisfies (\ref{keyassepsilon}) and hence, (\ref{STvalueconv}) in Theorem \ref{VALUEconv} and (\ref{nearOC}) in Theorem \ref{VALUEconvcontrol} hold true for the controlled SDE (\ref{pdsdeBM}).
\end{theorem}

\subsection{Optimal control of drifts for path-dependent SDEs driven by fractional Brownian motion}
In this section, we illustrate the abstract results of this paper to a stochastic control problem driven by a path-dependent SDE driven by fractional Brownian motion with $\frac{1}{2} < H < 1$,

$$B_H(t) = \int_0^t \rho_H(t,s)B(s)ds; 0\le t\le T,$$
where
$$\rho_H(t,s):= d'_H\Bigg[\big(H-\frac{1}{2}\big)s^{-H-\frac{1}{2}}\int_s^tu^{H-\frac{1}{2}} (u-s)^{H-\frac{3}{2}}du -s^{-H-\frac{1}{2}}t^{H+\frac{1}{2}}(t-s)^{H-\frac{3}{2}}\Bigg],$$
and $d'_H = (H-1/2)d_H$ for a constant $d_H$. See Hu \cite{Hu} for more details on this representation. We are going to analyze the stochastic control problem

\begin{equation}\label{limsdefbmC}
\sup_{\phi\in U^T_0}\mathbb{E}\big[\xi(X^\phi)\big]
\end{equation}
driven by
\begin{equation}\label{limsdefbm}
dX^u(t) = \alpha(t,X_t,u(t))dt + \sigma dB_H(t)
\end{equation}
where $X(0)=x_0\in \mathbb{R}$, $\sigma$ is a constant, $\alpha$ is a non-anticipative functional satisfying \textbf{(C1)} and $\frac{1}{2}< H < 1$. The path-dependence feature is much more sophisticated than previous example because the lack of Markov property comes from distorting the Brownian motion by the singular kernel $\rho_H$ and from the drift $\alpha$.

Under Assumption \textbf{(C1)}, by a standard fixed point argument, one can show there exists a unique strong solution for (\ref{limsdefbm}) for each $u\in U^T_0$. In the sequel, we denote

$$B^k_{H}(t): = \int_0^t\rho_H(t,s)A^k(s)ds; 0\le t\le T.$$
To get a piecewise constant process, we set

$$W^k_{H}(t) := \sum_{n=0}^\infty B^k_{H}(T^k_n)\mathds{1}_{\{T^k_n\le t < T^k_{n+1}\}}; 0\le t\le T.$$
By Theorem 5.1 in \cite{LEAO_OHASHI2017.2}, for every $0 < \gamma < 2H-2$, there exists a constant $C_{H,T,\gamma}$ independent from $k$ such that

\begin{equation}\label{estFBM}
\mathbb{E}\sup_{0\le t\le T}|W^k_{H}(t) - B_H(t)|\le C_{H,T,\gamma} \epsilon_k^{\gamma}; k\ge 1.
\end{equation}
The controlled structure $\big((X^k)_{k\ge 1},\mathscr{D}\big)$ associated with (\ref{limsdefbm}) is given as follows: Let us fix a control $u^k\in U^k_0$ defined by $g^k_n:\mathbb{S}^n_k\rightarrow\mathbb{R}; n\ge 0$. Let us define

\begin{eqnarray*}
\mathbb{X}^{k,u^k}(T^k_{m})&:=&\mathbb{X}^{k,u^k}(T^k_{m-1}) + \alpha\big(T^k_{m-1},\mathbb{X}^{k,u^k}_{T^k_{m-1}},g^k_{m-1}(\mathcal{A}^k_{m-1})\big)\Delta T^k_{m}\\
& &\\
&+& \sigma\Delta W^k_{H}(T^k_{m}); m\ge 1,
\end{eqnarray*}
and then set
$$X^{k,u^k}(t):=\sum_{\ell=0}^{\infty}\mathbb{X}^{k,u^k}(T^k_\ell) 1\!\!1_{\{T^k_\ell\le t \wedge T^k_{e(k,T)}< T^k_{\ell+1}\}}; 0\le t\le T.$$

\begin{proposition}\label{FBMconSDE}
Assume that $\frac{1}{2} < H < 1$, $\alpha$ satisfies assumption \textbf{(C1)} and take $0 < \gamma < 2H-2$. Let $(\eta^k,\eta)\in U^{k,e(k,T)}_0\times U_0^T$ be an arbitrary pair of controls. Then, there exists a constant $C$ which depends on $T,\alpha, \beta\in (0,1), p> 1$ and $0 < \gamma < 2H-2$ such that

\begin{equation}\label{FBMcon}
\mathbb{E}\|X^{k,\eta^k}_T - X^{\eta}_T\|_\infty\le C \Bigg\{ \epsilon_k^{2}\lceil\epsilon^{-2}_kT\rceil^{\frac{1-\beta}{p}} + \mathbb{E}\int_0^T|\eta^k(s) - \eta(s)|ds + \|T^k_{e(k,T)} - T\|_{L^2(\mathbb{P})} + \epsilon_k^{\gamma}\Bigg\}; k\ge 1.
\end{equation}
In particular, $\big(X^{k},\mathscr{D}\big)$ is a controlled structure for (\ref{limsdefbm}) and it also satisfies assumption (\ref{keyassepsilon}). Hence, Theorems \ref{VALUEconv} and \ref{VALUEconvcontrol} apply to the control problem (\ref{limsdefbmC}).
\end{proposition}
\begin{proof}
By repeating exactly the same steps (with the obvious modification by replacing $A^k$ by $W^k_H$) as in the proof of Proposition 5.2 in \cite{LEAO_OHASHI2017.2} and Lemma \ref{lemmaXmathk}, we can find a constant $C$ (depending on $\alpha, T$ and $\beta$) such that

$$
\mathbb{E}\|\mathbb{X}^{k,\eta^k}_T - X^{\eta}_T\|_\infty\le C \Bigg\{ \epsilon_k^{2}\lceil\epsilon^{-2}_kT\rceil^{\frac{1-\beta}{p}} + \mathbb{E}\int_0^T|\eta^k(s) - \eta(s)|ds
$$
$$+\mathbb{E}\sup_{0\le t\le T}|W^k_{H}(t) - B_H(t)|\Bigg\}$$
for $\beta \in (0,1)$ and $p> 1$. In order to estimate $\|X^{k,\eta^k}_T - X^\eta_T\|_\infty$, we shall proceed in the same way as in (\ref{need}) (but with exponent equals one). The only thing one has to estimate is the quantity $\mathbb{E}\sup_{0\le t\le T}|B_H(t) - B_H(t\wedge T^k_{e(k,T)})|$. For this purpose, take $1-H < \beta < 1/2$, $0< \epsilon < \beta-1+ H$. It is well-known (see e.g the proof of Lemma 1.17.1 in \cite{mishura}) there exists $G_{T,\epsilon,\beta}\in \cap_{q\ge 1}L^q(\mathbb{P})$ and a deterministic constant $C$ such that

$$|B_H(t) - B_H(s)|\le C |t-s|^{(H-\epsilon)}G_{T,\epsilon,\beta}~a.s$$
for every $t,s\in [0,T]$. Therefore,
$$\sup_{0\le t\le T}|B_H(t) - B_H(t\wedge T^k_{e(k,T)})|\le C |T - T^k_{e(k,T)}|^{(H-\epsilon)}G_{T,\epsilon,\beta}~a.s.$$
This estimate together with (\ref{estFBM}) allow us to conclude the proof of (\ref{FBMcon}). Since the system is driven by the one-dimensional Brownian motion, then (\ref{FBMcon}) immediately shows that $\big((X^k)_{k\ge 1},\mathscr{D}\big)$ is a controlled imbedded structure w.r.t (\ref{limsdefbm}) satisfying (\ref{keyassepsilon}). In this case, Theorems \ref{VALUEconv} and \ref{VALUEconvcontrol} can be applied.
\end{proof}
\subsection{Example: Non-Markovian Portfolio Optimization}\label{portfolioOPT}
Let us now illustrate the theory of this article to a classical problem in Mathematical Finance. We present the construction of a near-optimal control for the problem
\begin{equation}\label{concreteportOPT}
\sup_{u\in U^T_0}\mathbb{E}\Big[ \frac{1}{\gamma}(X^u(T))^\gamma\Big]
\end{equation}
where $0 < \gamma < 1$, $X^u$ is the wealth process associated to a path-dependent one-dimensional SDE

$$dS(t) = \alpha(t)S(t)dt + \sigma(t)S(t)dB(t)$$
and the risk-free asset whose dynamics is given by
$$dS_0(t) = rS_0(t)dt$$
for a constant $r>0$. We observe the functional of interest $\xi(f) = \gamma^{-1} (f(T))^\gamma$ for $f\in \mathbf{D}_T^1$ satisfies (\ref{descA1}), i.e., it is $\gamma$-H\"{o}lder continuous on $[0,+\infty)$. The controlled wealth process is given by

\begin{equation}\label{port1}
dX^u(t) = \Big((1-u(t))r + u(t)\alpha(t)\Big)X^u(t) dt + u(t)\sigma(t)X^u(t)dB(t).
\end{equation}
Here, the control $u$ is interpreted as the fraction of the total wealth invested in $S$ and $1-u$ denotes the fraction
of total wealth invested in the risk free asset $S_0$. In order to avoid technicalities, we assume $u$, $\alpha$, $\sigma$ are progressively measurable w.r.t $\mathbb{F}$, (\ref{port1}) admits a strong solution in $\mathbf{B}^1(\mathbb{F})$, (\ref{concreteportOPT}) is finite and the action space equals to $\mathbb{A} = [-1,1]$. In addition, we assume the diffusion component is not degenerated, i.e., $|\sigma(t)| >0$ a.s for every $t
\in [0,T]$. Ito's formula yields

$$X^u(t) = x_0 \exp\Bigg\{\int_0^t \Big[u(s)(\alpha(s)-r) + r\Big]ds -\frac{1}{2}\int_0^t u^2(s)\sigma^2(s)ds + \int_0^t u(s)\sigma(s)dB(s)\Bigg\}$$
for $0\le t\le T$ and $X^u(0) = x_0>0$. Let $\alpha^k$ and $\sigma^k$ be approximation processes associated to $\alpha$ and $\sigma$, given respectively by

$$\sigma^k(t) = \sigma^k(0)\mathds{1}_{\{t=0\}} + \sum_{n=0}^\infty \sigma^k(T^k_n)\mathds{1}_{\{T^k_n < t \le T^k_{n+1}\}}$$
$$\alpha^k(t) = \sum_{n=0}^\infty \alpha^k(T^k_n)\mathds{1}_{\{T^k_n \le t < T^k_{n+1}\}}; 0\le t\le T.$$
We assume that $|\sigma^k(t)|>0$ a.s for every $t
\in [0,T]$ and for every $k
\ge 1$. Let us denote
$$X^{k,u^k}(T^k_n) = x_0\exp\Bigg\{\int_0^{T^k_n} \Big[u^k(s)\Big(r+\alpha^k(s)\Big) + r\Big]ds -\frac{1}{2}\int_0^{T^k_n} |u^k(s)\sigma^k(s)|^2ds +  \int_0^{T^k_n} u^k(s)\sigma^k(s)dA^k(s)\Bigg\}$$
for $n\ge 0$ and we set

$$X^{k,u^k}(t) = \sum_{n=0}^{\infty} X^{k,u^k}(T^k_n)\mathds{1}_{\{T^k_n \le t \wedge T^k_{e(k,T)} < T^k_{n+1}\}}; 0\le t\le T.$$
We assume $u^k\mapsto X^{k,u^k}$ satisfies (\ref{keyassepsilon}) in Theorem \ref{VALUEconv}. In the sequel, we fix $k\ge 1$ and to keep notation simple we set $T=1$. The pathwise description is given by the stepwise constant function

$$\bar{\gamma}^k(\mathbf{o}^k_n,t):=x_0 \exp \Bigg\{\int_0^{t^k_n} \Big[\mathbf{a}^k(s)\Big(r+\alpha^k(s)\Big) + r\Big]ds -\frac{1}{2}\int_0^{t^k_n} |\mathbf{a}^k(s)\sigma^k(s)|^2ds + \int_0^{t^k_n} \mathbf{a}^k(s)\sigma^k(s)dw^k(s)\Bigg\} $$
for $t^k_n \le t < t^k_{n+1}; n\ge 0$, where
$$\mathbf{a}^k(s):=\sum_{\ell=1}^\infty a^k_{\ell-1}\mathds{1}_{\{t^k_{\ell-1} < s\le t^k_\ell\}},\sigma^k(s)=\sum_{\ell=1}^\infty \sigma^k(\mathbf{b}^k_{\ell-1})\mathds{1}_{\{t^k_{\ell-1} < s\le t^k_\ell\}}, \alpha^k(s)=\sum_{\ell=0}^\infty \alpha^k(\mathbf{b}^k_{\ell})\mathds{1}_{\{t^k_{\ell} \le s <  t^k_{\ell+1}\}} $$
and

$$w^k(s) = \sum_{n=0}^\infty \epsilon_k \tilde{i}^k_n \mathds{1}_{\{ t^k_n \le s < t^k_{n+1}\}}.$$
Let us denote $m=e(k,T)$. The first step is to evaluate

$$\sup_{a^k_{m-1}\in \mathbb{A}} \int \mathbb{V}^k_m (\mathbf{o}^k_{m-1},a^k_{m-1},s^k_m,\tilde{i}^k_m)\nu^k_m(ds^k_m,d\tilde{i}^k_m|\mathbf{b}^k_{m-1}),$$
where

$$\mathbb{V}^k_m(\mathbf{o}^k_{m})=\xi\big( \gamma^k_m(\mathbf{o}^k_m)\big),$$
and
$$\gamma^k_m(\mathbf{o}^k_m)(\cdot)= \bar{\gamma}^k(\mathbf{o}^k_\infty)(\cdot\wedge t^k_m).$$
Let us fix $\mathbf{b}^k_{m-1}$. We may assume $s^k_m < T-t^k_{m-1}$. We observe that

$$\int_{\mathbb{S}_k}\mathbb{V}^k_m\big(\mathbf{o}^k_{m-1},a^k_{m-1},s^k_m,\tilde{i}^k_m \big)\nu^k_m(ds^k_m d\tilde{i}^k_m|\mathbf{b}^k_{m-1})$$
$$ = x_0^\gamma\exp\Bigg(\gamma\sum^{m-2}_{\ell=0} a^k_{\ell} \sigma(t^k_\ell)\Delta w_k(t^k_{\ell+1}) + \gamma\sum_{\ell=0}^{m-2} \Big[a^k_\ell\big(\alpha^k(t^k_\ell)-r\big)+r\Big]s^k_{\ell+1} -\frac{\gamma}{2}\sum_{\ell=0}^{m-2}|a^k_\ell\sigma^k(t^k_\ell)|^2s^k_{\ell+1}$$
$$g(a^k_{m-1},\mathbf{b}^k_{m-1})$$
where

$$g(a^k_{m-1},\mathbf{b}^k_{m-1}) = \frac{1}{\gamma}\int G^k(a^k_{m-1},\mathbf{b}^k_{m-1},s^k_m,\tilde{i}^k_m)\nu^k_m(ds^k_m d\tilde{i}^k_m|\mathbf{b}^k_{m-1}),$$
$$G^k(a^k_{m-1},\mathbf{b}^k_{m-1},s^k_m,\tilde{i}^k_m)=\exp\Bigg(\gamma a^k_{m-1} \sigma(t^k_{m-1})\Delta w_k(t^k_{m}) + \gamma \Big[a^k_{m-1}\big(\alpha^k(t^k_{m-1})-r\big)+r\Big]s^k_{m}$$
$$-\frac{\gamma}{2}|a^k_{m-1}\sigma^k(t^k_{m-1})|^2s^k_{m}\Bigg).$$

Therefore, it is sufficient to study the map

$$a\mapsto g(a,\mathbf{b}^k_{m-1})$$
for each $\mathbf{b}^k_{m-1}$. By definition,

$$g(a,\mathbf{b}^k_{m-1}) = \frac{1}{\gamma\epsilon^2_k}\text{cosh}\Big(\gamma\sigma^k(t^k_{m-1})a\epsilon_k\Big)\int_0^{T-t^k_{m-1}}e^{p(a,\mathbf{b}^k_{m-1})x}f_{\tau}(x\epsilon_k^{-2})dx$$
where
$$p(a,\mathbf{b}^k_{\ell})=\gamma \Big[a \big(\alpha^k(t^k_{\ell})-r\big)+r\Big]-\frac{\gamma}{2}|a\sigma^k(t^k_{\ell})|^2; (a,\mathbf{b}^k_\ell)\in \mathbb{A}\times \mathbb{S}^{\ell}_k$$
for $\ell=m-1, \ldots, 0$ and $f_{T^k_1}(x) = f_\tau(\epsilon_k^{-2}x)\epsilon_k^{-2}$, where $f_{T^k_1}$ and $f_\tau$ are the densities of $T^k_1$ and $\tau = \inf\{t>0; |W(t)|=1\}$ for a Brownian motion $W$, respectively. Following (Chapter 5~\cite{milstein}), we have

\begin{equation}\label{s1}
f_\tau(x) = \frac{2}{\sqrt{2\pi x^3}}\sum_{n=0}^\infty(-1)^n (2n+1)\exp\big(\frac{-1}{2x}(2n+1)^2\big).
\end{equation}
Changing the variables, we have

$$
g(a^k_{m-1},\mathbf{b}^k_{m-1}) =\frac{1}{\gamma}\text{cosh}\Big(\gamma\sigma^k(t^k_{m-1})\epsilon_ka^k_{m-1}\Big)\int_0^{(T-t^k_{m-1})\epsilon^{-2}_k}e^{p(a^k_{m-1},\mathbf{b}^k_{m-1})u
\epsilon^2_k}f_{\tau}(u)du.
$$
We observe (see Lemma 3.1 in \cite{milstein}) that (\ref{s1}) can also be written as

\begin{equation}\label{s2}
f_\tau(x) = \frac{\pi}{2}\sum_{n=0}^\infty (-1)^n(2n+1)\exp\big(\frac{-\pi^2x}{8}(2n+1)^2\big).
\end{equation}
The general term of the alternate series in (\ref{s1})

$$\frac{2}{\sqrt{2\pi x^3}}(2n+1)\exp\big(\frac{-1}{2x}(2n+1)^2\big)$$
is decreasing for each $0 < x < \frac{2}{\pi}$. The general term of the alternate series in (\ref{s2})

$$\frac{\pi}{2}(2n+1)\exp\big(\frac{-\pi^2x}{8}(2n+1)^2\big)$$
is decreasing for each $\frac{2}{\pi}< x < \infty$. As a result, if we set

$$f_{\tau,n-1}(x)=\left\{
\begin{array}{rl}
\frac{2}{\sqrt{2\pi x^3}}\sum_{\ell=0}^{n-1}(-1)^{\ell}(2\ell+1)e^{\frac{-1(2\ell+1)^2}{2x}}; & \hbox{if} \ 0<  x< \frac{2}{\pi} \\
\frac{\pi}{2}\sum_{\ell=0}^{n-1} (-1)^{\ell}(2\ell+1)e^{\frac{-\pi^2 x(2\ell+1)^2}{8}};& \hbox{if} \ x > \frac{2}{\pi}.
\end{array}
\right.
$$
We then have,
\begin{equation}\label{estftau}
|f_\tau(x) - f_{\tau,n-1}(x)|\le \left\{
\begin{array}{rl}
\frac{2(2n+1)}{\sqrt{2\pi x^3}}\exp\Big(-\frac{(2n+1)^2}{2x}\Big); & \hbox{if} \ 0 < x < \frac{2}{\pi} \\
\frac{\pi}{2}(2n+1)\exp\Big( -\frac{\pi^2x(2n+1)^2}{8} \Big);& \hbox{if} \ \frac{2}{\pi} < x < \infty. \\
\end{array}
\right.
\end{equation}
Let us define

$$g_n(a,\mathbf{b}^k_{\ell}) =\frac{1}{\gamma\epsilon_k}\text{cosh}\Big(\gamma\sigma^k(t^k_{\ell})\epsilon_ka\Big)\int_0^{(T-t^k_{\ell})}e^{p(a,\mathbf{b}^k_{\ell})x}f_{\tau,n-1}(x\epsilon^{-2}_k)dx$$
for $\ell=m-1, \ldots, 0$.
\begin{proposition}\label{trunc1}
For each $\mathbf{b}^k_{m-1}\in \mathbb{S}^{m-1}_{k}$ and $\epsilon^{-2}_k> \frac{2}{\pi}$, there exists a positive constant $C(\mathbf{b}^k_{m-1})$ which depends on $\alpha^k(t^k_{m-1})$ and $\sigma^k(t^k_{m-1})$ such that

\begin{equation}\label{unifEST}
\sup_{a\in \mathbb{R}}|g_n(a,\mathbf{b}^k_{m-1}) - g(a,\mathbf{b}^k_{m-1})|
\le C(\mathbf{b}^k_{m-1}) \exp\Big(-\frac{(2n+1)^2}{2}\Big); n\ge 1.
\end{equation}
\end{proposition}
\begin{proof}
Let us fix $\mathbf{b}^k_{m-1}$. To keep notation simple, we set $M = \gamma \epsilon_k \sigma^k(t^k_{m-1})$. For a given $M$, let $q_1(M)$ and $q_2(M)$ be the roots of the polynomial $p(a) + Ma$. If $M>0$, then $q_1(0) < q_1(M)$ and $q_2(0) < q_2(M)$ and if $M < 0$, then $q_1(M) < q_1(0)$ and $q_2(M) < q_2(0)$. By symmetry, we may assume $M>0$. Then,

$$\max_{q_1(0)\le a\le q_2(M)}|g_n(a,\mathbf{b}^k_{m-1}) - g(a,\mathbf{b}^k_{m-1})|\le \frac{1}{\gamma}e^{p(\bar{a},\mathbf{b}^k_{m-1})}\max_{q_1(0)\le a\le q_2(M)}\text{cosh}(aM)\int_0^{\epsilon^{-2}_k}|f_\tau(u) - f_{\tau,n-1}(u)|du$$
where $\bar{a} = \frac{\alpha^k(t^k_{m-1})-r}{|\sigma^k(t^k_{m-1})|^2}$. A simple integration together with the exponential bound on the complementary error function in \cite{chiani} yield

$$\int_0^{\epsilon^{-2}_k}|f_\tau(u) - f_{\tau,n-1}(u)|du\le \Big( 2\exp\big(-(2n+1)^2/2\big) + \frac{4}{\pi(2n+1)}\exp\big(-\pi n^2-\pi n -\frac{\pi}{4}\big)\Big)$$
If $a\in R(M)=(-\infty,q_1(0))\cup (q_2(M),+\infty)$, we proceed as follows. At first, we notice

$$|g_n(a,\mathbf{b}^k_{m-1}) - g(a,\mathbf{b}^k_{m-1})|\le \frac{1}{\gamma}\text{cosh}\big(aM\big)\int_0^{\epsilon^{-2}_k}e^{p(a,\mathbf{b}^k_{m-1})u\epsilon^{2}_k}|f_\tau(u) - f_{\tau,n-1}(u)|du.$$
By using the fact that $\epsilon^{-2}_k> 2/\pi$ and (\ref{estftau}), we have
$$\int_0^{\epsilon^{-2}_k}e^{p(a,\mathbf{b}^k_{m-1})u\epsilon^{2}_k}|f_\tau(u) - f_{\tau,n-1}(u)|du \le  I^k_1(a) + I^k_2(a)$$
where

$$I^k_1(a) = \int_0^{2/\pi}\frac{2(2n+1)}{\sqrt{2\pi u^3}}\exp\Big(p(a,\mathbf{b}^k_{m-1})u\epsilon^{2}_k -\frac{(2n+1)^2}{2u}\Big)du$$
$$I^k_2(a) = \int_{2/\pi}^\infty \frac{\pi}{2}(2n+1)\exp\Big( p(a,\mathbf{b}^k_{m-1})u\epsilon^{2}_k-\frac{\pi^2u(2n+1)^2}{8} \Big)du$$
A simple integration yields

$$I^k_1(a) = e^{(2n+1)\sqrt{2\epsilon_k d(a)}}\text{erfc} \Big(\frac{2n+1}{\sqrt{4/\pi}}+\sqrt{\epsilon_k d(a) 2/\pi} \Big)$$
$$+e^{-(2n+1)\sqrt{2\epsilon_k d(a)}}\text{erfc} \Big(\frac{2n+1}{\sqrt{4/\pi}}-\sqrt{\epsilon_k d(a) 2/\pi} \Big)$$
and
$$I^k_2(a) = \frac{4\pi(2n+1)}{2\pi^2(2n+1)^2-8\epsilon_kp(a)}e^{\frac{8\epsilon_kp(a)-\pi^2(2n+1)^2}{4\pi}}$$
where $d(a) = -p(a)$ and erfc denotes the complementary error function. Let us denote $\bar{M} = 2\epsilon_k\pi^{-1}$. Let $q_1(\bar{M}) < q_2(\bar{M})$ be the roots of the polynomial $\bar{M}p(a) + aM$. Since $\bar{M} < 1$, we observe $q_1(M) < q_1(\bar{M}) < q_2(\bar{M}) < q_2(M)$ and hence,

$$\text{cosh}(aM)I^k_2(a)\le \frac{1}{\pi (2n+1)}e^{\frac{-\pi(2n+1)^2}{4}}\exp(\bar{M}p(a) + Ma)$$
$$+ \frac{1}{\pi (2n+1)}e^{\frac{-\pi(2n+1)^2}{4}}\exp(\bar{M}p(a) - Ma)$$
$$\le \frac{2}{\pi (2n+1)}e^{\frac{-\pi(2n+1)^2}{4}} $$
for every $a\in R(M)$.

The exponential inequality on the complementary error function in \cite{chiani} yields


$$2\text{cosh}(aM)I^k_1(a)\le\exp\Big(-(J_+(n,a))^2 + (2n+1)\sqrt{2\epsilon_kd(a)}+ aM\Big)$$
$$+\exp\Big(-(J_+(n,a))^2 + (2n+1)\sqrt{2\epsilon_kd(a)}- aM\Big)$$
$$+\exp\Big(-(J_-(n,a))^2 + (2n+1)\sqrt{2\epsilon_kd(a)}+ aM\Big)$$
$$\exp\Big(-(J_-(n,a))^2 + (2n+1)\sqrt{2\epsilon_kd(a)}- aM\Big)$$
where $J_+(n,a)=\frac{2n+1}{\sqrt{4/\pi}}+\sqrt{\epsilon_k d(a) 2/\pi}$ and $J_-(n,a)=\frac{2n+1}{\sqrt{4/\pi}}-\sqrt{\epsilon_k d(a) 2/\pi} $. By symmetry, it is sufficient to estimate

$$\exp\Big(-(J_+(n,a))^2 + (2n+1)\sqrt{2\epsilon_kd(a)}+ aM\Big)$$
We observe that
$$-(J_+(n,a))^2 + (2n+1)\sqrt{2\epsilon_kd(a)}+ aM = -\frac{\pi}{4}(2n+1)^2 + \bar{M}p(a) + aM.$$
This shows that
$$\sup_{a\in R(M)}2\text{cosh}(aM)I^k_1(a)\le 4 \exp\big(-\frac{\pi}{4}(2n+1)^2\big).$$
Summing up the above estimates, we can find a positive constant $C(\mathbf{b}^k_{m-1})$ such that (\ref{unifEST}) holds.
\end{proof}

Proposition \ref{trunc1} allows us to construct a near optimal control. Let us fix $\mathbf{b}^k_{m-1}$ and a point $\bar{a}^{k,m-1}_n = \bar{a}^{k,m-1}_n(\mathbf{b}^k_{m-1})$ such that

$$\bar{a}^{k,m-1}_n\in \argmax_{a\in \mathbb{A}} g_n(a,\mathbf{b}^k_{m-1}).$$
By Proposition \ref{trunc1}, given $\epsilon>0$, for each $\mathbf{b}^k_{m-1}$, there exists $N(\mathbf{b}^k_{m-1},\epsilon)$ such that

\begin{equation}\label{gop}
g(\bar{a}^{k,m-1}_n,\mathbf{b}^k_{m-1}) + \epsilon \ge \sup_{a\in \mathbb{A}} g(a,\mathbf{b}^k_{m-1})
\end{equation}
for every $n\ge N(\mathbf{b}^k_{m-1},\epsilon)$.




The conclusion is then the following. For a given $\epsilon>0$ and $m=e(k,T)$, we define

$$C^\epsilon_{k,m-1}(\mathbf{o}^k_{m-1}):=\bar{a}^{k,m-1}_{n}(\mathbf{b}^k_{m-1});~n\ge N(\mathbf{b}^k_{m-1},\epsilon)$$
and this function realizes

$$\mathbb{V}^k_{m-1}(\mathbf{o}^k_{m-1})\le \int_{\mathbb{S}_{k}}\mathbb{V}^k_{m}(\mathbf{o}^k_{m-1},C^\epsilon_{k,m-1}(\mathbf{o}^k_{m-1}),s^k_m,\tilde{i}^k_m)
\nu^k_m(ds^k_m,d\tilde{i}^k_m|\mathbf{b}^k_{m-1})+ \epsilon$$
for each $\mathbf{o}^k_{m-1}\in \mathbb{H}^{k,m-1}$. Indeed, let $a^{*,k,m-1}$ be the maximum point of $a\mapsto g(a,\mathbf{b}^k_{m-1})$. Then,

\begin{eqnarray*}
\int_{\mathbb{S}_{k}}\mathbb{V}^k_{m}(\mathbf{o}^k_{m-1},a^{*,k,m-1},s^k_m,\tilde{i}^k_m)
\nu^k_m(ds^k_m,d\tilde{i}^k_m|\mathbf{b}^k_{m-1})&=&\sup_{a
\in \mathbb{A}}\int_{\mathbb{S}_{k}}\mathbb{V}^k_{m}(\mathbf{o}^k_{m-1},a,s^k_m,\tilde{i}^k_m)
\nu^k_m(ds^k_m,d\tilde{i}^k_m|\mathbf{b}^k_{m-1})\\
& &\\
&=&\mathbb{V}^k_{m-1}(\mathbf{o}^k_{m-1})
\end{eqnarray*}
From (\ref{gop}), for $\epsilon>0$, we have



$$\mathbb{V}^k_{m-1}(\mathbf{o}^k_{m-1}) \le \int_{\mathbb{S}_{k}}\mathbb{V}^k_{m}(\mathbf{o}^k_{m-1},a^{k,*,m-1}_n(\mathbf{b}^k_{m-1}),s^k_m,\tilde{i}^k_m)
\nu^k_m(ds^k_m,d\tilde{i}^k_m|\mathbf{b}^k_{m-1}) +
\epsilon$$
for each $\mathbf{o}^k_{m-1}\in \mathbb{H}^{k,m-1}$ and $n\ge N(\mathbf{b}^k_{m-1},\epsilon)$. We observe that Proposition \ref{trunc1} holds true for

$$a\mapsto g(a,\mathbf{b}^k_{\ell})=\frac{1}{\gamma}\text{cosh}~(\epsilon_k \gamma a\sigma^k(t^k_{\ell}))\int_0^{T-t^k_{\ell}}\exp\Big(p(a,\mathbf{b}^k_{\ell})x\Big)f_{T^k_1}(x)dx; \ell=m-1, \ldots, 1$$
so that for a given $\epsilon>0$, we get a sequence of controls defined by

$$\bar{a}^{k,\ell}_n\in \argmax_{a\in \mathbb{A}} g_n(a,\mathbf{b}^k_{\ell}),$$

$$C^\epsilon_{k,\ell}(\mathbf{o}^k_{\ell}):=\bar{a}^{k, \ell}_n(\mathbf{b}^k_{\ell}); \ell=m-1, \ldots,0, n\ge N(\mathbf{b}^k_{\ell},\epsilon)$$
and by construction $C^\epsilon_{k,\ell}$ realizes

$$\mathbb{V}^k_{\ell-1}(\mathbf{o}^k_{\ell-1})\le \int_{\mathbb{S}_{k}}\mathbb{V}^k_{\ell}(\mathbf{o}^k_{\ell-1},C^\epsilon_{k,\ell-1}(\mathbf{o}^k_{\ell-1}),s^k_{\ell},\tilde{i}^k_{\ell})
\nu^k_{\ell}(ds^k_{\ell},d\tilde{i}^k_{\ell}|\mathbf{b}^k_{\ell-1})+ \epsilon;~\mathbf{o}^k_{\ell-1}\in \mathbb{H}^{k,\ell-1}$$
for every $\ell=m, \ldots, 1$. By applying Proposition \ref{epsiloncTH}, Theorem \ref{VALUEconvcontrol} and (\ref{constructionoptimalcontrol}), we have then constructed a near-optimal control for the problem (\ref{concreteportOPT}).


\

\noindent \textbf{Acknowledgments.} The second author would like to thank UMA-ENSTA-ParisTech for the very kind hospitality during the last stage of this project. He also acknowledges the financial support from ENSTA ParisTech. The authors would like to thank Marcelo Fragoso for stimulating discussions.

\end{document}